\documentclass[11,reqno]{amsart}
\usepackage{amsaddr}
\usepackage{amsfonts}
\usepackage{amssymb}
\usepackage{a4wide}
\usepackage{pgf,pgfarrows,pgfautomata,pgfheaps,pgfnodes,pgfshade}
\usepackage{url}
\usepackage{lineno}

\newcommand*\patchAmsMathEnvironmentForLineno[1]{%
  \expandafter\let\csname old#1\expandafter\endcsname\csname #1\endcsname
  \expandafter\let\csname oldend#1\expandafter\endcsname\csname end#1\endcsname
  \renewenvironment{#1}%
     {\linenomath\csname old#1\endcsname}%
     {\csname oldend#1\endcsname\endlinenomath}}%
\newcommand*\patchBothAmsMathEnvironmentsForLineno[1]{%
  \patchAmsMathEnvironmentForLineno{#1}%
  \patchAmsMathEnvironmentForLineno{#1*}}%
\AtBeginDocument{%
\patchBothAmsMathEnvironmentsForLineno{equation}%
\patchBothAmsMathEnvironmentsForLineno{align}%
\patchBothAmsMathEnvironmentsForLineno{flalign}%
\patchBothAmsMathEnvironmentsForLineno{alignat}%
\patchBothAmsMathEnvironmentsForLineno{gather}%
\patchBothAmsMathEnvironmentsForLineno{multline}%
}

\usepackage{hyperref}
\numberwithin{equation}{section}

\newtheorem{defi}{Definition}[section]
\newtheorem{thm}[defi]{Theorem}
\newtheorem{lemm}[defi]{Lemma}
\newtheorem{rem}[defi]{Remark}
\newtheorem{cor}[defi]{Corollary}
\newtheorem{prop}[defi]{Proposition}

\DeclareMathOperator{\essup}{ess \, sup}

\newcommand{\diffns}{\mathrm{d}}

\begin{document}

\title[]
{
Flows for Singular Stochastic Differential Equations with Unbounded Drifts}

\author{Olivier Menoukeu-Pamen}

\address{African Institute for Mathematical Sciences, Ghana, \\ University of Legon, Ghana\\Institute for Financial and Actuarial Mathematics, Department of Mathematical Sciences, University of Liverpool, L69 7ZL, United Kingdom}
\email{menoukeu@liv.ac.uk}
\thanks{The project on which this publication is based has been carried out with funding provided by the Alexander von Humboldt Foundation, under the programme financed by the German Federal Ministry of Education and Research entitled German Research Chair No 01DG15010.}
\thanks{$^*$\textbf{Dedicated to the memory of Salah E. A. Mohammed}}

\author{Salah E. A. Mohammed $^*$}
\address{formally at the Department of Mathematics, Southern Illinois University, Carbondale, Illinois 62901, USA}

 \keywords{Strong solutions of SDE's, irregular drift coefficient, Malliavin calculus, relative $L^2$-compactness, Sobolev flows.}

\subjclass[2010]{ 60H10, 60H15, 60H40.}

\date{\today}

\maketitle

\begin{abstract}
In this paper, we are interested in the following singular stochastic differential equation (SDE)
\begin{equation*}
\diffns X_{t}=b(t,X_{t}) \diffns t +\diffns B_{t},\,\,\,0\leq t\leq T,\,\,\,\text{ }X_{0}=\,x\in \mathbb{R%
}^{d},
\end{equation*}%
where the drift coefficient $b:[0,T]\times \mathbb{R}^{d}\longrightarrow
\mathbb{R}^{d}$ is Borel measurable, possibly unbounded and has spatial linear growth. The driving noise $B_{t}$ \ is a $d-$%
dimensional Brownian motion. The main objective of the paper is to establish the existence and uniqueness of a strong solution and a Sobolev differentiable stochastic flow for the above SDE. Malliavin differentiability of the solution is also obtained (cf.\cite{MMNPZ13, MNP2015}).
 Our results constitute significant extensions to those in \cite{Zvon74, Ver79, KR05, MMNPZ13, MNP2015} 
by allowing the drift $b$ to be unbounded.
We employ methods from white-noise analysis and the Malliavin calculus.
As application, we prove existence of a unique strong Malliavin differentiable solution to the following stochastic delay differential equation
$$ \diffns X (t) = b (X(t-r), X(t,0,(v,\eta)) \diffns t + \diffns B(t), \,t \geq 0 ,\textbf{  }
 (X(0), X_0)= (v, \eta) \in \mathbb{R}^d \times L^2 ([-r,0], \mathbb{R}^d),
$$
with the drift coefficient $b: \mathbb{R}^d \times \mathbb{R}^d \rightarrow \mathbb{R}^d$  is a Borel-measurable function bounded in the first argument and has linear growth in the second argument.
\end{abstract}


\section{Introduction}

 The purpose of this paper is twofold: First, it aims at studying existence, uniqueness and Malliavin regularity of solutions of   stochastic differential equations with singular drift coefficients satisfying {\it linear growth condition}. Second, it studies existence of Sobolev differentiable flows for this class of SDEs. More specifically, we consider the following
SDE

\begin{equation}
\diffns X_{t}=b(t,X_{t}) \diffns t+ \diffns B_{t},\,\,\,0\leq t\leq T,\,\,\,\text{ }X_{0}=\,x\in \mathbb{R%
}^{d},  \label{Itodiffusion}
\end{equation}%
where the drift coefficient $b:[0,T]\times \mathbb{R}^{d}\longrightarrow
\mathbb{R}^{d}$ is a Borel measurable function and $B_{t}$ \ is a $d-$%
dimensional Brownian motion on a probability space $\left( \Omega ,\mathcal{F},P
\right)$.


When the drift coefficient $b$ in \eqref{Itodiffusion} is globally Lipschitz continuous and of linear growth, it is known that the SDE \eqref{Itodiffusion}  admits a unique strong solution. However, the Lipschitz continuity condition 
is not always satisfied in many interesting SDEs used in practice.  Consequently, the study of SDEs with non-Lipschitz (singular) drift coefficients has received a lot of attention in recent year.

Assuming that the drift coefficient $b$ is bounded and measurable, Zvonkin \cite{Zvon74} proved existence of a unique strong solution of \eqref{Itodiffusion} in the one dimensional case.
It should be noted that,
under the above condition, the one-dimensional deterministic ordinary differential equation
$$
\diffns X_{t}=b(t,X_{t})\diffns t,\,\,\,0\leq t\leq T,\,\,\,\text{ }X_{0}=\,x\in \mathbb{R%
}^{d}
$$
may not have a {\it unique} solution even when one exists. Zvonkin's result was generalised to the $d$-dimensional case by Veretenikov \cite{Ver79}.
For other results in this direction, the reader may consult Gy\"{o}ngy, Krylov \cite{GyK96} or Gy\"{o}ngy, Mart\'{\i}nez \cite{GyM01}, Krylov, R\"{o}ckner \cite{KR05} and  Portenko \cite{Por90}. The proofs of the latter results 
are based either on estimates of solutions of parabolic partial differential equations, the Yamada-Watanabe principle, the Skorohod embedding or a technique due to Portenko \cite{Por90}.

Recently Meyer-Brandis and Proske \cite{MBP10} developed a new technique for the construction of strong solutions of SDEs with singular drift coefficients. This method is based on Malliavin calculus and white noise analysis. Assuming that the drift coefficient in \eqref{Itodiffusion} is bounded and measurable and satisfies a certain symmetry condition (see \cite[Definition 3]{MBP10}), they showed the existence of a unique strong solution. 
The argument in \cite{MBP10} is not based on a pathwise uniqueness argument but it rather gives a direct construction of a strong solution of the SDE. Menoukeu-Pamen et al \cite{MMNPZ13} further developed this approach and derived the results obtained in \cite{MBP10} by relaxing the symmetry condition on the drift coefficient. A striking fact when using this technique is that the strong solution of the SDE with a bounded and measurable drift coefficient is Malliavin differentiable.

The first objective of this paper is to study existence and uniqueness of the solution of the SDE \eqref{Itodiffusion}  when the drift coefficient is Borel measurable and satisfies a linear growth condition. This problem was studied by Engelbert and Schmidt \cite{EnSc89} in the one-dimensional time autonomous case. See also Nilssen \cite{Nil92} for the time dependent case. In this article, we extend the results obtained in \cite{EnSc89, Nil92} to the multidimensional case assuming that the drift coefficient is time dependent and has spatial linear growth. This result constitute a notable extension to the unbounded case of existing results.
We develop estimates on the Malliavin derivatives for solutions to the SDE \eqref{Itodiffusion} when the  drift $b$ is smooth and has linear growth. These estimates are obtained on a time interval that depends only on the linear growth rate of the dift $b$ and is independent of the initial condition. A continuation argument together with uniqueness of the solution is then used to show that the estimates still hold for any finite time horizon. See Lemma \ref{lemmainres1} and Proposition \ref{mainEstimate} (cf. \cite[Lemma 3.5]{MMNPZ13}). Using successive integration by parts, we are able to show
that {\it the estimates are independent of the spatial derivative of the drift coefficient}.

The second objective of the paper is to establish spatial Sobolev regularity of the strong solution 
of the SDE
 \begin{align}\label{eqmainf111}
 	X_t^{s,x}=x+\int_s^t b(u,X_u^{s,x})\diffns u +B_t-B_s, \,\, s,t \in \mathbb{R} \text{ and } x  \in \mathbb{R}^d,
 \end{align}
 when the drift coefficient $b : [0, 1] \times \mathbb{R}^d  \rightarrow \mathbb{R}^d$ is a Borel measurable function with spatial linear growth.  More specifically, we obtain a two-parameter stochastic flow for the SDE \eqref{eqmainf111} defined by
 $$
 \mathbb{R} \times \mathbb{R} \times \mathbb{R}^d \ni (s,t,x) \mapsto\phi_{s,t}(x):= X_t^{s,x} \in \mathbb{R}^d
 $$
 which is pathwise Sobolev differentiable in the spatial variable $x$. Moreover, each flow map
 $$
  \mathbb{R}^d \ni x \mapsto\phi_{s,t}(x) \in \mathbb{R}^d
 $$
 is a Sobolev diffeomorphism in the sense that
 $$
 \phi_{s,t}(x)  \text{ and } \phi_{s,t}^{-1}(x) \in L^2(\Omega, W^{1,p}(\mathbb{R}^d,\mathfrak{p}))
 $$
 for all $s,t \in \mathbb{R}$ and all $p \in(0,\infty)$. The symbol $W^{1,p}(\mathbb{R}^d,\mathfrak{p})$ stands for the  weighted Sobolev space with the weight function $\mathfrak{p}$ having exponential of the $2^{nd}$ moment with respect to Lebesgue measure on $\mathbb{R}^d$. The above result is an extension to the case of unbounded drifts of the results in \cite{MNP2015}. Cf. also \cite{Kun90}

As application of our results, we study the following stochastic delay differential equation
\begin{align}
\diffns X (t) = b (X(t-r), X(t,0,(v,\eta)) \diffns t + \diffns B(t), \,t \geq 0 ,\textbf{  }
	(X(0), X_0)= (v, \eta) \in \mathbb{R}^d \times L^2 ([-r,0], \mathbb{R}^d),
 \end{align}
when the drift coefficient $b: \mathbb{R}^d \times \mathbb{R}^d \rightarrow \mathbb{R}^d$  is a Borel-measurable function bounded in the first argument and has linear growth in the second argument. We prove that there exists a unique strong Malliavin differentiable solution to the above delay equation.

 The paper is organized as follows: In Section \ref{mallcalgaus}, we recall some basic definitions and facts on Malliavin calculus, Gaussian white noise theory and stochastic flows. We also give some preliminary results in this section. In Section \ref{main results}, we present the main results of the paper. Section \ref{proofmainres} is devoted to the proof of the main results.

 \section{Preliminary Background}\label{mallcalgaus}

 In this section we briefly review some facts from Gaussian white noise analysis, Malliavin calculus  and stochastic dynamical systems.  These facts will be used in establishing our results in the forthcoming sections. We refer the reader to \cite{HKPS93, Oba94, Kuo96} for
 more information on white noise theory. For the Malliavin calculus the reader may consult \cite{Nua95, Mall78,  DOP08}. For stochastic flow theory, see \cite{Kun90}.

 \subsection{Basic facts of Gaussian white noise theory}

 In this section, we recall the definition of the Hida distribution space; See T. Hida et al. \cite{HKPS93}. This space plays a crucial role in our proof of the construction of a unique strong solution for the singular SDE \eqref{Itodiffusion}.

Let us fix a time horizon $0<T<\infty .$ Consider a positive
 self-adjoint operator $A$ on $L^{2}([0,T])$ with discrete eigenvalues greater that $1$.

 	
Denote by $\mathcal{S}([0,T])$ the standard nuclear countably
Hilbertian subspace of $(L^{2}([0,T]),A)$ constructed in \cite{Oba94}.
Denote by $%
 	\mathcal{S}^{\prime }([0,T])$ its
topological dual. Denote by $\mathcal{B}(\mathcal{S}^{\prime }([0,T]))$ the Borel
$\sigma -$algebra of $\mathcal{S}^{\prime }([0,T])$.
It follows from the Bochner-Minlos
 	theorem that there exists a unique probability measure $\pi $ on $%
 	\mathcal{B}(\mathcal{S}^{\prime }([0,T]))$
such that%
 \begin{equation*}
 		\int_{\mathcal{S}^{\prime }([0,T])}e^{i\left\langle \omega ,\phi
 			\right\rangle }\pi (\diffns \omega )=e^{-\frac{1}{2}\left\Vert \phi \right\Vert
 			_{L^{2}([0,T])}^{2}}
 	\end{equation*}%
 	is valid for all $\phi \in \mathcal{S}([0,T]),$ where $\left\langle \omega
 	,\phi \right\rangle $ is the action of $\omega \in \mathcal{S}^{\shortmid
 	}([0,T])$ on $\phi \in \mathcal{S}([0,T]).$ The $d$\textsl{-dimensional white noise probability measure} $P$ is defined as the product measure
 	\begin{equation}
 	P =\underset{i=1}{\overset{d}\times }\pi  \label{mu}
 	\end{equation}
 	 on the measurable space
 	\begin{equation}
 		\left( \Omega ,\mathcal{F}\right) :=\left( \underset{i=1}{\overset{d}\prod}\mathcal{S}^{\prime }([0,T]),\underset{i=1}{\overset{d}\otimes}\mathcal{B}%
 			(\mathcal{S}^{\prime }([0,T]))\right)  .\label{Meas}
 	\end{equation}

 	For $\omega =(\omega _{1},\ldots,\omega _{d})\in (\mathcal{S}^{\prime
 	}([0,T]))^{d}$ and $\phi =(\phi ^{(1)},\ldots,\phi ^{(d)})\in (\mathcal{S}%
 	([0,T]))^{d}$ let

 	\begin{equation*}
 		\widetilde{e}(\phi ,\omega )=\exp \left( \left\langle \omega ,\phi
 		\right\rangle -\frac{1}{2}\left\| \phi \right\| _{L^{2}([0,T];\mathbb{R%
 			}^{d})}^{2}\right) ,
 	\end{equation*}
 		be the stochastic exponential. Here $\left\langle \omega ,\phi \right\rangle
 	:=\sum_{i=1}^{d}\left\langle \omega _{i},\phi _{i}\right\rangle .$
 	
 	In what follows, denote by $\left( (\mathcal{S}([0,T]))^{d}\right) ^{\widehat{%
 			\otimes }n}$ the $n-$th complete symmetric tensor product of $(\mathcal{S%
 	}([0,T]))^{d}$ with itself. Since $\widetilde{e}(\phi ,\omega )$
 	is holomorphic in $\phi $ around zero, there exist generalized Hermite
 	polynomials $H_{n}(\omega )\in \left( \left( (\mathcal{S}([0,T]))^{d}\right)
 	^{\widehat{\otimes }n}\right) ^{\prime }$ satisfying
 	\begin{equation}
 	\widetilde{e}(\phi ,\omega )=\sum_{n\geq 0}\frac{1}{n!}\left\langle
 	H_{n}(\omega ),\phi ^{\otimes n}\right\rangle   \label{powerexpansion}
 	\end{equation}%
 	for $\phi $ in a neighbourhood of zero in $(\mathcal{S}([0,T]))^{d}.$
 	One can prove that %
 	\begin{equation}
 	\left\{ \left\langle H_{n}(\omega ),\phi ^{(n)}\right\rangle :\phi ^{(n)}\in
 	\left( (\mathcal{S}([0,T]))^{d}\right) ^{\widehat{\otimes }n},\text{ }n\in
 	\mathbb{N}_{0}\right\}
 	\end{equation}%
 	is a total subset of $L^{2}(P ).$ Moreover, for all $n,m\in \mathbb{N}_{0}$, $\phi ^{(n)}\in \left( (\mathcal{S}%
 	([0,T]))^{d}\right) ^{\widehat{\otimes }n},$ $\psi ^{(m)}\in \left( (%
 	\mathcal{S}([0,T]))^{d}\right) ^{\widehat{\otimes }m}$, we have the following orthogonality
 	relation
 	\begin{equation}
 	\int_{\mathcal{S}^{\prime }}\left\langle H_{n}(\omega ),\phi
 	^{(n)}\right\rangle \left\langle H_{m}(\omega ),\psi ^{(m)}\right\rangle P
 	(\diffns \omega )=\delta _{n,m}n!\left( \phi ^{(n)},\psi ^{(n)}\right)
 	_{L^{2}([0,T]^{n};(\mathbb{R}^{d})^{\otimes n})} . \label{ortho}
 	\end{equation}%
 	Let $\widehat{L}^{2}([0,T]^{n};(\mathbb{R}^{d})^{\otimes n})$ be the
 	space of square integrable symmetric functions  $f: [0,t]^n \to (\mathbb{R}^{d})^{\otimes n})$  
with values in $(\mathbb{R}^{d})^{\otimes n}.$ It follows from \eqref{ortho} that we can uniquely extend the action  $\langle H_{n}(\omega ),\phi^{(n)}\rangle$ of $ H_{n} \in \Big(\left( \mathcal{S([}0,T]\mathcal{)}^{d}\right) ^{\widehat{\otimes }n}\Big)^\prime$ on $\phi ^{(n)} \in \left( \mathcal{S([}0,T]\mathcal{)}^{d}\right) ^{\widehat{\otimes }n}$ to the action of $ H_{n} \in \Big(\left( \mathcal{S([}0,T]\mathcal{)}^{d}\right) ^{\widehat{\otimes }n}\Big)^\prime$ on $\phi ^{(n)} \in \widehat{L}^{2}([0,T]^{n};(\mathbb{R}^{d})^{\otimes n})$ for $\omega$ a.e and for all $n$. Denote by $I_n$ this extension, then
 	$I_{n}(\phi ^{(n)})$ can be viewed as $n$-fold iterated It\^{o} integral of $\phi ^{(n)}\in \widehat{L}%
 	^{2}([0,T]^{n};(\mathbb{R}^{d})^{\otimes n})$ with respect to a $d$-%
 	dimensional Wiener process%
 	\begin{equation}
 	B_{t}=\left( B_{t}^{(1)},\ldots,B_{t}^{(d)}\right)  \label{Bm}
 	\end{equation}%
 	on the white noise probability space
 	\begin{equation}
 	\left( \Omega ,\mathcal{F},P \right) \,.
 	\end{equation}%
 	Using \eqref{powerexpansion} and \eqref{ortho}, we deduce the Wiener-It\^o decomposition of  square integrable Brownian functional that is
 	
 	\begin{equation}
 	L^{2}(P )=\bigoplus_{n\geq 0} I_{n}(\widehat{L}^{2}([0,T]^{n};(\mathbb{R}%
 	^{d})^{\otimes n})).  \label{chaosrepr}
 	\end{equation}
 	
 	The construction of the Hida stochastic test function and distribution space is based on the  decomposition \eqref{chaosrepr}. To this end
 	 let
 	\begin{equation}
 	A^{d}:=(A,\ldots,A)\,,  \label{operator}
 	\end{equation}
 	with $A$ being the self-adjoint operator introduced earlier.  It follows from a second quantization argument that the \textit{Hida stochastic test function space} $(\mathcal{S})$ can be defined as the
 	space of all $f=\sum_{n\geq 0}\left\langle H_{n}(\cdot ),\phi
 	^{(n)}\right\rangle \in L^{2}(P )$ such that%
 	\begin{equation}
 	\left\Vert f\right\Vert _{0,p}^{2}:=\sum_{n\geq 0}n!\left\Vert \left(
 	(A^{d})^{\otimes n}\right) ^{p}\phi ^{(n)}\right\Vert _{L^{2}([0,T]^{n};(%
 		\mathbb{R}^{d})^{\otimes n})}^{2}<\infty  \label{Hidaseminorm}
 	\end{equation}%
 	for all $p\geq 0$. Endowed with the seminorm $\left\Vert \cdot \right\Vert _{0,p},$ $%
 	p\geq 0$, the space $(\mathcal{S})$ is a nuclear
 	Fr\'{e}chet algebra with respect to multiplication of functions. 
 	One sees from \eqref{powerexpansion} that%
 	\begin{equation}
 	\widetilde{e}(\phi ,\omega )\in (\mathcal{S})  \label{Doleans}
 	\end{equation}%
 	for all $\phi \in (\mathcal{S}([0,T]))^{d}.$
 	
 	The \textit{Hida stochastic distribution space} $(\mathcal{S})^{\ast }$ is the topological dual of $(\mathcal{S})$.
 	Hence we get the Gel'fand triple%
 	\begin{equation*}
 		(\mathcal{S})\hookrightarrow L^{2}(P )\hookrightarrow (\mathcal{S})^{\ast
 		}.
 	\end{equation*}%
 	By construction, the Hida distribution space $(\mathcal{S})^{\ast }$ contains the  the time derivatives of the $%
 	d$-dimensional Wiener process $B_{t}$ i.e.,
 	\begin{equation}
 	W_{t}^{i}:=\frac{\diffns}{\diffns t}B_{t}^{i}\in (\mathcal{S})^{\ast },\text{ }i=1,\ldots,d.  \label{whitenoise}
 	\end{equation}%
 	
 	We will also need the definition of the {\it $S$-transform}, see \cite{PS91}. Let $S(\Phi )$ be the $S$-transform of $\Phi
 	\in (\mathcal{S})^{\ast }$ then $S(\Phi )$ is defined by the following  dual
 	pairing
 	\begin{equation}
 	S(\Phi )(\phi ):=\left\langle \Phi ,\widetilde{e}(\phi ,\omega )\right\rangle
 	\label{Stransform}
 	\end{equation}%
 	for $\phi \in (\mathcal{S}_{\mathbb{C}}([0,T]))^{d}$ ($\mathcal{S}_{%
 		\mathbb{C}}([0,T])$ denotes the complexification of $\mathcal{S}([0,T])$.) The $S-$transform is an injective map from $(\mathcal{S})^{\ast }$ to $%
 	\mathbb{C}$ i.e., if
 	\begin{equation*}
 		S(\Phi )=S(\Psi )\text{ for }\Phi ,\Psi \in (\mathcal{S})^{\ast },
 	\end{equation*}%
 	then%
 	\begin{equation*}
 		\Phi =\Psi .
 	\end{equation*}%
 	Furthermore,%
 	\begin{equation}
 	S(W_{t}^{i})(\phi )=\phi ^{i}(t),\text{ }i=1,\ldots ,d\text{ }
 	\label{StransfvonW}
 	\end{equation}%
 	for $\phi =(\phi ^{(1)},\ldots,\phi ^{(d)})\in (\mathcal{S}_{\mathbb{C}%
 	}([0,T]))^{d}.$
 	
 	Finally, we recall the concept of the \textsl{%
 		Wick-Grassmann product}. Let $\Phi ,\Psi \in (\mathcal{S})^{\ast }$ be two distribution then the Wick product $\Phi \diamond \Psi$ of $\Phi$ and $\Psi $ is the unique element in
$(\mathcal{S})^{\ast }$
 	such that%
 	\begin{equation}
 	S(\Phi \diamond \Psi )(\phi )=S(\Phi )(\phi )S(\Psi )(\phi )  \label{Sprop}
 	\end{equation}%
 	for all $\phi \in (\mathcal{S}_{\mathbb{C}}([0,T]))^{d}.$ As an example, we
 	have %
 	\begin{equation}
 	\left\langle H_{n}(\omega ),\phi ^{(n)}\right\rangle \diamond \left\langle
 	H_{m}(\omega ),\psi ^{(m)}\right\rangle =\left\langle H_{n+m}(\omega ),\phi
 	^{(n)}\widehat{\otimes }\psi ^{(m)}\right\rangle  \label{Wick}
 	\end{equation}%
 	for $\phi ^{(n)}\in \left( (\mathcal{S}([0,T]))^{d}\right) ^{\widehat{%
 			\otimes }n}$ and $\psi ^{(m)}\in \left( (\mathcal{S}([0,T]))^{d}\right) ^{%
 		\widehat{\otimes }m}.$ Using \eqref{Wick} and \eqref{powerexpansion}, we get
 	\begin{equation}
 	\widetilde{e}(\phi ,\omega )=\exp ^{\diamond }(\left\langle \omega ,\phi
 	\right\rangle )  \label{Wickident}
 	\end{equation}%
 	for $\phi \in (\mathcal{S}([0,T]))^{d}.$ Here the Wick exponential $\exp
 	^{\diamond }(X)$ of an element $X\in (\mathcal{S})^{\ast }$ is defined as
 	\begin{equation}
 	\exp ^{\diamond }(X)=\sum_{n\geq 0}\frac{1}{n!}X^{\diamond n},
 	\label{wickexp}
 	\end{equation}%
 	where $X^{\diamond n}=X\diamond \ldots \diamond X,$ assuming that the sum on the right hand
 	side of (2.18) converges in $(\mathcal{S})^{\ast }$.


 	\subsection{The Malliavin Derivative}\label{somedefmalcal}
 	
 	Here we recall the definitions of the Malliavin derivative within the context of white noise theory. Without loss of generality we assume that $d=1$. Let $F\in L^{2}(P). $ Then using (\ref{chaosrepr}), there exists a unique sequence $\phi ^{(n)}\in \widehat{L}^{2}([0,T]^{n}), \,n \geq 1$ such that
 	\begin{equation}
 	F=\sum_{n\geq 0}\left\langle H_{n}(\cdot ),\phi ^{(n)}\right\rangle .
 	\label{ch}
 	\end{equation}%
 	 Suppose that
 	\begin{equation}
 	\sum_{n\geq 1}nn!\left\Vert \phi ^{(n)}\right\Vert
 	_{L^{2}([0,T]^{n})}^{2}<\infty \,  \label{D1,2}
 	\end{equation}%
 	and denote by $D_{t}F$ the Malliavin derivative of $F$ in the direction of the Brownian motion. Then
 	
 	\begin{equation}
 	D_{t}F:=\sum_{n\geq 1}n\left\langle H_{n-1}(\cdot ),\phi ^{(n)}(\cdot
 	,t)\right\rangle .  \label{Dt}
 	\end{equation}%
 	
 Denote by $\mathbb{D}_{1,2}$ the family of
 	all $F\in L^{2}(P )$ which are Malliavin differentiable.
 	Define the norm $\left\Vert \cdot \right\Vert _{1,2}$ on $\mathbb{D}_{1,2}$ by
 		\begin{equation}
 		\left\Vert F\right\Vert _{1,2}^{2}:=\left\Vert F\right\Vert _{L^{2}(P
 			)}^{2}+\left\Vert D_{\cdot }F\right\Vert _{L^{2}([0,T]\times \Omega ,\lambda
 			\times P )}^{2}.  \label{norm1,2}
 		\end{equation}%
 	Endowed with this norm, $\mathbb{D}_{1,2}$ is a Hilbert space and
 	the following chain of continuous inclusions are satisfied:%
 	\begin{equation}
 	(\mathcal{S})\hookrightarrow \mathbb{D}_{1,2}\hookrightarrow L^{2}(P
 	)\hookrightarrow \mathbb{D}_{-1,2}\hookrightarrow (\mathcal{S})^{\ast },
 	\label{Inclusion}
 	\end{equation}%
 	where $\mathbb{D}_{-1,2}$ is the dual of $\mathbb{D}_{1,2}.$

  \subsection{Some basic facts on stochastic flows}
 In this section we state some basic facts needed to describe the Sobolev differentiable flow generated by
the singular SDE \eqref{Itodiffusion}.

 \begin{defi}
 	The $P$-preserving (ergodic) Wiener shift $\theta(t,\cdot): \Omega \mapsto \Omega$
is defined   by
 	\begin{align}\label{eqmupreswi1}
 	\theta(t,\omega)(s):=\omega(t+s)-\omega(t),\omega\in \Omega, \,t,s \in \mathbb{R}.
 	\end{align}
\end{defi}
 	
 Note that the Brownian motion satisfies the following {\it perfect helix property}:
 	\begin{align}\label{perfhelixwi1}
 	B_{t_1+t_2}(\omega)-B_{t_1}(\omega)=B_{t_2}(\theta(t_1,\omega)).
 	\end{align}
The above perfect helix property expresses, in a pathwise manner, the fact that the Brownian motion $B$ has stationary ergodic increments.

Next, consider the 
SDE
 		\begin{align}\label{eqmainflow1}
 		X_t^{s,x}=x+\int_s^t b(u,X_u)\, \diffns u +B_t-B_s, \,\, s,t \in \mathbb{R} \text{ and } x  \in \mathbb{R}^d,
 		\end{align}
 where the drift coefficient $b:[0,T]\times \mathbb{R}^{d}\longrightarrow
 \mathbb{R}^{d}$ is a Borel measurable function satisfying a linear growth condition. The following definition describes the dynamics of the SDE \eqref{eqmainflow1}:

 \begin{defi}
 A stochastic flow of homeomorphisms for the SDE \eqref{eqmainflow1} is a map
 $$
\begin{array}{lllll}
 \phi:&\mathbb{R}\times \mathbb{R}\times \mathbb{R}^d&\rightarrow &\mathbb{R}^d \times \Omega\\
 &(s,t,x,\omega) &\mapsto &\phi_{s,t}(x,\omega)
 \end{array}
 $$
 with a universal set $\Omega^\ast \in\mathcal{F}$ of full Wiener measure such that, for all $\omega \in \Omega^\ast$, the following statements hold:
 	\begin{enumerate}
 		\item For any $x\in \mathbb{R}^{d}$, the process $\phi_{s,t}(x,\omega), s,t \in \mathbb{R}$, is a global strong solution to the SDE \eqref{eqmainflow1}.
 		
 		\item $\phi_{s,t}(x,\omega)$ is continuous in $(s,t,x)\in \mathbb{R}\times \mathbb{R}\times \mathbb{R}^d$.
 		
 		\item $\phi_{s,t}(\cdot,\omega)=\phi_{u,t}(\cdot,\omega)\circ \phi_{s,u}(\cdot,\omega)$ for all $s,u,t\in \mathbb{R}$.
 		
 		\item $\phi_{t,t}(x,\omega)=x$ for all $x\in \mathbb{R}^d$ and $t\in \mathbb{R}$.
 		
 		\item $\phi_{s,t}(\cdot,\omega):\mathbb{R}^d\rightarrow \mathbb{R}^d$ are homeomorphisms for all $s,t\in \mathbb{R}$.
 	\end{enumerate}
 	\end{defi}

 We next define a class of weighted Sobolev spaces. Let $\mathfrak{p} :\mathbb{R}^d\rightarrow(0,\infty)$ be a Borel measurable function satisfying
 \begin{align}\label{eqweght11}
 \int_{\mathbb{R}^d} e^{|x|^2}\mathfrak{p}(x) \diffns x <\infty.
 \end{align}
 Denote by $L^p(\mathbb{R}^d,\mathfrak{p})$ the Banach space of all Borel measurable functions $u=(u_1,\ldots,u_d):\mathbb{R}^d\rightarrow \mathbb{R}^d$ satisfying
  \begin{align}\label{eqweght12}
  \int_{\mathbb{R}^d} |u(x)|^p\mathfrak{p}(x) \diffns x <\infty
  \end{align}
 and equipped with the norm
   \begin{align}\label{eqweght13}
  \|u\|_{L^p(\mathbb{R}^d,\mathfrak{p})}:= \Big[\int_{\mathbb{R}^d} |u(x)|^p\mathfrak{p}(x) \diffns x \Big]^{\frac{1}{p}}.
   \end{align}
 Moreover, denote by $W^{1,p}(\mathbb{R}^d,\mathfrak{p})$ the space of functions $u\in L^p(\mathbb{R}^d,\mathfrak{p})$ with weak partial derivatives $\nabla_ju\in L^p(\mathbb{R}^d,\mathfrak{p})$ for $j=1,\ldots,d$. For each $u\in W^{1,p}(\mathbb{R}^d,\mathfrak{p})$, define its norm $\|u\|_{1,p,\mathfrak{p}}$ as follows

  \begin{align}\label{eqweght14}
  \|u\|_{1,p,\mathfrak{p}}:= \|u\|_{L^p(\mathbb{R}^d,\mathfrak{p})}+\sum{}_{i,j=1}^d \|\nabla_ju_i\|_{L^p(\mathbb{R}^d,\mathfrak{p})}.
  \end{align}
  Equipped with the norm \eqref{eqweght14}, the space $W^{1,p}(\mathbb{R}^d,\mathfrak{p})$ is a Banach space.
  \begin{defi}\label{defisobdiff}
  	A stochastic flow $\phi_{s,t}(\cdot,\omega)$ of homeomorphisms is said to be Sobolev-differentiable if for all $s, t \in \mathbb{R}$, the maps $\phi_{s,t}(\cdot,\omega) $ and $\phi_{s,t}^{-1}(\cdot,\omega)$ belong to $W^{1,p}(\mathbb{R}^d,\mathfrak{p})$.
  	\end{defi}

  	\begin{defi}\label{defiperfcocy}
  	Let $\theta(t,\cdot):\Omega \rightarrow \Omega$	be a $P$-preserving Wiener shift for each $t \in \mathbb{R}$. A stochastic flow
is a  perfect Sobolev-differentiable cocycle $\phi_{0,t}(\cdot,\theta(t,\cdot))$ if it satisfies the following property
  	\begin{align}\label{eqweght15}
  	\phi_{0,t_1+t_2}(\cdot,\omega)=\phi_{0,t_2}(\cdot,\theta(t_1,\omega))\circ \phi_{0,t_1}(\cdot,\omega)
  	 \end{align}
  	
  	 for all $\omega \in \Omega$ and $t_1,t_2\in \mathbb{R}$.
  	\end{defi}

  	In the next section, we recall some preliminary facts results which will be useful in the proof of our main results.

 \subsection{Preliminary Facts}

 Our method of constructing the strong solution for the singular SDE
employs the following result which is an extension of the result in \cite{LP04}.
 \begin{prop}
 	\label{explicit} Let $T>0$ be sufficiently small. Assume that the drift coefficient $b:[0,T]\times \mathbb{R}%
 	^{d}\mathbb{\longrightarrow R}^{d}$ in \eqref{Itodiffusion} is
  Lipschitz continuous and of linear growth. Then the unique strong solution $X_{t}=(X^1_{t},...,X^d_{t})$ of \eqref{Itodiffusion} has the following representation
 	\begin{equation}
 	\varphi \left( t,X_{t}^{i}(\omega )\right) =E_{\widetilde{P}}\left[
 	\varphi \left( t,\widetilde{B}_{t}^{i}(\widetilde{\omega })\right)
 	\mathcal{E}_{T}^{\diamond }(b)\right]   \label{exprep}
 	\end{equation}%
 	for all $\varphi :[0,T]\times \mathbb{R\longrightarrow R}$ such that $%
 	\varphi \left( t,B_{t}^{i}\right) \in L^{2}(P)$ for all $0\leq t\leq T,$
 	$i=1,\ldots,d,$. The symbol $\mathcal{E}_T^{\diamond }(b)$ stands for%
 	\begin{eqnarray}
 	\mathcal{E}_{T}^{\diamond }(b)(\omega ,\widetilde{\omega })
 	:= &  \exp^{\diamond }\Big( \sum_{j=1}^{d}\int_{0}^{T}\left( W_{s}^{j}(\omega
 	)+b^{j}(s,\widetilde{B}_{s}(\widetilde{\omega }))\right) \diffns \widetilde{B}%
 	_{s}^{j}(\widetilde{\omega })   \notag \\
 	&  -\frac{1}{2}\int_{0}^{T}\left( W_{s}^{j}(\omega )+b^{j}(s,%
 	\widetilde{B}_{s}(\widetilde{\omega }))\right) ^{\diamond 2}\diffns s\Big) .
 	\label{Doleanswick}
 	\end{eqnarray}%
 $\left( \widetilde{\Omega },\widetilde{\mathcal{F}},\widetilde{P}%
 	\right) ,\left( \widetilde{B}_{t}\right) _{t\geq 0}$ is a copy of the
 	quadruple $\left( \Omega ,\mathcal{F},P\right) ,$ $\left( B_{t}\right)
 	_{t\geq 0}$ in \eqref{Stochbasis}. Moreover $E_{\widetilde{P}}$ denotes
 	the Pettis integral of random elements $\Phi :\widetilde{\Omega }$ $%
 	\longrightarrow \left( \mathcal{S}\right) ^{\ast }$ with respect to the
 	measure $\widetilde{P}.$ The Wick product $\diamond $ in the Wick
 	exponential of \eqref{Doleanswick} is taken with respect to $P$ and $%
 	W_{t}^{j}$ is the white noise derivative of $B_{t}^{j}$ in the Hida space $\left(
 	\mathcal{S}\right) ^{\ast }$ 
 	The stochastic
 	integrals $\int_{0}^{T}\phi (t,\widetilde{\omega })\diffns \widetilde{B}_{s}^{j}(%
 	\widetilde{\omega })$ in \eqref{Doleanswick} are defined for predictable
 	integrands $\phi $ with values in the conuclear space $\left( \mathcal{S}%
 	\right) ^{\ast }$, and the second
 	integral  in \eqref{Doleanswick} is in the sense of
 	Pettis.
 \end{prop}
 \begin{proof}
 	It follows as in \cite{LP04} using Ben$\check{e}$s theorem.
 	\end{proof}


 In the following, we denote by $Y_{t}^{i,b}$ the expectation on the right hand side of \eqref{exprep} for $\varphi (t,x)=x$, that is%
 \begin{equation*}
 Y_{t}^{i,b}:=E_{\widetilde{P}}\left[ \widetilde{B}_{t}^{(i)}\mathcal{E}%
 _{T}^{\diamond }(b)\right]
 \end{equation*}%
 for $i=1,\ldots,d.$ We set
 \begin{equation}
 Y_{t}^{b}:=\left( Y_{t}^{1,b},\ldots,Y_{t}^{d,b}\right)\,.   \label{equation}
 \end{equation}%

The next result gives an exponential estimate of the square of solution to the SDE \eqref{Itodiffusion} and plays a vital part in the proof of our main results.

 \begin{lemm} \label{lemmaproexp1}

Let $t_0 \in [0, 1]$ and $Y : \Omega \to R^d$ be a $\mathcal F_{t_0}$-measurable random variable independent of the $P$-augmented filtration generated by the Brownian motion $B_t$. Let $b: [0,1] \times \mathbb{R}^d \rightarrow \mathbb{R}^d$ be a smooth coefficient with compact support satisfying a global linear growth condition; that is
$$
k:= \displaystyle \essup \bigg \{\frac{|b(t,z)|}{1+|z|}: t \in [0,T], z \in \mathbb{R}^d \bigg \} <\infty.
$$
 Denote by $X^{t_0,Y}_{t}$ the unique strong solution \textup{(}if it exists\textup{)} to the SDE \eqref{Itodiffusion} starting at $Y$ and with drift coefficient $b$. Then there exists a positive  number $\delta_0$ independent of $t_0$ and $Y$ \textup{(}but may depend on $k$\textup{)} such that
\begin{align}\label{eqsupproA1}
	E\exp \{\delta_0 \sup_{t_0\leq t\leq 1}|X^{n,t_0,Y}_{t}|^2 \} \leq C_1 E\exp \{C_2\delta_0|Y|^2 \},
\end{align}
where $C_1, C_2$ are positive constants independent of $Y$, but may depend on $k$. In addition, $C_1$ may depend on $\delta_0$. This expectation is finite provided that the right hand side of \eqref{eqsupproA1} is finite.
 \end{lemm}

\begin{proof}

	Start with the almost sure (a.s.) relation
	\begin{equation}
	\begin{array}{ll}
	X^{t_0,Y}_{t} = Y +  \displaystyle \int_{t_0}^{t} b (u,X^{t_0,Y}_{u})  \diffns u + B_t -B_{t_0},  \, t_0 \leq t \leq 1.
	\end{array}
	\end{equation}
	Using the above equation, we get the following a.s. estimates using H\"older's inequality
	\begin{align}
		|X^{t_0,Y}_{t}|^2 &\leq 3|Y|^2 + 3 \Big | \int_{t_0}^{t} b (u,X^{t_0,Y}_{u}) \diffns u \Big |^2 + 3|B_t -B_{t_0}|^2\notag  \\
		& \leq 3|Y|^2 + 3 \Big(\int_{t_0}^{t} k(1+  |X^{t_0,Y}_{u} |) \diffns u \Big)^2 + 3|B_t -B_{t_0}|^2 \notag \\
		& \leq 3|Y|^2 + 3 \Big(\int_{t_0}^{t} k(1+ |X^{t_0,Y}_{u}|) \diffns u  \Big )^2 + 3|    B_t -B_{t_0}|^2 \notag \\
		& \leq 3|Y|^2 + 6 k^2  (t-t_0)^2\int_{t_0}^{t} \{1+ |X^{t_0,Y}_{u}|^2 \} \diffns u
		+ 3|B_t -B_{t_0}|^2 \notag\\
		& \leq 3|Y|^2 + 6 k^2 (t-t_0)^3 + 6 k^2 (t-t_0)^2 \int_{t_0}^{t} |X^{t_0,Y}_{u}|^2 \diffns u +3|B_t -B_{t_0}|^2.
	\end{align}
	Taking the supremum on both sides and multiplying by $\delta_0$, we have
	
	\begin{align}
		\delta_0\sup_{t_0\leq t\leq 1} |X^{t_0,Y}_{t}|^2
		& \leq 3\delta_0|Y|^2 + 6 k^2 \delta_0+3 \delta_0\sup_{t_0 \leq u \leq 1}|B_u -B_{t_0}|^2 + 6 k^2  \int_{t_0}^{1}\delta_0 \sup_{0\leq u\leq s} |X^{t_0,Y}_{u}|^2 \diffns s,   \text{ a.s. }
	\end{align}
 Applying Gronwall's lemma to the last inequality, we get
	\begin{align}
		\delta_0\sup_{t_0\leq t\leq 1} |X^{t_0,Y}_{t}|^2
		\leq \Big [3\delta_0|Y|^2 + 6 k^2 \delta_0 +3 \delta_0\sup_{t_0 \leq u \leq 1}|B_u -B_{t_0}|^2 \Big ] e^{6 k^2 }    ,   \text{ a.s. }
	\end{align}
 Denote $C_2:=3e^{6 k^2 }$. Then from the above inequality, we get
	\begin{align}
		\exp \{\delta_0\sup_{t_0\leq t\leq 1} |X^{t_0,Y}_{t}|^2 \}
		\leq \exp \{2C_2\delta_0k^2\} \cdot \exp \{\delta_0 C_2|Y|^2\}\cdot \exp \{C_2\delta_0\displaystyle\sup_{t_0 \leq u \leq 1}|B_u -B_{t_0}|^2 \}    ,   \text{ a.s. }
	\end{align}
Taking expectations in the above inequality and using the fact that $Y$	and $|B_u -B_{t_0}|, \, u \geq t_0,$ are independent, we obtain
	\begin{align}\label{eqbmsupA1}
		E\exp \{\delta_0\sup_{t_0\leq t\leq 1}|X^{t_0,Y}_{t}|^2\}
		\leq \exp \{2C_2k^2\delta_0\} \cdot E\exp \{C_2\delta_0|Y|^2\}\cdot E\exp \{C_2\delta_0\sup_{t_0 \leq u \leq t}|B_u -B_{t_0}|^2 \}  .
	\end{align}
	The result will follow if we can find  $\delta_0$ independent of $Y$ and $t_0$ such that
	\begin{equation}\label{eqbmsupA2}
	\begin{array}{ll}
	E\exp \{C_2\delta_0\displaystyle\sup_{t_0 \leq u \leq 1 }|B_u -B_{t_0}|^2 \}  < \infty   .
	\end{array}
	\end{equation}
In order prove \eqref{eqbmsupA2}, we use the exponential series expansion of the left hand side followed by Doob's maximal inequality to obtain the following estimates:
	\begin{align*}
		E\exp \{C_2\delta_0\displaystyle\sup_{t_0 \leq u \leq 1}|B_u -B_{t_0}|^2 \}
		\leq & 1 + \sum_{n=1}^\infty  \frac{C_2^n\delta_0^n}{n!}E \bigg [\sup_{t_0\leq u\leq t_0+(1-t_0)}|B_u -B_{t_0}|^{2n}\bigg ]\\
		=& 1 + \sum_{n=1}^\infty  \frac{C_2^n\delta_0^n}{n!}\Big(\frac{2n}{2n-1}\Big)^{2n}\frac{(2n)!}{2^n.n!}(1-t_0)^n\\
		\leq &1 + \sum_{n=1}^\infty  \frac{C_2^n\delta_0^n}{n!}\Big(\frac{2n}{2n-1}\Big)^{2n}\frac{(2n)!}{2^n.n!}.
	\end{align*}
	Apply the ratio test to the above series $\displaystyle \sum_{n=1}^\infty a_n$, where $a_n := \displaystyle \frac{C_2^n \delta_0^n}{n!}\Big(\frac{2n}{2n-1}\Big)^{2n}\frac{(2n)!}{2^n.n!}, \ n \geq 1$ to get
	\begin{equation}
	\begin{array}{ll}
	\displaystyle \lim_{n \to \infty} \frac{a_{n+1}}{a_n}
	&= \displaystyle \lim_{n \to \infty} \frac{C_2^{n+1} \delta_0^{n+1}[2(n+1)]!}{(n+1)!2^{(n+1)}
		(n+1)!} \frac{n!2^n n!}{C_2^n\delta_0^n (2n)!} \bigg (\frac{2(n+1)}{2(n+1)-1} \bigg )^{2(n+1)}\Big(\frac{2n}{2n-1}\Big)^{-2n} \\
	&= \displaystyle \lim_{n \to \infty} \frac{C_2^{n+1}\delta_0^{1+n} [2(n+1)]!}{(n+1)!2^{(n+1)}
		(n+1)!} \frac{n!2^n n!}{C_2^n\delta_0^n (2n)!} \\
	&= \displaystyle \lim_{n \to \infty} \frac{C_2\delta_0}{2} \frac{(2n+2)(2n+1)}{(n+1)^2} \\
	&= 2 C_2 \delta_0.
	\end{array}
	\end{equation}
	By the ration test, it follows that the series $\displaystyle \sum_{n=1}^\infty a_n$
	converges for $\delta_0< \frac{1}{2C_2}$ (e.g. for $\delta_0 := \frac{1}{12C_2}$). With this choice of $\delta_0$, take
	\begin{align} \label{eqC11}
		C_1:=  \exp \{2C_2k^2\delta_0\} \cdot E\exp \{C_2\delta_0\displaystyle\sup_{t_0 \leq u \leq 1}|B_u -B_{t_0}|^2 \}  < \infty.
	\end{align}
	
	Therefore \eqref{eqbmsupA1} gives
	\begin{align}
		E\exp \{\delta_0\sup_{t_0\leq t\leq 1}|X^{n,t_0,Y}_{t}|^2 \}  \leq C_1 Ee^{C_2\delta_0|Y|^2}.
	\end{align}
	Note that $C_1, C_2$ and $\delta_0$ are independent of $Y$ and $t_0$
	(but may depend on $k$). Thus the claim \eqref{eqsupproA1} holds for the above choice of $\delta_0$.
	\end{proof}

\begin{rem}
	Note that for $Y$ deterministic, the above expectation is finite. If $Y$ is for example independent drifted Brownian motion, the above expression is also finite given the choice of $\delta_0$.
	\end{rem}

\section{Statements of the Main Results}\label{main results}

\subsection{Existence and uniqueness of the strong solution}

Let $B_t$ be a $d$-dimensional Brownian motion with respect to the stochastic basis
\begin{equation}
\left( \Omega ,\mathcal{F},P\right) ,\left\{ \mathcal{F}_{t}\right\},
_{0\leq t\leq T}  \label{Stochbasis}
\end{equation}%
with $\left\{ \mathcal{F}_{t}\right\} _{0\leq
	t\leq T}$ the $P-$augmented filtration generated by $B_{t}$. In this section, we will extend the results
in \cite{MBP10} to cover singular drifts with linear growth.  In particular, we establish the existence and uniqueness of 
a strong solution of the singular SDE
\begin{align}\label{eqmain1}
\diffns X_t=b(t,X_t)\diffns t + \diffns B_t, \,\,0\leq t\leq 1,\,\,\, X_0=x  \in \mathbb{R}^d\,,
\end{align}
where the drift coefficient $b:[0,T]\times \mathbb{R}^{d}\longrightarrow
\mathbb{R}^{d}$ is a Borel measurable function which has linear growth; that is
%
\begin{align}\label{linegrothwcond1}
k:= \displaystyle \essup \bigg \{\frac{|b(t,z)|}{1+|z|}: t \in [0,T], z \in \mathbb{R}^d \bigg \} <\infty.
\end{align}
%
 One of the main results of our article is the following:

\begin{thm}\label{thmainres1}
Suppose that the drift coefficient $b: [0,T] \times \mathbb{R}^d \rightarrow \mathbb{R}^d$ in the SDE \eqref{eqmain1} is a Borel-measurable function such that
$k:=\displaystyle \essup \Big \{\frac{|b(t,z)|}{1+|z|}: t \in [0,T], z \in \mathbb{R}^d \Big \} <\infty$. Then there exists a unique global strong solution $X$ to the SDE \eqref{eqmain1} adapted to the filtration $\left\{ \mathcal{F}_{t}\right\}
_{0\leq t\leq T}$. Furthermore,
 the solution $X_t$ is Malliavin differentiable for all $0 \leq t \leq T$.
\end{thm}

\begin{rem}
In the one dimensional autonomous case, existence and uniqueness for the solution of \eqref{eqmain1} was first obtain by Engelbert and Schmidt \cite{EnSc89}. 
Malliavin differentiability of the solution in the one dimensional case was obtained by Nilssen in \cite{Nil92} for \textsl{small time intervals}.
Thus Theorem \ref{thmainres1} extends the above results to the multidimensional case on any time horizon.
	\end{rem}

Theorem \ref{thmainres1} can further be generalized to cover a class of non-degenerate $d-$dimensional It\^{o}-diffusions as follows:

\begin{thm}
	\label{generalsde}Consider the autonomous $\mathbb{R}^{d}-$valued SDE%
	\begin{equation}
	dX_{t}=b(X_{t})\diffns t+\sigma (X_{t})\diffns B_{t},\,\,\text{ }X_{0}=x\in \mathbb{R}^{d},%
	\text{ }\,\,0\leq t\leq T,  \label{SDE}
	\end{equation}%
	where the coefficients $b:\mathbb{R}^{d}\longrightarrow \mathbb{R}^{d}$ and $%
	\sigma :\mathbb{R}^{d}\longrightarrow \mathbb{R}^{d}\times $ $\mathbb{R}^{d}$%
	are Borel measurable. Suppose that there exists a bijection $\Lambda :%
	\mathbb{R}^{d}\longrightarrow \mathbb{R}^{d}$, which is twice continuously
	differentiable and satisfies the following requirements. Let $\Lambda _{x}:\mathbb{R}^{d}\longrightarrow L\left(
	\mathbb{R}^{d},\mathbb{R}^{d}\right) $ and $\Lambda _{xx}:\mathbb{R}%
	^{d}\longrightarrow L\left( \mathbb{R}^{d}\times \mathbb{R}^{d},\mathbb{R}%
	^{d}\right) $ be the corresponding derivatives of $\Lambda $ and assume that%
	\begin{equation*}
		\Lambda _{x}(y)\sigma (y)=id_{\mathbb{R}^{d}}\text{ for }y\text{ a.e.}
	\end{equation*}%
	and %
	\begin{equation*}
		\Lambda ^{-1}\text{ is Lipschitz continuous.}
	\end{equation*}%
	Suppose that the function $b_{\ast }:\mathbb{R}^{d}\longrightarrow \mathbb{R}%
	^{d}$ given by
	\begin{align*}
		b_{\ast }(x)&:=\Lambda _{x}\left( \Lambda ^{-1}\left( x\right) \right)
		\left[ b(\Lambda ^{-1}\left( x\right) )\right] \\
		&\,\,\,+\frac{1}{2}\Lambda _{xx}\left( \Lambda ^{-1}\left( x\right) \right) \left[
		\sum_{i=1}^{d}\sigma (\Lambda ^{-1}\left( x\right) )\left[ e_{i}\right]
		,\sum_{i=1}^{d}\sigma (\Lambda ^{-1}\left( x\right) )\left[ e_{i}\right] %
		\right]
	\end{align*}%
	satisfies the conditions of Theorem \ref{thmainres1}, where $e_{i},$ $%
	i=1,\ldots,d$, is a basis of $\mathbb{R}^{d}.$  Then the SDE \eqref{SDE} has a stochastic
flow $\phi_{s,t} \in L^2 (\Omega, W^{1,p})$ for all $p > 1$.
Furthermore, each map  $\phi_{s,t}: \Omega \to W^{1,p}$ is Malliavin
	differentiable.
\end{thm}

\begin{proof}
	The proof can be directly obtained from It\^o's Lemma. See \cite{MBP10}.
\end{proof}


\subsection{Existence of a Sobolev differentiable stochastic flow}

In this section, we aim at showing existence of a Sobolev differentiable stochastic flow for the following $d$-dimensional SDE
	\begin{align}\label{eqmain1111}
	X_t^{s,x}=x+\int_s^t b(u,X_u)\diffns u +B_t-B_s, \,\, s,t \in \mathbb{R} \text{ and } x  \in \mathbb{R}^d
	\end{align}
where the drift coefficient $b : [0, T] \times \mathbb{R}^d  \rightarrow \mathbb{R}^d$ is a measurable function satisfying $\|\tilde{b}\|_{\infty} < \infty$, where $\tilde{b}(t,z):=\dfrac{b(t,z)}{1+|z|}, t \in [0,T], z \in \mathbb{R}^d $.

By Theorem \ref{thmainres1}, the SDE \eqref{eqmain1111} has a unique strong solution which we will denote by $X_\cdot^{s,x}$. 
The following theorem gives the existence of a Sobolev differentiable flow for the SDE \eqref{eqmain1111}, and is the main result of this section.

	\begin{thm}\label{thmmain1}
		Assume that the drift coefficient $b$ in the SDE \eqref{eqmain1111} is Borel-measurable and has linear growth. Then the SDE \eqref{eqmain1111} has a Sobolev differentiable stochastic flow $\phi_{t,s}:\mathbb{R}^d \rightarrow \mathbb{R}^d,\, s, t \in \mathbb{R}$; that is
\begin{enumerate}
\item$
		\phi_{t,s}(\cdot) \text{ and } \phi_{t,s}^{-1}(\cdot)\in L^2(\Omega, W^{1,p}(\mathbb{R}^d;\mathfrak{p}))
		$
		for all $s,t\in \mathbb{R}$ and $p\in (1,\infty)$;

\item $ \phi_{t,u}(\cdot, \omega) = \phi_{t,s}(\cdot, \omega) \circ \phi_{s,u}(\cdot, \omega)$, for a.a. $\omega \in \Omega$ and $t \leq s \leq u$;
\item$ \phi_{t,t}(\cdot,\omega) = id_{\mathbb{R}^d}$, for a.a. $\omega \in \Omega$ and
all $t \in \mathbb{R}$.

\end{enumerate}
	\end{thm}

\medskip

	We also have the following cocycle property in the autonomous case.
	\begin{cor}\label{corcocyl}
		Consider the autonomous SDE
		\begin{align}\label{eqmain1112}
			X_t^{s,x}=x+\int_s^t b(X_u^{s,x})\diffns u +B_t-B_s, \,\, s,t \in \mathbb{R}, \, x  \in \mathbb{R}^d,
		\end{align}
		where $b:\mathbb{R}^d\rightarrow \mathbb{R}^d$ is Borel measurable and has
linear growth. Then the stochastic flow of the SDE \eqref{eqmain1112} has a version which generates a perfect Sobolev-differentiable cocycle $(\phi_{0,t},\theta(t,\cdot))$ where $\theta(t,\cdot): \Omega \rightarrow \Omega$ is the $P$-preserving Wiener shift. More specifically, the following perfect cocycle property holds for all $\omega \in \Omega$ and all $t_1,t_2 \in \mathbb{R}$
		\begin{align}\label{cocyprop1}
			\phi_{0,t_1+t_2}(\cdot,\omega)= \phi_{0,t_2}(\cdot,\theta(t,\omega))\circ \phi_{0,t_1}(\cdot,\omega).
		\end{align}
	\end{cor}
	
	\begin{rem}
Similar results were proved in \cite{MNP2015}, assuming that the drift coefficient $b$ is measurable and globally bounded. Considering unbounded drift coefficients, Nilssen in \cite{Nil92}  proves similar result in one dimension and only for small time interval. Note however that the technique, based on 
the It\^o's-Tanaka formula used in \cite{Nil92}, cannot be applied here since the local time for multidimensional Brownian is not defined.

		\end{rem}
	\begin{thm}
		\label{generalsdef}Consider the autonomous $d$-dimensional
SDE%
		\begin{equation}
		dX_{t}=b(X_{t})dt+\sigma (X_{t})dB_{t},\,\,\text{ }X_{0}=x\in \mathbb{R}^{d},%
		\text{ }\,\,0\leq t\leq T,  \label{SDEf1}
		\end{equation}%
		where the coefficients $b:\mathbb{R}^{d}\longrightarrow \mathbb{R}^{d}$ and $%
		\sigma :\mathbb{R}^{d}\longrightarrow \mathbb{R}^{d}\times $ $\mathbb{R}^{d} \,$%
		are Borel measurable. Assume that $\sigma(x)$ has an inverse $\sigma^{-1}(x)$ for all $x\in \mathbb{R}^d$. Moreover, suppose that $\sigma^{-1}:\mathbb{R}^d \rightarrow \mathbb{R}^d\times \mathbb{R}^d$ is in $C^1(\mathbb{R}^d)$ and
		$$
		\frac{\partial}{\partial x_k}\sigma^{-1}_{lj}=\frac{\partial}{\partial x_j}\sigma^{-1}_{lk}
		$$
		for all $l,k,j=1,\ldots,d$. Furthermore, suppose that the function $\Lambda:\mathbb{R}^d \rightarrow \mathbb{R}^d$ defined by
		$$
		\Lambda(x):=\int_0^1\sigma^{-1}(tx)\cdot x\diffns t
		$$
		has a Lipschitz continuous inverse $\Lambda^{-1}:\mathbb{R}^d \rightarrow \mathbb{R}^d$. Suppose $\Lambda$ is $C^2$ with derivatives $D\Lambda:\mathbb{R}^d \rightarrow L(\mathbb{R}^d,\mathbb{R}^d)$ and $D^2\Lambda:\mathbb{R}^d \rightarrow L(\mathbb{R}^d\times \mathbb{R}^d,\mathbb{R}^d)$.
		Suppose that the function $b_{\ast }:\mathbb{R}^{d}\longrightarrow \mathbb{R}%
		^{d}$ given by
		\begin{align*}
			b_{\ast }(x)&:=D\Lambda \left( \Lambda ^{-1}\left( x\right) \right)
			\left[ b(\Lambda ^{-1}\left( x\right) )\right] \\
			&\,\,\,+\frac{1}{2}D^2\Lambda\left( \Lambda ^{-1}\left( x\right) \right) \left[
			\sum_{i=1}^{d}\sigma (\Lambda ^{-1}\left( x\right) )\left[ e_{i}\right]
			,\sum_{i=1}^{d}\sigma (\Lambda ^{-1}\left( x\right) )\left[ e_{i}\right] %
			\right]
		\end{align*}%
		satisfies the conditions of Theorem \ref{thmmain1}, where $e_{i},$ $%
		i=1,\ldots,d$, is a basis of $\mathbb{R}^{d}.$ Then there exists a stochastic flow $(s,t,x)\mapsto \phi_{s,t}(x)$ of the SDE \eqref{SDEf1} such that
		$$
		\phi_{s,t}(x)\in L^2(\Omega,W^p(\mathbb{R}^d,\mathfrak{p}))
		$$
		for all $0\leq s\leq t\leq 1$ and for all $p>1$.
	\end{thm}
	
	\begin{proof}
		The proof can be directly obtained from It\^o's Lemma. See \cite{MNP2015}.
	\end{proof}
	
\medskip
{\it Need conditions on $\Lambda$ and $b$ such that $b_{\ast }$ has linear growwth.}

\medskip

	We also have the following proposition which is essential for the proof of Theorem \ref{thmmain1}.
	
	\begin{prop}\label{propmainres1}
		Let $b:\mathbb{R} \times \mathbb{R}^d  \rightarrow \mathbb{R}^d$ be measurable and and has linear growth.  Le $U$ be an open a bounded subset of $\mathbb{R}^d $. Then for each $t\in \mathbb{R}$ and $p>1$ we have
		$$
		X_t^\cdot \in L^2(\Omega;W^{1,p}(U)).
		$$
	\end{prop}

	\section{Proofs of the Main Results}\label{proofmainres}

	\subsection{Proof of existence and uniqueness of 
the strong solution}

	We first prove Theorem \ref{thmainres1} on a small time interval $[0,t_1]$. Since  $t_1$ is necessarily independent of the initial point, the result will then follow by a continuation argument on successive intervals of length $t_1$ (using random non-anticipative initial conditions and Lemma \ref{lemmaproexp1}).

	The proof of Theorem \ref{thmainres1} is performed in two steps. In first step, we  consider a sequence $b_n: [0,t_1] \times \mathbb{R}^d \rightarrow \mathbb{R}^d$, $n\ge 1$ of smooth coefficients with compact support satisfying a global linear growth condition and converging a.e. to $b$. Using the relative compactness criteria (Corollary \ref{compactcrit}), 
 we prove that for each $0 \leq t \leq t_1$ the sequence of corresponding strong solutions $X^n_{t} = Y_t^{b_n}$, $n \geq 1$, of the SDEs
\begin{align}\label{eqmain}
\diffns X^n_t=b_n(t,X^n_t)\diffns t + \diffns B_t, \,\,0\leq t\leq t_1,\,\,\, X^n_0=x  \in \mathbb{R}^d, \, n \geq 1,
\end{align}
is relatively compact in $L^2(P;\mathbb{R}^d)$.
	
In the second step of the argument, we show that for a measurable drift coefficient $b$ satisfying linear growth condition, the solution $Y_t^{b}, \, 0 \leq t \leq t_1$, of \eqref{eqmain1} is a generalized process in the Hida distribution space. Then using the $S$-transform \eqref{Stransform}, we show that for a sequence
$\{b_n\}_{n=1}^\infty$ of a.e. compactly supported smooth coefficients approximating $b$ and satisfying a uniform global linear growth condition, there exists a convergent subsequence of the corresponding strong solutions $X_{n_j,t} = Y_t^{b_{n_j}}$ satisfying
$$
Y_t^{b_{n_j}} \rightarrow Y_t^{b} \textbf{ in } L^2(P,\mathbb{R}^d)
$$
	for $0 \leq t \leq t_1$. We then check that the limiting process $Y_t^{b}$ is a
strong solution of the SDE \eqref{eqmain1}, using a transformation property for $Y_t^{b}$.

The following lemma is an essential part of the first step of our procedure.
	
	\begin{lemm} \label{lemmainres1}
		
In the SDE \eqref{eqmain1}, let $b: [0,t_1] \times \mathbb{R}^d \rightarrow \mathbb{R}^d$ be a smooth function with compact support. Then the corresponding strong solution $X$ of \eqref{eqmain1} satisfies
$$
E \left[ | D_t X_s - D_{t'} X_s |^2 \right] \leq C_d \Big(    \essup_{(u,z) \in [0,t_1] \times \mathbb{R}^d}  \frac{|b(u,z)|}{1+|z|}\Big)  |t -t'|^{\alpha}
		$$
		for $0 \leq t' \leq t \leq t_1$, $\alpha = \alpha(s) > 0$ with $t_1$ small and
		$$
		\sup_{0 \leq t \leq t_1} E \left[ | D_t X_s |^2 \right] \leq C_d
 \Big (\essup_{(u,z) \in [0,t_1] \times \mathbb{R}^d}  \frac{|b(u,z)|}{1+|z|}\Big),
		$$
		where $C_{d,t_1} : [0, \infty) \rightarrow [0, \infty)$ is an increasing, continuous function, $| \cdot |$ a matrix-norm on $\mathbb{R}^{d \times d}$.

		
	\end{lemm}
	
Combination of Lemma \ref{lemmainres1} and Corollary \ref{compactcrit} yields the following result 
	\begin{cor} \label{maincor}
		Let $b_n: [0,t_1] \times \mathbb{R}^d \rightarrow \mathbb{R}^d$, $n\ge 1$, be a sequence of smooth coefficients with compact support and with
 $\sup_{n \geq 1}\essup_{(u,z) \in [0,t_1] \times \mathbb{R}^d}  \frac{|b_n(u,z)|}{1+|z|} < \infty$
Then for each $0 \leq t \leq t_1$ the sequence of corresponding strong solutions $X_{n,t} = Y_t^{b_n}$, $n \geq 1$, is relatively compact in $L^2(P;\mathbb{R}^d)$.
	\end{cor}

	The following crucial estimate which generalises those in \cite{Da07, MMNPZ13} and plays an important role in deriving our results:
	
	\begin{prop} \label{mainEstimate}
		Let $B$ be a $d$-dimensional Brownian motion starting from the origin and $b_1, \dots , b_n$ be compactly supported continuously differentiable functions $b_i : [0,1] \times \mathbb{R}^d \rightarrow \mathbb{R}$ for $i=1,2, \dots n$. Let $\alpha_i\in \{0,1\}^d$ be a multiindex such that $| \alpha_i| = 1$ for $i = 1,2, \dots, n$. Then there exists a universal constant $C$ (independent of $\{ b_i \}$, $n$, and $\{ \alpha_i \}$) such that
		\begin{equation}
		\label{estimate}
		\left| E \left[ \int_{t_0 < t_1 < \dots < t_n < t} \left( \prod_{i=1}^n D^{\alpha_i}b_i(t_i,B(t_i))  \right) \diffns t_1 \dots \diffns t_n \right] \right| \leq \frac{C^n \prod_{i=1}^n \|\tilde{b}_i \|_{\infty} (t-t_0)^{n/2}}{\Gamma(\frac{n}{2} + 1)},
		\end{equation}
		where $\Gamma$ is the Gamma-function and $\tilde{b}(t,z):= \frac{b(t,z)}{1+|z|}, \, (t,z) \in [0,1] \times \mathbb{R}^d$ satisfies $\|\tilde{b}(t,z)\|\leq 1$. Here $D^{\alpha_i}$ denotes the partial derivative with respect to the $j$th space variable, where $j$ is the position of the $1$ in $\alpha_i$.
	\end{prop}

\begin{proof} See Appendix B
	
	\end{proof}
	 We are now ready to give the proof of Lemma \ref{lemmainres1}.
	\begin{proof}[Proof of Lemma \ref{lemmainres1}]
		
	The chain-rule for the Malliavin derivatives (see \cite{Nua95}) yields%
		\begin{equation} \label{MalliavinDerivativeEquation}
		D_tX_s = \mathcal{I}_d + \int_t^s b'(u,X_u)D_t X_u \diffns u,
		\end{equation}
		$P$-a.e., for all $t \leq s \leq t_1$. Here $\mathcal{I}_d$ is the $d \times d$ identity matrix and $b' = \left( \frac{ \partial}{ \partial x_i} b^{(j)}(t,x) \right)_{1 \leq i,j \leq d} $ is the spatial Jacobian derivative of $b$.
Then, we get for $t_1 \geq s \geq t \geq t' $,
		\begin{align}\label{eqmalderpro111}
			D_{t'} X_s - D_t X_s &= \int_{t'}^s b'(u,X_u)D_{t'} X_u \diffns u - \int_t^s b'(u,X_u)D_t X_u  \diffns u \notag
			\\
			&= \int_{t'}^t b'(u,X_u)D_{t'} X_u \diffns u + \int_t^s b'(u,X_u)\left( D_{t'} X_u - D_t X_u \right) \diffns u \notag
			\\
			& = D_{t'} X_t - \mathcal{I}_d + \int_t^s b'(u,X_u)\left( D_{t'} X_u - D_t X_u \right) \diffns u .
		\end{align}
		Iterating \eqref{eqmalderpro111} leads to
		
		\begin{align} \label{third}
			& D_{t'} X_s -  D_t X_s  \nonumber
			\\
			& =  \Big( \mathcal{I}_d + \sum_{n=1}^{\infty} \int_{t < s_1 < \dots < s_n < s} b'(s_1, X_{s_1}) : \dots : b'(s_n, X_{s_n}) \diffns s_1 \dots \diffns s_n \Big) \Big(D_{t'}X_t - \mathcal{I}_d \Big)
		\end{align}
		in $L^2(P)$, 
 where ``$:$" denotes 
matrix multiplication.
		On the other hand we also have that
		\begin{equation} \label{fourth}
		D_{t'}X_t - \mathcal{I}_d = \sum_{n=1}^{\infty} \int_{t' < s_1 < \dots < s_n < t} b'(s_1, X_{s_1}) : \dots : b'(s_n, X_{s_n}) \diffns s_1 \dots \diffns s_n \,.
		\end{equation}
		Let $| \cdot |$ be the maximum norm on $\mathbb{R}^{d \times d}$. Applying Girsanov's theorem, H\"{o}lder's inequality and the Ben$\check{e}$s condition in relation to \eqref{third} and \eqref{fourth}, we get
		
		\begin{align}\label{eqmaldiffprof1111}
			& E \Big[ | D_{t'} X_s - D_t X_s |^2 \Big] \notag \\
			=&	E \Big[ \Big| \Big( \mathcal{I}_d + \sum_{n=1}^{\infty} \int_{t < s_1 < \dots < s_n < s} b'(s_1, B_{s_1}) : \dots : b'(s_n, B_{s_n}) \diffns s_1 \dots \diffns s_n \Big)
			\notag \\
			&  \times \Big( \sum_{n=1}^{\infty} \int_{t' < s_1 < \dots < s_n < t} b'(s_1, B_{s_1}) : \dots : b'(s_n, B_{s_n}) \diffns s_1 \dots \diffns s_n \Big) \Big|^2
			\notag \\
			&  \times\exp\Big( \sum_{j=1}^d \int_0^{t_1} b^{(j)}(u,B_u) \diffns B^{(j)}_u -\sum_{j=1}^d \int_0^{t_1} (b^{(j)}(u,B_u) )^2\diffns u+\frac{1}{2}\sum_{j=1}^d \int_0^{t_1} (b^{(j)}(u,B_u))^2 \diffns u\Big) \Big]
			\notag \\
			\leq & E\Big[\exp\Big(2\sum_{j=1}^d \int_0^{t_1} (b^{(j)}(u,B_u))^2 \diffns u\Big) \Big]^{\frac{1}{4}}\times E\Big[\exp\Big( 2 \sum_{j=1}^d \int_0^{t_1} b^{(j)}(u,B_u) \diffns B^{(j)}_u -2\sum_{j=1}^d \int_0^{t_1} (b^{(j)}(u,B_u) )^2\diffns u \Big)\Big]^{\frac{1}{2}}\notag \\
			&\times \Big\| \mathcal{I}_d + \sum_{n=1}^{\infty} \int_{t < s_1 < \dots < s_n < s} b'(s_1, B_{s_1}) : \dots : b'(s_n, B_{s_n}) \diffns s_1 \dots \diffns s_n \Big \|^2_{L^{16}(P;\mathbb{R}^{d \times d})} \notag
			\\
			& \times \Big\| \sum_{n=1}^{\infty} \int_{t' < s_1 < \dots < s_n < t} b'(s_1, B_{s_1}) : \dots : b'(s_n, B_{s_n}) \diffns s_1 \dots \diffns s_n \Big \|^2_{L^{16}(P;\mathbb{R}^{d \times d})}.
		\end{align}
		The term $E\Big[ \exp\Big(2 \sum_{j=1}^d \int_0^{t_1} b^{(j)}(u,B_u) \diffns B^{(j)}_u -2\sum_{j=1}^d \int_0^{t_1} (b^{(j)}(u,B_u) )^2\diffns u \Big)\Big]$ is equal to one, by  Ben$\check{e}$s theorem applied to the martingale $2\sum_{j=1}^d \int_0^{t_1} b^{(j)}(u,B_u) \diffns B^{(j)}_u.$ In addition, the term $E\Big[\exp\Big(2\sum_{j=1}^d \int_0^{t_1} (b^{(j)}(u,B_u))^2 \diffns u\Big) \Big]$ is also finite for small $t_1$ (see for example proof of Lemma \ref{lemupGir1}). Hence, there is a positive constant $C_{t_1}$ such that for all $t' \leq t \leq s \leq t_1$, we have

		\begin{align} \label{fifth}
			E &\Big[ | D_{t'} X_s - D_t X_s |^2  \Big]  \notag  \\
			\leq &   C_{t_1} \Big\| \mathcal{I}_d + \sum_{n=1}^{\infty} \int_{t < s_1 < \dots < s_n < s} b'(s_1, B_{s_1}) : \dots : b'(s_n, B_{s_n}) \diffns s_1 \dots \diffns s_n \Big \|^2_{L^{16}(P;\mathbb{R}^{d \times d})}
			\notag	\\
			&  \times \Big\| \sum_{n=1}^{\infty} \int_{t' < s_1 < \dots < s_n < t} b'(s_1, B_{s_1}) : \dots : b'(s_n, B_{s_n}) \diffns s_1 \dots \diffns s_n \Big \|^2_{L^{16}(P;\mathbb{R}^{d \times d})}
			\notag	\\
			\leq 	&   C_{t_1} \Big( 1 + \sum_{n=1}^{\infty} \sum_{i,j = 1}^d \sum_{l_1, \dots l_{n-1} = 1}^d \Big\| \int_{t < s_1 < \dots < s_n < s} \frac{\partial}{\partial x_{l_1}} b^{(i)}(s_1, B_{s_1}) \frac{\partial}{\partial x_{l_2}} b^{(l_1)}(s_2,B_{s_2}) \dots
			\notag	\\
			&   \dots \frac{\partial}{\partial x_j} b^{(l_{n-1})}(s_n, B_{s_n})  \diffns s_1 \dots  \diffns s_n \Big\|_{L^{16}(P;\mathbb{R})} \Big)^2
			\notag	\\
			&   \times \Big( \sum_{n=1}^{\infty} \sum_{i,j=1}^d \sum_{l_1, \dots l_{n-1} =1}^d \Big\| \int_{t' < s_1 < \dots < s_n < t} \frac{\partial}{\partial x_{l_1}} b^{(i)}(s_1, B_{s_1}) \frac{\partial}{\partial x_{l_2}} b^{(l_1)}(s_2,B_{s_2})  \dots  \notag \\
			&\dots\frac{\partial}{\partial x_{j}} b^{(l_{n-1})}(s_n, B_{s_n}) \diffns s_1 \dots \diffns s_n \Big\|_{L^{16}(P;\mathbb{R})} \Big)^2 \,.
		\end{align}
		
		Consider the following the expression
		\begin{equation}
		A := \int_{t' < s_1 < \dots < s_n < t} \frac{\partial}{\partial x_{l_1}} b^{(i)}(s_1, B_{s_1}) \frac{\partial}{\partial x_{l_2}} b^{(l_1)}(s_2,B_{s_2})  \dots \frac{\partial}{\partial x_{l_n}} b^{(l_n)}(s_n, B_{s_n}) \diffns s_1 \dots \diffns s_n .
		\end{equation}
		It follows from repeated use of (deterministic) integration by parts that $A^2$ is equal to a sum of at most $2^{2n}$ summands of the form
		\begin{equation} \label{seventh}
		\int_{t' < s_1 < \dots < s_{2n} < t} g_1 (s_1) \dots g_{2n}(s_{2n}) \diffns s_1 \dots \diffns s_{2n} \,,
		\end{equation}
		where $g_l \in \left\{ \frac{\partial}{\partial x_j} b^{(i)}(\cdot, B_{ \cdot }) : 1 \leq i,j \leq d \right\} $, $l = 1, 2 \dots 2n$.
		Since $A^4 = A^2 A^2$, then by a similar argument there are \emph{at most} $2^{8n}$ such summands (of length $4n$).
		Applying the same concept to $A^{16}$, one gets that it can be represented as a sum of at most $2^{128n}$ terms of the form \eqref{seventh}  with length $16n$.
		
		Using this fact and Proposition \ref{mainEstimate} we obtain
		\begin{align}\label{eqmalderprof211}
			&\left\| \int_{t' < s_1 < \dots < s_n < t} \frac{\partial}{\partial x_{l_1}} b^{(i)}(s_1, B_{s_1}) \frac{\partial}{\partial x_{l_2}} b^{(l_1)}(s_2,B_{s_2})  \dots \frac{\partial}{\partial x_{j}} b^{(l_{n-1})}(s_n, B_{s_n}) \diffns s_1 \dots \diffns s_n \right\|_{L^{16}(\mu;\mathbb{R})} \notag\\
			\leq &\left( \frac{ 2^{128n} C^{16n} \| \tilde{b} \|_{\infty}^{16n}|t - t'|^{8n}}{\Gamma(8n + 1)} \right)^{1/16} \notag\\
			\leq &\frac{ 2^{8n} C^n \|\tilde{b} \|_{\infty}^{n} |t - t'|^{n/2}}{(8n!)^{1/16}} .
		\end{align}
		
		Combining \eqref{fifth} and \eqref{eqmalderprof211} we get
		\begin{align*}
			E \left[ | D_t X_s - D_{t'} X_s |^2 \right] &\leq  C_{t_1}  \left( 1 + \sum_{n=1}^{\infty} \frac{d^{n+2} 2^{8n} C^n \| \tilde{b} \|_{\infty}^n |t-s|^{n/2}}{(8n!)^{1/16}} \right)^2
			\\
			& \quad \times \left(\sum_{n=1}^{\infty} \frac{d^{n+2} 2^{8n} C^n \| \tilde{b}\|_{\infty}^n |t-t'|^{(n-1)/2}}{(4n!)^{1/16}} \right)^2 |t - t'|
			\\
			& \leq   C_{d,t_1}(\| \tilde{b} \|_{\infty}) |t - t'|
		\end{align*}
		for $t' \leq t \leq s \leq t_1$ and a function $C_{d,t_1}$ as claimed in the theorem.
		
		Similarly, we get the estimate for $\sup_{0 \leq t \leq s} E [ |D_t X_s |^2 ]$. This completes the proof of Lemma \ref{lemmainres1}.
	\end{proof}

The next step based on White-noise analysis will be used to identify the process $Y^b$ defined by \eqref{equation} as the Malliavin differentiable solution to \eqref{eqmain1}. Assume that $b$ is Borel measurable function satisfying the linear growth condition. We first show that $Y_{t}^{b}$ is a well-defined element in the Hida distribution space $(\mathcal{S})^{\ast }$, $0\leq t\leq t_1$. Then we show that for a.e. approximating sequence of smooth coefficients $b_n$ with compact support and satisfying a uniform linear growth condition, a subsequence of the corresponding strong solutions $X_{n_j,t} = Y_t^{b_{n_j}}$, satisfies  $Y^{b_{n_j}}_t \rightarrow Y^b_t$ in $L^2(P)$ for $0 \leq t \leq 1$. Finally, we identify $Y_{t}^{b}$ as a strong solution to \eqref{eqmain1} by applying a transformation property for $Y_{t}^{b}$.

	\begin{lemm}\label{lemupGir1}
		Assume that the sequence $\{b_n \}_{n=1}^\infty$ satisfies conditions of \textup{Corollary \ref{maincor}}.  Then the expectation $E\Big[\Big\{ 512\exp\Big(\frac{1}{2}\int_0^{t_1}|b_n(u,B_u)|^2\diffns  u\Big)\Big\}\Big]$ is finite for a small time $t_1$ independent of $n$. 
	\end{lemm}
	\begin{proof}
		\begin{align*}
			E\Big[\Big\{ \exp\Big(512\int_0^T|b_n(u,B_u)|^2\diffns  u\Big)\Big\}\Big]\leq &E\Big[ \exp\Big(512\int_0^Tk^2(1+|B_u|^2)\diffns  u\Big)\Big]\\
			\leq&  \exp\Big(512k^2T\Big)\times E\Big[ \exp\Big(512k^2T\max_{0\leq u\leq T}|B_u|^2\Big)\Big].
		\end{align*}
		
		The process $|B_u|^2$ is a non negative submartingale. Using the Doob's martingale inequality and the exponential expansion, one can show as in the proof of Lemma \ref{lemmaproexp1} that for $t_1 \leq \frac{1}{32 \sqrt{2}k^2}$ the expectation is finite.


	\end{proof}

	From now on, we assume that $t_1 \leq \frac{1}{32 \sqrt{2}k^2}$.
	The next lemma gives a condition under which the process $Y_{t}^{b}$
	belongs to the Hida distribution space.
	
	\begin{lemm}
		\label{hidadistr} Suppose that the drift $b:[0,t_1]\times \mathbb{R}^{d}\mathbb{\longrightarrow R}^{d}$
		is Borel measurable 
and satisfies the linear growth condition \eqref{linegrothwcond1}. 
Then the coordinates of the process $Y_{t}^{b}$ defined in \eqref{equation}, by
		\begin{equation}
		Y_{t}^{i,b}:=E_{\widetilde{P}}\left[ \widetilde{B}_{t}^{(i)}\mathcal{E}%
		_{t_1}^{\diamond }(b)\right]  \,,  \label{explobject}
		\end{equation}%
		\bigskip belong to the Hida distribution space.
	\end{lemm}
	
	\begin{proof}
		Since $t_1\leq \frac{1}{32 \sqrt{2}k^2}$, using the same argument as in the proof of Lemma \ref{lemupGir1}, the result follows in a similar way as in \cite[Lemma 11]{MBP10}.
	\end{proof}

	\begin{lemm}
		\label{diffestimate} Let $b_{n}:[0,t_1]\times \mathbb{R}^{d}\mathbb{%
			\longrightarrow R}^{d}$ be a sequence of Borel measurable functions with
 $b_{0}=b$ and satisfying $\sup_{n \geq 1}\essup_{(u,z) \in [0,t_1] \times \mathbb{R}^d}  \frac{|b_n(u,z)|}{1+|z|} < \infty$. Then the following estimate holds
		\begin{equation*}
			\left\vert S(Y_{t}^{i,b_{n}}-Y_{t}^{i,b})(\phi )\right\vert \leq C\cdot
			E[J_{n}]^{\frac{1}{2}}\cdot \exp (34\int_{0}^{t_1}|\phi (s)| ^{2}\diffns s), \quad t \in [0,T]
		\end{equation*}%
		for all $\phi \in (S_{\mathbb{C}}([0,t_1]))^{d}$, $i=1,\ldots,d$, where
		$J_{n}$ is given by
		\begin{align} \label{Jn}
			J_n := & \sum_{j=1}^d \left( 2 \int_0^{t_1} \left( b^{(j)}_n (u,B_u) - b^{(j)}(u,B_u) \right)^2 du \right. \nonumber
			\\
			& + \left. \left( \int_0^{t_1} \left| (b^{(j)}_n(u,B_u))^2 - (b^{(j)}(u,B_u))^2 \right| du \right)^2 \right), \quad n \geq 1.
		\end{align}
		More specifically, if $b_n$ approximates $b$ in the following sense
		\begin{equation} \label{Jnestimate}
		E[J_n] \rightarrow 0 \text{ as } n \rightarrow \infty,
		\end{equation}
then
		$$
		Y^{b_n}_t \rightarrow Y^b_t \textrm{ in } ( \mathcal{S})^* \text{ as } n \rightarrow \infty
		$$
		 for all $0 \leq t \leq t_1$, $i = 1,\dots, d$.
	\end{lemm}
	
	\begin{proof}
		The proof follows in a similar way as in \cite[Lemma 11]{MBP10} and the conclusion follows since $t_1$ is sufficiently small, 
 using the same reasoning as in the proof of Lemma \ref{lemupGir1}.
	\end{proof}

	\begin{lemm}
		\label{squareint} Let $b_{n}:[0,t_1]\times \mathbb{R}^{d}\mathbb{\longrightarrow } \mathbb{R}^{d}$ be a sequence of Borel-measurable, smooth functions with compact support such that $ \underset{n \rightarrow \infty}{\lim}\essup_{(t,z) \in [0,T]  \times \mathbb{R}^{d}} \frac{|b_n(t,z)|}{1+|z|}<\infty$ which approximates a Borel-measurable, coefficient $b:[0,T]\times \mathbb{R}^{d}\mathbb{\longrightarrow } \mathbb{R}^{d}$ \textup{(}satisfying $\essup_{(t,z) \in [0,T]  \times \mathbb{R}^{d}} \frac{|b(t,z)|}{1+|z|}<\infty)$.
Then there exists a subsequence of the corresponding strong solutions $X_{n_j,t} = Y_t^{b_{n_j}}$, $j=1,2,\dots$, such that
		$$
		Y^{b_{n_j}}_t \longrightarrow Y^b_t, \quad 0 \leq t \leq t_1,
		$$
	in $L^2(\mu)$	as $j\rightarrow \infty$. In particular, $Y_{t}^{b}\in L^{2}(P)$ for $0\leq t\leq t_1$.
	\end{lemm}
	
	\begin{proof}
		Corollary \ref{maincor} guarantees existence of a subsequence $Y_t^{b_{n_j}}$, $j=1,2...$, converging in $L^2(P)$. Hence, using the uniform linear growth property   of the $b_n$'s, it follows by the dominated convergence theorem that  $E[J_{n_j}] \rightarrow 0$ in \eqref{Jnestimate}, and hence $Y^{b_{n_j}}_t \rightarrow Y^b_t $ in  $(\mathcal{S})^*$. Therefore, it follows by the uniqueness of the limit that $Y^{b_{n_j}}_t \rightarrow Y^b_t$ in $L^2(P)$ for all $t \in [0,t_1]$.
		\end{proof}


	
	\begin{lemm}
		\label{translemma} Assume that $b:[0,t_1]\times \mathbb{R}^{d}\mathbb{%
			\longrightarrow R}^{d}$ is Borel-measurable and satisfies
 $\essup_{(t,z) \in [0,T]  \times \mathbb{R}^{d}}\frac{|b(t,z)|}{1+|z|}<\infty$. 
 Then%
		\begin{equation}
		\varphi ^{(i)}\left( t,Y_{t}^{b}\right) =E_{\widetilde{P}}\left[ \varphi
		^{(i)}\left( t,\widetilde{B}_{t}\right) \mathcal{E}_{t_1}^{\diamond }(b)\right],\,\text{a.e.}
		\label{transprop}
		\end{equation}%
	 for all $0\leq t\leq t_1,i=1,\ldots,d$ and $\varphi =(\varphi
		^{(1)},\ldots,\varphi ^{(d)})$ such that $\varphi (B_{t})\in L^{2}(P;\mathbb{%
			R}^{d}).$
	\end{lemm}
	
	\begin{proof}
		Assuming that $t_1$ is small, the proof follows in a similar way as in \cite[Lemma 16]{Pro07}.
	\end{proof}

We now turn out to the proof of Theorem \ref{thmainres1}. We first prove the result for a particular $T\leq t_1$
	
	\begin{proof}[Proof of Theorem \protect\ref{thmainres1}]\leavevmode
		
	\textbf{Case 1}: We assume first that $T=\frac{\delta_0}{64\sqrt{2}k^2}\leq t_1$, where $\delta_0$ is given in Lemma \ref{lemmaproexp1}. We will use \eqref{transprop} of Lemma~%
			\ref{translemma} to verify that $Y_{t}^{b}$ is a unique strong solution of
			the SDE \eqref{eqmain1}. Set $\int_{0}^{t}\varphi
			(s,\omega )\diffns B_{s}:=\sum_{j=1}^{d}\int_{0}^{t}\varphi ^{(j)}(s,\omega
			)\diffns B_{s}^{(j)}$ and $x=0.$ Suppose the sequence $\{b_n \}_{n=1}^\infty$ satisfies
			the conditions of Lemma \ref{squareint}.
			We  first show that $Y_{\cdot }^{b}$ has a continuous modification. Since each $Y_{t}^{b_{n}}$ is a strong
			solution of the SDE \eqref{eqmain1} with $b=b_{n}$, we
			obtain, from Girsanov's theorem (using Ben$\check{e}$s condition \cite{Ben71}) and our assumptions, that%
			\begin{align*}
			E_{P}\Big[ \Big( Y_{t}^{i,b_{n}}-Y_{u}^{i,b_{n}}\Big) ^{4}\Big]
			=& E_{\widetilde{P}}\Big[ \Big( \widetilde{B}_{t}^{(i)}-\widetilde{B}%
			_{u}^{(i)}\Big) ^{4}\mathcal{E}\Big( \int_{0}^{T}b_{n}(s,\widetilde{B}%
			_{s})\diffns \widetilde{B}_{s}\Big) \Big] \\
			\leq & C\cdot \left\vert t-u\right\vert^2
			\end{align*}%
			
			for all $0\leq u,t\leq T$, $n\geq 1$, $i=1,\ldots,d.$ Although the sequence $\left\{ \mathcal{E}\left( \int_{0}^{T}b_{n}(s,\widetilde{B}_{s})\diffns \widetilde{B}_{s}\right) \right\}_{n \geq 1}$ may not be bounded in
			$L^2(\tilde{P},\mathbb{R}^d)$ for all $T >0$, we can still get an upper bound when $T= \frac{\delta_0}{64\sqrt{2}k^2}\leq t_1$ (as in \eqref{eqmaldiffprof1111}), using the supermartingale property of the Dol\'eans-Dade exponential and the proof of Lemma \ref{lemupGir1}.
			
			From Lemma \ref{squareint}, we have that%
			\begin{equation*}
			Y_{t}^{b_{n_j}}\longrightarrow Y_{t}^{b}\text{ in }L^{2}(P;\mathbb{R}%
			^{d}), \quad 0 \leq t \leq T ,
			\end{equation*}%
			for a subsequence of $\{Y_{t}^{b_{n}}\}_{n=1}^\infty$ and hence  almost sure convergence holds for a further subsequence, $0\leq t\leq T$. Therefore Fatou's Lemma yields
			\begin{equation}
			E_{P}\Big[ \Big( Y_{t}^{i,b}-Y_{u}^{i,b}\Big) ^{4}\Big] \leq C
			\cdot \left\vert t-u\right\vert^2  \label{Kolmogorov}
			\end{equation}%
			for all $0\leq u,t\leq T$, $i=1,\ldots,d.$ Thus $Y_{t}^{b}$ has a continuous modification from the Kolmogorov's continuity theorem.
			
			Now  $\widetilde{B}_{t}$ is a weak solution of (\ref{eqmain1}) with respect to the measure $dP^{\ast }=\mathcal{E}\left( \int_{0}^{T}\left( b(s,\widetilde{B}_{s})+\phi (s)\right) \diffns%
			\widetilde{B}_{s}\right) \diffns P$ when the drift coefficient $b(s,x)$ is replaced by $b(s,x)+\phi (s)$. Thus we get %
			\begin{align*}
			S(Y_{t}^{i,b})(\phi )&=E_{\widetilde{P}}\Big[ \widetilde{B}_{t}^{(i)}%
			\mathcal{E}\Big( \int_{0}^{T}\Big( b(s,\widetilde{B}_{s})+\phi (s)\Big) \diffns %
			\widetilde{B}_{s}\Big) \Big] \\
			& =E_{P^{\ast }}\Big[ \widetilde{B}_{t}^{(i)}\Big] \\
			& =E_{P^{\ast }}\Big[ \int_{0}^{t}\Big( b^{(i)}(s,\widetilde{B}%
			_{s})+\phi ^{(i)}(s)\Big) \diffns s\Big] \\
			& =\int_{0}^{t}E_{\widetilde{P}}\Big[ b^{(i)}(s,\widetilde{B}_{s})%
			\mathcal{E}\Big( \int_{0}^{T}\Big( b(u,\widetilde{B}_{u})+\phi (u)\Big) \diffns %
			\widetilde{B}_{u}\Big) \Big] \diffns s+S\Big( B_{t}^{(i)}\Big) (\phi )
			\end{align*}%
			for $i=1, 2, \ldots,d$.
			Applying the transformation property \eqref{transprop} to $b$, we get
			\begin{equation*}
			S(Y_{t}^{i,b})(\phi )=S\Big(\int_{0}^{t}b^{(i)}(u,Y_{u}^{i,b})\diffns u\Big)(\phi
			)+S(B_{t}^{(i)})(\phi ), \, \,  i=1, 2, \ldots, d.
			\end{equation*}%
			Since $S$ is injective, we get
			\begin{equation*}
			Y_{t}^{b}=\int_{0}^{t}b(s,Y_{s}^{b})\diffns s+B_{t} \,.
			\end{equation*}%
			The Malliavin differentiability of $Y_{t}^{b}$ follows from the fact that%
			\begin{equation*}
			\sup_{n\geq 1}\left\Vert Y_{t}^{i,b_{n}}\right\Vert _{1,2}\leq M<\infty
			\end{equation*}%
			for all $i=1,\ldots,d$ and $0\leq t\leq T.$ See e.g. \cite{Nua95}.
			
			On the other hand, using Ben$\check{e}$s condition, we can apply Girsanov's theorem to any other strong
			solution. Then the proof of Proposition \ref{explicit} (see e.g. \cite[%
			Proposition 1]{Pro07}) implies that any other solution must be equal to $Y_{t}^{b}$.
			
		The result is given for $T=\frac{\delta_0}{64\sqrt{2} k^2}$ independent of the initial condition. We will use induction and a continuation method with random non-anticipative initial conditions to iterate the above argument on successive intervals of lengths $\frac{\delta_0}{64\sqrt{2} k^2}$ and show that the result is valid for all $0\leq T\leq 1$.

	\textbf{Case 2}: We prove by induction that the unique solution to is Malliavin differentiable. Since there exist a unique strong solution to \eqref{eqmain1} in small time interval, combination of Gronwall Lemma and linear growth condition guarantee non explosion of the unique strong solution. The result will follow if we show that its Malliavin derivative is finite in the $L^2(P)$ norm. We will prove this by induction.
	
		Choose once more $\delta_0$ as in Lemma \ref{lemmaproexp1} and set $ \tau:=\frac{\delta_0}{64\sqrt{2} k^2},\, s_i = i\tau, x_i:= X^{n,0,x}_{s_i}, \frac{\partial}{\partial x_i}X^{n,s_i,x_i}_{s_{i+1}}:= \frac{\partial}{\partial x}X^{n,s_i,x}_{s_{i+1}}\bigg |_{x=x_i}, i \geq 1$. For $0\leq t\leq s_1$, the result is true. Assume that there exists a Malliavin differentiable solution $\{X_t,0\leq t\leq s_{m}\}$. Set $t$ such that $s_{m}\leq t< s_{m+1}$.
		
	Let	 $X^n_{t}$, $n \geq 1$ be the approximating sequence defined by \eqref{eqmain} and let start with the almost sure (a.s.) relation
			\begin{equation*}
			\begin{array}{ll}
			X^{n}_{t} = X_{s_m}^n +  \displaystyle \int_{s_m}^{t} b_n (u,X^n_{u})  \diffns u + B_t -B_{s_m},  \,  s_m\leq t \leq s_{m+1}.
			\end{array}
			\end{equation*}
		
	Using once more the chain-rule for the Malliavin derivatives, it follows that
	\begin{equation}
	D_sX^n_t =\left\{\begin{array}{llll}
	 D_sX_{s_m}^n  + \int_{s_m}^t b_n'(u,X_u)D_s X^n_u \diffns u  & \text{ if } s\leq s_m\\
 \mathcal{I}_d +	\int_{s}^t b_n'(u,X_u)D_s X^n_u \diffns u  &  \text{ if } s>s_m
	\end{array}\right.
	\end{equation}
	$P$-a.e., for all $0 \leq s \leq t$, that is
	\begin{equation} \label{eqthemain1}
D_sX^n_t	=\left\{\begin{array}{llll}
	\Big( \mathcal{I}_d +	\sum_{q=1}^{\infty} \int_{s_m < u_1 < \dots < u_q < t} b_n'(u_1, X^n_{u_1}) : \dots : b_n'(u_q, X^n_{u_q}) \diffns u_1 \dots \diffns u_n  \Big) D_sX_{s_m}^n  & \text{ if } s\leq s_m\\
	\mathcal{I}_d +	\sum_{q=1}^{\infty} \int_{s < u_1 < \dots < u_q < t} b_n'(u_1, X^n_{u_1}) : \dots : b_n'(u_q, X^n_{u_q}) \diffns u_1 \dots \diffns u_q  &  \text{ if } s>s_m
	\end{array}\right.
	\end{equation}
	$P$-a.e., for all $0 \leq s \leq t$. We only consider the first term in \eqref{eqthemain1}.
		\begin{align*}
	&	E\Big[|D_sX^n_t|^2\Big]\notag\\
		=& E\Big[E\Big[\Big|	\Big( \mathcal{I}_d +	\sum_{q=1}^{\infty} \int_{s_m < u_1 < \dots < u_n < t} b_n'(u_1, X^n_{u_1}) : \dots : b_n'(u_q, X^n_{u_q}) \diffns u_1 \dots \diffns u_n  \Big) D_sX_{s_m}^n\Big|^2\Big | \mathcal F_{s_m} \Big]\Big]\notag \\
		=& E\Big[E\Big[\Big|	\Big( \mathcal{I}_d +	\sum_{q=1}^{\infty} \int_{s_m < u_1 < \dots < u_n < t} b_n'(u_1, X^n_{u_1}) : \dots : b_n'(u_q, X^n_{u_q}) \diffns u_1 \dots \diffns u_q  \Big) \Big|^2\Big | \mathcal F_{s_m} \Big] \Big|D_sX_{s_m}^n\Big|^2\Big]\notag\\
		\leq & C_1 E\Big[\Big(1+ \sum_{q=1}^\infty\sum_{i,j=1}^d \sum_{l_1,\ldots,l_{q-1}=1}^d  \Big\|\int_{s_m<u_1<\ldots <u_q<t} \frac{\partial }{\partial x_{l_1}}b_n^{(i)}(u_1,X_{s_m}+B_{u_1})\frac{\partial }{\partial x_{l_2}}b^{(l_1)}_n(u_2,X_{s_m}+B_{u_2}): \ldots \notag\\
		& \ldots : \frac{\partial }{\partial x_j}b_n^{(l_{q-1})}(u_q,X_{s_m}+B_{u_q})\diffns u_1 \ldots \diffns u_q\Big\|_{L^{16}(P,\mathbb{R}^{d},\mathcal{F}_{s_m})}\Big)^2 \times \exp\Big(12k^2\tau(1+|X_{s_m}|^2)\Big)\Big|D_sX_{s_m}^n\Big|^2 \Big].
		\end{align*}
	The last inequality follows from Girsanov transform and H\"older inequality. One can show as in Proposition \ref{flowEstimate} (see also \eqref{linearizedSDE124}) that there exist a positive constant and an increasing and continuous function $C_p : [0, \infty) \rightarrow [0, \infty)$ such that
	\begin{align}
	& \Big(1+ \sum_{q=1}^\infty\sum_{i,j=1}^d \sum_{l_1,\ldots,l_{n-1}=1}^d  \Big\|\int_{s_m<u_1<\ldots <u_q<t} \frac{\partial }{\partial x_{l_1}}b_n^{(i)}(u_1,X_{s_m}+B_{u_1})\frac{\partial }{\partial x_{l_2}}b^{(l_1)}(u_2,X_{s_m}+B_{u_2}): \ldots \notag\\
	& \ldots : \frac{\partial }{\partial x_j}b^{(l_{q-1})}(u_q,X_{s_m}+B_{u_q})\diffns u_1 \ldots \diffns u_n\Big\|_{L^{16}(P,\mathbb{R}^{d},\mathcal{F}_{s_m})}\Big)^2\leq CC_p(|X_{s_m}|,\| \tilde{b} \|_{\infty}).
	\end{align}
	Therefore, using Cauchy-Schwartz inequality, we obtain that there exists $C>0$, which may change from one line to the other and such that
		\begin{align}\label{estimmalder11}
	E\Big[|D_sX^n_t|^2\Big]\leq&  C E\Big[C_p(|X_{s_m}|,\| \tilde{b} \|_{\infty})\exp\Big(12k^2\tau(1+|X_{s_m}|^2)\Big)\Big|D_sX_{s_m}^n\Big|^2 \Big]\notag\\
			\leq &CE\Big[C^2_p(|X_{s_m}|,\| \tilde{b} \|_{\infty})\Big]^{1/4} E\Big[\exp\Big(48k^2\tau(1+|X_{s_m}|^2)\Big)\Big]^{1/4} E\Big[\Big|D_sX_{s_m}^n\Big|^4 \Big]^{1/2}\notag\\
			\leq& C_1 \exp \{C_2\delta_0|x|^2 \} E\Big[\Big|D_sX_{s_m}^n\Big|^4 \Big]\leq C.
			\end{align}
			
			Let us notice that $E[|D_sX_{s_m}^n|^4] < C$. In fact, one can show (see for example Proposition \ref{flowEstimate}) that $E[|D_sX_{s_1}^n|^p] < C$ for all $p\geq 1$. Hence using continuation argument on the successive intervals and Lemma \ref{lemmaproexp1}, the fourth moment of the Malliavin derivative $D_sX_{s_m}^n$ can be controlled.
	\end{proof}
	
	\begin{rem}
		One limitation in our argument is the use of Girsanov transform \textup{(}via Ben$\check{e}$s Theorem\textup{)} in the sense that we cannot cover SDEs with superlinear growth drift coefficients at the moment.
		\end{rem}

	\subsection{Proof of existence of a Sobolev differentiable stochastic flow}

	In order to prove Theorem \ref{thmmain1}, we first need to prove Proposition \ref{propmainres1}. 
	
	We first consider a smooth drift coefficient $b:[0,T]\times \mathbb{R}^d\rightarrow \mathbb{R}^d$  with compact support.
We estimate the norm of $X_t^\cdot \in L^2(\Omega;W^{1,p}(U))$ in terms of $\|\tilde{b}\|_{\infty}$, where $U$ is a bounded open subset of $\mathbb{R}^d$.
	
	Second, we consider a sequence of smooth functions $b_n:[0, T] \times \mathbb{R}^d  \rightarrow \mathbb{R}^d$ with compact support and such that $b_n(t,x)\rightarrow b(t,x)\,\,\diffns t\times \diffns x$-a.e. and there is a positive $k$ with 
$$
\underset{n\geq 1}{\sup} |b_n(t,x)|\leq k(1+|x|)
$$
for all $(t,x) \in [0,T]\times \mathbb{R}^d$.
Let $\{X_\cdot^{n,s,x}\}_{n=1}^\infty$ be the sequence of solutions of \eqref{eqmain1111} when $b$ is replaced by $b_n,\,\,n\geq 1$. We will show that $\{X_\cdot^{n,s,x}\}_{n=1}^\infty$ is relatively compact in $L^2(\Omega)$ when integrated against test functions on $\mathbb{R}^d$.

One can show 
(see \cite{Kun90}) that when $b$ is continuously differentiable with a bounded derivative, the process $X_s^x$ is 
  differentiable in $x$ and its spatial derivative is given by the following linear ODE
\begin{equation} \label{linearizedSDE1}
\frac{\partial}{\partial x} X_s^x = \mathcal{I}_{d} + \int_0^s b'(u,X_u^x) \frac{\partial}{\partial x} X_u^x \diffns u .
\end{equation}
Observe that \eqref{linearizedSDE1} is analogous to \eqref{MalliavinDerivativeEquation} when $t=0$ with the Malliavin derivative of $X_s^x$ replaced by the spatial derivative
$\frac{\partial}{\partial x} X_s^x$. Therefore, similar assertions to  Lemmas \ref{lemm1mainpro}-\ref{cnvLmm} also hold for the spatial derivative $\frac{\partial}{\partial x} X_s^x$.  Using this observation and the proof of Proposition \ref{mainEstimate}, we get the following proposition. 

\begin{prop} \label{mainEstimate1f}
	Let $B$ be a $d$-dimensional Brownian motion starting from the origin and $b_1, \dots , b_n$ be compactly supported continuously differentiable functions $b_i : [0,1] \times \mathbb{R}^d \rightarrow \mathbb{R}$ for $i=1,2, \dots n$. Let $\alpha_i\in \{0,1\}^d$ be a multiindex such that $| \alpha_i| = 1$ for $i = 1,2, \dots, n$. Then there exists a universal constant $C$ \textup{(}independent of $\{ b_i \}_i$, $n$, and $\{ \alpha_i \}_i$\textup{)} such that
	\begin{equation}
	\label{estimate11}
	\Big| E \Big[ \int_{t_0 < t_1 < \dots < t_n < t} \Big( \prod_{i=1}^n D^{\alpha_i}b_i(t_i,x+B(t_i))  \Big) \diffns t_1 \dots \diffns t_n \Big] \Big| \leq \frac{C^n |x|^n\prod_{i=1}^n \|\tilde{b}_i \|_{\infty} (t-t_0)^{n/2}}{\Gamma(\frac{n}{2} + 1)},
	\end{equation}
	where $\Gamma$ is the Gamma-function and $\tilde{b}(t,z):= \frac{b(t,z)}{1+|z|}, (t,z) \in [0,1] \times \mathbb{R}^d$. Here $D^{\alpha_i}$ denotes the partial derivative with respect to the $j'$th space variable, where $j$ is the position of the $1$ in $\alpha_i$.
\end{prop}

The following result gives a key estimate on the spatial weak derivative in terms of  $\| \tilde{b} \|_{\infty}$.

\begin{prop} \label{flowEstimate}
	
	Let $b: [0,1] \times \mathbb{R}^d \rightarrow \mathbb{R}^d$ 				
be a smooth function with compact support. Then for any $p\in[1,\infty)$ and $t\in [0,T]$, the following estimate holds 
\begin{align}\label{eqmaxint1}
E \Big[ \Big| \frac{\partial}{\partial x} X^x_t \Big|^p \Big] \leq C_d (
|x|^2)C_p(|x|,\| \tilde{b} \|_{\infty}),
	\end{align}
	where $C_p,C_d : [0, \infty) \rightarrow [0, \infty)$ are increasing, continuous functions, $| \cdot |$ a matrix-norm on $\mathbb{R}^{d \times d},\,\| \cdot \|_{\infty}$ the supremum norm. 
	
\end{prop}

\begin{proof} 
	We prove this in two steps. We first prove the result for $T\in [0,1] $ sufficiently small. Then we show that there does not exists a maximal interval $[0,T_1]\subset [0,1]$ for which estimate \eqref{eqmaxint1} holds.
	
Let $t\in [0,T]$ for $T$ sufficiently small. Iterating \eqref{linearizedSDE1} gives
\begin{equation} \label{linearizedSDE12}
\frac{\partial}{\partial x} X_t^x = \mathcal{I}_{d} + \sum_{i=1}^\infty\int_{0<s_1<\ldots <s_n<t} b'(s_1,X_{s_1}^x): \ldots :b'(s_n,X_{s_n}^x) \diffns s_1 \ldots \diffns s_n,
\end{equation}
where $ b'$ is the spatial Jacobian matrix of $b$. Applying Ben$\check{e}$s' Theorem to the martingale $2\sum_{j=1}^d \int_0^{t} b^{(j)}(u,B_u) \diffns B^{(j)}_u,$ and using Girsanov's theorem and H\"older's inequality, we have
\begin{align} \label{linearizedSDE123}
E\Big[\Big|\frac{\partial}{\partial x} X_t^x\Big|^p\Big] =& E\Big[\Big|\mathcal{I}_{d} + \sum_{n=1}^\infty\int_{0<s_1<\ldots <s_n<t} b'(s_1,x+B_{s_1}): \ldots :b'(s_n,x+B_{s_n}) \diffns s_1 \ldots \diffns s_n\Big|^p\notag \\
& \times \mathcal{E}\Big(\int_0^Tb(u,x+B_{u})\diffns B_u\Big)\Big] \notag\\
 =& E\Big[\Big|\mathcal{I}_{d} + \sum_{n=1}^\infty\int_{0<s_1<\ldots <s_n<t} b'(s_1,x+B_{s_1}): \ldots :b'(s_n,x+B_{s_n}) \diffns s_1 \ldots \diffns s_n\Big|^p\notag \\
& \times \exp\Big( \sum_{j=1}^d \int_0^{T} b^{(j)}(u,x+B_u) \diffns B^{(j)}_u -\sum_{j=1}^d \int_0^{T} (b^{(j)}(u,x+B_u) )^2\diffns u \notag\\
&+\frac{1}{2}\sum_{j=1}^d \int_0^{T} (b^{(j)}(u,x+B_u))^2 \diffns u\Big) \Big]\notag\\
\leq & C \Big\|\mathcal{I}_{d} + \sum_{n=1}^\infty\int_{0<s_1<\ldots <s_n<t} b'(s_1,x+B_{s_1}): \ldots :b'(s_n,x+B_{s_n}) \diffns s_1 \ldots \diffns s_n\Big\|^p_{L^{4p}(P,\mathbb{R}^{d\times d})} \notag \\
& \times  E\Big[\exp\Big(2\sum_{j=1}^d \int_0^{T} (b^{(j)}(u,x+B_u))^2 \diffns u\Big) \Big]^{\frac{1}{4}}\notag\\
\leq & C \Big\|\mathcal{I}_{d} + \sum_{n=1}^\infty\int_{0<s_1<\ldots <s_n<t} b'(s_1,x+B_{s_1}): \ldots :b'(s_n,x+B_{s_n}) \diffns s_1 \ldots \diffns s_n\Big\|^p_{L^{4p}(\Omega,\mathbb{R}^{d\times d})} \notag \\
& \times \exp\Big(12k^2T(1+|x|^2)\Big)E\Big[ \exp\Big(12k^2T\max_{0\leq u\leq T}|B_u|^2\Big)\Big]\notag\\
\leq & C_1 \Big(1+ \sum_{n=1}^\infty\sum_{i,j=1}^d \sum_{l_1,\ldots,l_{n-1}=1}^d  \Big\|\int_{0<s_1<\ldots <s_n<t} \frac{\partial }{\partial x_{l_1}}b^{(i)}(s_1,x+B_{s_1})\frac{\partial }{\partial x_{l_2}}b^{(l_1)}(s_1,x+B_{s_2}): \ldots \notag\\
& \ldots : \frac{\partial }{\partial x_j}b^{(l_{n-1})}(s_n,x+B_{s_n})\diffns s_1 \ldots \diffns s_n\Big\|_{L^{4p}(P,\mathbb{R}^{d})}\Big)^p \times \exp\Big(12k^2T(1+|x|^2)\Big)\notag\\
\leq & C_1 \Big(1+ \sum_{n=1}^\infty\sum_{i,j=1}^d \sum_{l_1,\ldots,l_{n-1}=1}^d  \Big\|\int_{0<s_1<\ldots <s_n<t} \frac{\partial }{\partial x_{l_1}}b^{(i)}(s_1,x+B_{s_1})\frac{\partial }{\partial x_{l_2}}b^{(l_1)}(s_2,x+B_{s_2}): \ldots \notag\\
& \ldots : \frac{\partial }{\partial x_j}b^{(l_{n-1})}(s_n,x+B_{s_n})\diffns s_1 \ldots \diffns s_n\Big\|_{L^{4^p}(P,\mathbb{R}^{d})}\Big)^p \times \exp\Big(12k^2T(1+|x|^2)\Big),
\end{align}
where the positive constant $C_1$ is obtained in a similar way with the choice of $T$ small as in Lemma \ref{lemupGir1}. The last inequality follows from H\"older's inequality since $4p\leq 4^p$.

	As in the proof of Lemma \ref{lemmainres1}, we consider the  expression
	\begin{equation}
	A := \int_{0 < s_1 < \dots < s_n < t} \frac{\partial}{\partial x_{l_1}} b^{(i)}(s_1, x+B_{s_1}) \frac{\partial}{\partial x_{l_2}} b^{(l_1)}(s_2,x+B_{s_2})  \dots \frac{\partial}{\partial x_{l_n}} b^{(l_n)}(s_n, x+B_{s_n}) \diffns s_1 \dots \diffns s_n .
	\end{equation}
Repeated use of deterministic integration by part shows that $A^2$ can be written as a sum of at most $2^{2n}$ summands of the form
	\begin{equation} \label{seventh1}
	\int_{0 < s_1 < \dots < s_{2n} < t} g_1 (s_1) \dots g_{2n}(s_{2n}) \diffns s_1 \dots \diffns s_{2n} \,,
	\end{equation}
	where $g_l \in \left\{ \frac{\partial}{\partial x_j} b^{(i)}(\cdot, B_{ \cdot }) : 1 \leq i,j \leq d \right\} $, $l = 1, 2 \dots, 2n$. We deduce by induction that $A^{4^p}$ is the sum of at most $4^{p4^pn}$ terms of the form \eqref{seventh1} with length $4^pn$.
	Using this fact and Proposition \ref{mainEstimate1f}, we have that
	\begin{align}\label{eqmalderprof211f}
		&\left\| \int_{t' < s_1 < \dots < s_n < t} \frac{\partial}{\partial x_{l_1}} b^{(i)}(s_1, x+B_{s_1}) \dots 
		\frac{\partial}{\partial x_{j}} b^{(l_{n-1})}(s_n, x+B_{s_n}) \diffns s_1 \dots \diffns s_n \right\|_{L^{4^p}(\mu;\mathbb{R})} \notag\\
		\leq &\Big( \frac{ 4^{p4^pn} C^{4^pn} |x|^{4^pn}\| \tilde{b} \|_{\infty}^{4^pn}|t|^{4^{p}2n}}{\Gamma(4^{p}2n + 1)} \Big)^{4^{-p}}
		\leq \frac{ 4^{pn} C^n |x|^n\|\tilde{b} \|_{\infty}^{n}}{((4^{p}2n)!)^{4^{-p}}} ,
	\end{align}
	from which we get
	\begin{align} \label{linearizedSDE124}
		E\Big[\Big|\frac{\partial}{\partial x} X_t^x\Big|^p\Big] 	\leq & C_1\exp\Big(12k^2T(1+|x|^2)\Big)\Big(1+\sum_{n=1}^{\infty}\frac{ 4^{pn} C^n |x|^n\|\tilde{b} \|_{\infty}^{n}}{((4^{p}2n)!)^{4^{-p}}} \Big)^p\notag\\
		&=C_{d,p}(|x|^2)C_p(|x|,\| \tilde{b} \|_{\infty}).
			\end{align}
			
		Hence for $T$ sufficiently small, the above expectation is finite. Note also that $C_p(|x|,\| \tilde{b} \|_{\infty})$ can be bounded by $ 2^{p-1}(1+\exp \{Cp\| \tilde{b} \|_{\infty}|x|\})$.

		Next, assume that there exists a maximal interval $[0,T_1]$ for which \eqref{eqmaxint1} holds and let show that there exists $1\geq T_2 > T_1$ such that conclusion of \textup{Proposition \ref{flowEstimate}} is valid.
		
		Choose $\delta_0$ as in Lemma \ref{lemmaproexp1} and set	$ \tau:=\frac{\delta_0}{64\sqrt{2} k^2},\, s_i = i\tau, x_i:= X^{n,0,x}_{s_i}, \frac{\partial}{\partial x_i}X^{n,s_i,x_i}_{s_{i+1}}:= \frac{\partial}{\partial x}X^{n,s_i,x}_{s_{i+1}}\bigg |_{x=x_i}, i \geq 1$. There exists $m$ positive integer such that $m\tau\leq T_1< (m+1)\tau$. Let show that the conclusion of Proposition \ref{flowEstimate} is valid for $T_2=(m+1)\tau$.
		
		Consider the times $0 < s_m < t \leq s_{m+1} $. Then by the flow property
		$$ X^{n,0,x}_{t} = X^{n,s_m, \cdot}_{t}(X^{n,0, x}_{s_m})$$
		a.s. By the chain rule
		$$
		\frac{\partial}{\partial x} X^{n,0,x}_{t}=\frac{\partial}{\partial x_m}X^{n,s_m,x_m}_{t} \cdot \frac{\partial}{\partial x} X^{n,0,x}_{s_m}
		$$
		a.s., where $\cdot$ denotes matrix multiplication for the Jacobian derivatives.
		
		Let $\{\mathcal F_t\}_{ t \geq 0}$ denote the filtration generated by the driving Brownian motion $B$. Taking conditional expectation w.r.t to  $.\mathcal F_{s_m}$ in the above equality, using \eqref{linearizedSDE124} and Lemma \ref{lemmaproexp1}, we get
		\begin{align}
		E\Big |\frac{\partial}{\partial x} X^{n,0,x}_{t} \Big |^p = & E \Big ( E\Big [\Big |\frac{\partial}{\partial x_m} X^{n,s_m,x_m}_{t} \cdot \frac{\partial}{\partial x}X^{n,0,x}_{s_m}\Big |^p \Big | \mathcal F_{s_m} \Big] \Big ) \notag\\
		&= E \Big ( E\Big [ \Big|\frac{\partial}{\partial x_m}X^{n,s_m,x_m}_{t} \Big |^p \Big| \mathcal F_{s_m} \Big ] \cdot \Big |\frac{\partial}{\partial x} X^{n,0,x}_{s_m}\Big |^p \Big )  \notag\\
		& \leq C  E\Big ( C_p(|x_m|,\| \tilde{b} \|_{\infty}) e^{6k^2\tau |x_m|^2}  \Big |\frac{\partial}{\partial x} X^{n,0,x}_{s_1}\Big |^p \Big )\notag\\
&\leq  C \Big (E\Big[ C^2_p(|x_m|,\| \tilde{b} \|_{\infty})\Big] \Big )^{1/4} \Big (Ee^{\delta_0|x_m|^2}\Big )^{1/4}  \Big( E\Big [ \Big |\frac{\partial}{\partial x} X^{n,0,x}_{s_m}\Big |^{2p} \Big ] \Big)^{1/2} \notag \\
		&\leq   C \Big (Ee^{\delta_0|x_m|^2}\Big )^{1/4}  \Big( E\Big[ \Big |\frac{\partial}{\partial x} X^{n,0,x}_{s_m}\Big |^{2p} \Big] \Big)^{1/2} \notag \\
		& \leq C \Big (C_1 e^{C_2\delta_0|x|^2 } \Big )^{1/2} \cdot C_d ( |x|^2)C_p(|x|,\| \tilde{b} \|_{\infty}) \notag \\
		& \leq C_{1,\delta_0}C_{d,p}(\|\tilde{b} \|_{\infty},|x|^2,|x|),
		\end{align}
				where $C_1, C_2$ are positive constants independent of $x$, but may depend on $k$.		 
		\end{proof}

This complete the first step of the proof of Theorem \ref{thmmain1}.

We now fix a measurable coefficient  $b:[0,T]\times \mathbb{R}^d\rightarrow \mathbb{R}^d$ satisfying the linear growth condition. By Theorem \ref{thmainres1}, the SDE  \eqref{eqmain1111} admits a (Malliavin differentiable) unique strong solution. Let $X_{\cdot}^{s,x}$ be that solution. Consider the approximating sequence $b_n$ introduced at the beginning of the section and denote by $X_{\cdot}^{n,s,x}$ the corresponding  sequence of solutions of the SDE. We then have the following result

\begin{lemm}\label{lemweakcon1}
	Fix $s,t\in \mathbb{R}$ and $x\in \mathbb{R}^d$. Then the sequence $X_t^{n,s,x}$ converges weakly in $L^2(\Omega;\mathbb{R}^d)$ to $X_t^{s,x}$.
	\end{lemm}

\begin{proof}
	We assume here without loss of generality that $d=1$ and $s=0$. We know that the set
	$$
	\Big\{\mathcal{E}\Big(\int_0^Th(u)\diffns u\Big): h\in C^1_b(\mathbb{R})\Big\}
	$$
	generates a dense subspace of $L^2(\Omega;\mathbb{R}^d)$. Hence it is enough to show that $E\Big[X_t^{n,x}\mathcal{E}\Big(\int_0^Th(u)\diffns u\Big)\Big]\rightarrow E\Big[X_t^{x}\mathcal{E}\Big(\int_0^Th(u)\diffns u\Big)\Big]$. Noting that the function $(t,x)\mapsto b(t,x)+h^\prime (t)$ is of linear growth in $x$, it follows by the Cameron-Martin theorem and the uniqueness in law for the SDE \eqref{eqmain1111} that
	\begin{align*}
		&E\Big[X^{n,x}_t\mathcal{E}
		\Big(\int_0^Th(u) \diffns  B_u\Big)\Big]-E\Big[X^{x}_t\mathcal{E}\Big(\int_0^Th(u)\diffns  B_u\Big)\Big]\\
	=&		E\Big[(x+B_t)\Big\{ \mathcal{E}\Big(\int_0^Tb_n(u,x+B_u)+h^\prime(u)\diffns  B_u\Big)-\mathcal{E}\Big(\int_0^Tb(u,x+B_u)+h^\prime(u) \diffns  B_u\Big) \Big\}\Big]\\
	=& E\Big[(x+B_t)\Big\{ \exp\Big(\int_0^Tb_n(u,x+B_u)+h^\prime(u)\diffns  B_u-\frac{1}{2}\int_0^T(b_n(u,x+B_u)+h^\prime(u))^2\diffns  u\Big)\\
&	-\exp\Big(\int_0^Tb(u,x+B_u)+h^\prime(u)\diffns  B_u-\frac{1}{2}\int_0^T(b(u,x+B_u)+h^\prime(u))^2\diffns  u\Big)\Big\}\Big].
\end{align*}
Using the inequality $|e^a -e^b|\leq |e^a + e^b||a - b|$, we get
\begin{align*}
&E\Big[X^{n,x}_t\mathcal{E}
\Big(\int_0^Th(u) \diffns  B_u\Big)\Big]-E\Big[X^{x}_t\mathcal{E}\Big(\int_0^Th(u)\diffns  B_u\Big)\Big]\\
\leq		& E\Big[(x+B_t)\Big\{ \exp\Big(\int_0^Tb_n(u,x+B_u)+h^\prime(u)\diffns  B_u-\frac{1}{2}\int_0^T(b_n(u,x+B_u)+h^\prime(u))^2\diffns  u\Big)\\
&	+\exp\Big(\int_0^Tb(u,x+B_u)+h^\prime(u)\diffns  B_u-\frac{1}{2}\int_0^T(b(u,x+B_u)+h^\prime(u))^2\diffns  u\Big)\Big\}\times\\
&\Big|\int_0^Tb_n(u,x+B_u)-b(u,x+B_u)\diffns  B_u+\frac{1}{2}\int_0^T\Big((b(u,x+B_u)+h^\prime(u))^2-(b_n(u,x+B_u)+h^\prime(u))^2\Big)\diffns  u\Big|\Big].
\end{align*}
 Using H\"older inequality, we get
\begin{align*}
&E\Big[X^{n,x}_t\mathcal{E}
\Big(\int_0^Th(u) \diffns  B_u\Big)\Big]-E\Big[X^{x}_t\mathcal{E}\Big(\int_0^Th(u)\diffns  B_u\Big)\Big]\\
\leq & E\Big[\Big\{ \exp\Big(\int_0^Tb_n(u,x+B_u)+h^\prime(u)\diffns  B_u-\frac{1}{2}\int_0^T(b_n(u,x+B_u)+h^\prime(u))^2\diffns  u\Big)\\
&+\exp\Big(\int_0^Tb(u,x+B_u)\diffns  B_u-\frac{1}{2}\int_0^T(b(u,x+B_u)+h^\prime(u))^2\diffns  u\Big)\Big\}^{\frac{3}{2}}\Big]^{\frac{2}{3}}\times E\Big[(x+B_t)^{12}\Big]^{\frac{1}{12}}\\
& \times E\Big[\Big\{\int_0^Tb_n(u,x+B_u)-b(u,x+B_u)\diffns  B_u\\
&+\frac{1}{2}\int_0^T\Big((b(u,x+B_u)+h^\prime(u))^2-(b_n(u,x+B_u)+h^\prime(u))^2\Big)\diffns  u \Big\}^4\Big]^{\frac{1}{4}}	.			
\end{align*}
It follows from Cauchy and Burkholder-Davis-Gundy inequality that there exists a constant C such that
				\begin{align*}
					&E\Big[X^{n,x}_t\mathcal{E}
					\Big(\int_0^Th(u) \diffns  B_u\Big)\Big]-E\Big[X^{x}_t\mathcal{E}\Big(\int_0^Th(u)\diffns  B_u\Big)\Big]\\
					\leq	& C E\Big[\Big\{ \exp\Big(\int_0^Tb_n(u,x+B_u)+h^\prime(u)\diffns  B_u-\int_0^T(b_n(u,x+B_u)+h^\prime(u))^2\diffns  u\\
					&
					+\frac{1}{2}\int_0^T(b_n(u,x+B_u)+h^\prime(u))^2\diffns  u\Big)\Big\}^{\frac{3}{2}}\\
					&	+\Big\{\exp\Big(\int_0^Tb(u,x+B_u)\diffns  B_u-\int_0^T(b(u,x+B_u)+h^\prime(u))^2\diffns  u\\
					&+\frac{1}{2}\int_0^T(b(u,x+B_u)+h^\prime(u))^2\diffns  u\Big)\Big\}^{\frac{3}{2}}\Big]^{\frac{2}{3}} \\
					&\times E\Big[\Big\{\int_0^T\Big(b_n(u,x+B_u)-b(u,x+B_u)\Big)^2\diffns u\Big\}^2\\
					 &+\Big\{\int_0^T\Big((b(u,x+B_u)+h^\prime(u))^2-(b_n(u,x+B_u)+h^\prime(u))^2\Big)\diffns u \Big\}^4\Big]^{\frac{1}{4}}\\
					\leq& C \Big(E\Big[ \Big\{\exp\Big(\int_0^Tb_n(u,x+B_u)+h^\prime(u)\diffns  B_u-\int_0^T(b_n(u,x+B_u)+h^\prime(u))^2\diffns  u\Big)\Big\}^2\Big]^{\frac{1}{2}}\\
					&\times E\Big[ \Big\{\exp\Big(\frac{1}{2}\int_0^T(b_n(u,x+B_u)+h^\prime(u))^2\diffns  u\Big)\Big\}^6\Big]^{\frac{1}{6}} \\
					&	+E\Big[ \Big\{\exp\Big(\int_0^Tb(u,x+B_u)+h^\prime(u)\diffns  B_u-\int_0^T(b(u,x+B_u)+h^\prime(u))^2\diffns  u\Big)\Big\}^2\Big]^{\frac{1}{2}}\\
					& \times E\Big[\Big\{ \exp\Big(\frac{1}{2}\int_0^T(b(u,x+B_u)+h^\prime(u))^2\diffns  u\Big)\Big\}^6\Big]^{\frac{1}{6}}\Big)\\
					& \times E\Big[\Big\{\int_0^T\Big(b_n(u,x+B_u)-b(u,x+B_u)\Big)^2\diffns u\Big\}^2\\
					 &+\Big\{\int_0^T\Big((b(u,x+B_u)+h^\prime(u))^2-(b_n(u,x+B_u)+h^\prime(u))^2\Big)\diffns u \Big\}^4\Big]^{\frac{1}{4}}.
				\end{align*}
				The terms $E\Big[ \Big\{\exp\Big(\int_0^Tb_n(u,x+B_u)+h^\prime(u)\diffns  B_u-\int_0^T(b_n(u,x+B_u)+h^\prime(u))^2\diffns  u\Big)\Big\}^2\Big]$ and $E\Big[ \Big\{\exp\Big(\int_0^Tb(u,x+B_u)+h^\prime(u)\diffns  B_u-\int_0^T(b(u,x+B_u)+h^\prime(u))^2\diffns  u\Big)\Big\}^2\Big]$ are finite and equal to one by the Ben$\check{e}$s condition applied to the stochastic integrals with drift $2b$ or $2b_n$.
		
				Moreover, the terms $E\Big[\Big\{ \exp\Big(\frac{1}{2}\int_0^T(b(u,x+B_u)+h^\prime(u))^2\diffns  u\Big)\Big\}^6\Big]$ and \\ $E\Big[\Big\{ \exp\Big(\frac{1}{2}\int_0^T(b_n(u,x+B_u)+h^\prime(u))^2\diffns  u\Big)\Big\}^6\Big]$ are finite for small time $T$. In fact
				\begin{align*}
				&	E\Big[\Big\{ \exp\Big(\frac{1}{2}\int_0^T(b(u,x+B_u)+h^\prime(u))^2\diffns  u\Big)\Big\}^6\Big]\\
					\leq &E\Big[ \exp\Big(6\int_0^T(k(1+|x+B_u|)^2+h^\prime(u)^2)\diffns  u\Big)\Big]\\
					\leq & C\exp(12k^2T(1+|x|^2))E\Big[ \exp\Big(12k^2T\max_{0\leq u\leq T}|B_u|^2\Big)\Big].
					\end{align*}
					This expectation is finite for $T$ sufficiently small and independent on the initial condition of the solution. We can once more use conditioning, induction an a continuation argument with random non-anticipative initial conditions to iterate the above argument on successive intervals of lengths $\tau$. Hence the above result holds for {\it all} $T > 0$.
\end{proof}

As a consequence of the compactness criteria, we have by combining Corollary \ref{compactcrit} and Proposition \ref{mainEstimate1f}, the following result
\begin{thm}
For any fixed $s, t \in \mathbb{R}$ and $x \in \mathbb{R}^d$, the sequence $\{X^{n,s,x}_t\}_{n=1}^\infty$ converges strongly in $L^2(\Omega,\mathbb{R}^d)$ to $X^{s,x}_t$.
\end{thm}

				The next result is a consequence of Proposition \ref{flowEstimate}.
\begin{cor}\label{propcontver}
	Let $X^{s,x}$ be the unique strong solution to the SDE \eqref{Itodiffusion}  and $q>1$ an integer. Then there exists a constant $C=C(d,k,q)<\infty$ independent of of $x_1,x_2$ in every bounded subset of $\mathbb{R}^d$ such that
	\begin{align}\label{estkolmogor+}
	E\Big[\Big|X_{t_1}^{s_1,x_1}-X_{t_2}^{s_2,x_2}\Big|^q\Big]\leq C\Big(|s_2-s_1|^{q/2}+|t_2-t_1|^{q/2}+|x_2-x_1|^{q}\Big)
	\end{align}
	for all $s_2,s_1,t_2,t_1,x_2,x_1$.
	
	In particular, there exists a locally  H\"older continuous version of the random field $(s,t,x)\mapsto X_{t}^{s,x}$ with H\"older constant $\alpha < \frac{1}{2}$ in $s,t$ and $\alpha < 1$ in $x$.
	\end{cor}				
			
			\begin{proof}
				Assume that the above condition holds. Moreover, without loss of generality, assume that $0\leq s_1<s_2<t_1<t_2$. Then
				\begin{align*}
				X_{t_1}^{n,s_1,x_1}-X_{t_2}^{n,s_2,x_2}=& x_1-x_2+\int_{s_1}^{t_1}b_n(u,X_u^{n,s_1,x_1})\diffns u -\int_{s_2}^{t_2}b_n(u,X_u^{n,s_2,x_2})\diffns u\\ &+(B_{t_1}-B_{s_1})-(B_{t_2}-B_{s_2})\\
			=&	x_1-x_2 +\int_{s_1}^{s_2}b_n(u,X_u^{n,s_1,x_1})\diffns u-\int_{t_1}^{t_2}b_n(u,X_u^{n,s_2,x_2})\diffns u\\
			&+\int_{s_2}^{t_1}\Big(b_n(u,X_u^{n,s_1,x_1})-b_n(u,X_u^{n,s_1,x_2})\Big)\diffns u\\ &+\int_{s_2}^{t_1}\Big(b_n(u,X_u^{n,s_1,x_2})-b_n(u,X_u^{n,s_2,x_2})\Big)\diffns u\\
			&+(B_{t_2}-B_{t_1})-(B_{s_2}-B_{s_1}).
				\end{align*}
				Using H\"older's inequality, we get
					\begin{align}
					E\Big[\Big|X_{t_1}^{n,s_1,x_1}-X_{t_2}^{n,s_2,x_2}\Big|^q\Big] \leq & 7^{q-1}\Big( |x_1-x_2|^p +E\Big[\Big |\int_{s_1}^{s_2}b_n(u,X_u^{n,s_1,x_1})\diffns u\Big |^q\Big]\notag\\
					&+	E\Big[\Big | \int_{t_1}^{t_2}b_n(u,X_u^{n,s_2,x_2})\diffns u\Big|^q\Big]\notag\\
				&+	E\Big[\Big| \int_{s_2}^{t_1}\Big(b_n(u,X_u^{n,s_1,x_1})-b_n(u,X_u^{n,s_1,x_2})\Big)\diffns u\Big|^q\Big]\notag\\
						&+	E\Big[\Big| \int_{s_2}^{t_1}\Big(b_n(u,X_u^{n,s_1,x_2})-b_n(u,X_u^{n,s_2,x_2})\Big)\diffns u\Big|^q\Big]\notag\\
							&	+E\Big[\Big| B_{s_2}-B_{s_1}\Big|^q\Big]+E\Big[\Big| B_{t_2}-B_{t_1}\Big|^q\Big]
					\Big)  \notag\\
					\leq & 7^{q-1}\frac{(2q)!}{2^qq!}\Big( |x_1-x_2|^q +| t_2-t_1|^{q/2}+| s_2-s_1|^{q/2}  \notag \\
					&+E\Big[\Big|\int_{s_1}^{s_2}b_n(u,X_u^{n,s_1,x_1})\diffns u\Big|^q\Big]+	 E\Big[\Big| \int_{t_1}^{t_2}b_n(u,X_u^{n,s_2,x_2})\diffns u\Big|^q\Big]\notag \\
					&+	E\Big[\Big| \int_{s_2}^{t_1}\Big(b_n(u,X_u^{n,s_1,x_1})-b_n(u,X_u^{n,s_1,x_2})\Big)\diffns u\Big|^q\Big]\notag \\
					&+	E\Big[\Big| \int_{s_2}^{t_1}\Big(b_n(u,X_u^{n,s_1,x_2})-b_n(u,X_u^{n,s_2,x_2})\Big)\diffns u\Big|^q\Big]
					\Big).  \label{estkolmogor}
								\end{align}
						The assumption on $b_n$ and H\"older's inequality yield
								\begin{align*}
							E\Big[\Big|\int_{s_1}^{s_2}b_n(u,X_u^{n,s_1,x_1})\diffns u\Big|^q\Big]
\leq& 							\int_{s_1}^{s_2}E\Big[\Big|b_n(u,X_u^{n,s_1,x_1})\Big|^q\Big]\diffns u\times |s_1-s_2|^{q-1}\\
&\leq \int_{s_1}^{s_2}E\Big[2^{q-1}k^q+2^{q-1}k^q|X_u^{n,s_1,x_1}|^q\Big]\diffns u\times |s_1-s_2|^{q-1}\\
& = 2^{q-1}k^q|s_1-s_2|^{q-1}\Big(|s_1-s_2|+\int_{s_1}^{s_2}E\Big[|X_u^{n,s_1,x_1}|^q\Big]\diffns u\Big).
								\end{align*}
								
				Using again H\"older's inequality and Gronwall's Lemma, we get
				\begin{align*}
	\int_{s_1}^{s_2}E\Big[|X_u^{n,s_1,x_1}|^q\Big]\diffns u \leq & 	\int_{s_1}^{s_2} |x_1|^q \diffns u + \int_{s_1}^{s_2} E\Big|\int_{s_1}^u  b_n(r_1,X_{r_1}^{n,s_1,x_1}) \diffns r_1\Big|^q\diffns u \\
	& + \int_{s_1}^{s_2}  E\Big[\Big| B_{u}-B_{s_1}\Big|^q\Big]\diffns u\\
	\leq &|x_1|^q|s_1-s_2|+ \int_{s_1}^{s_2} \int_{s_1}^u  E\Big|b_n(r_1,X_{r_1}^{n,s_1,x_1}) \Big|^q\diffns r_1 (u-s_1)^{q-1}\diffns u\\
	&+ \frac{(2q)!}{2^qq!} \int_{s_1}^{s_2} |u-s_1|^{q/2}\diffns u\\
	\leq &|x_1|^q|s_2-s_1|+ \frac{2^{q-1}k^q}{q}|s_2-s_1|^{q+1}+
	\frac{(2q)!}{2^qq!(q/2+1)} |s_2-s_1|^{q/2+1}\\
	&+2^{q-1}k^q\int_{s_1}^{s_2} (u-s_1)^{q-1}\Big(\int_{s_1}^u   E\Big[|X_{r_1}^{n,s_1,x_1}|^q\Big]\diffns r_1 \Big)\diffns u\\
	\leq & \Big(|x_1|^q|s_2-s_1|+ \frac{2^{q-1}k^q}{q}|s_2-s_1|^{q+1}+
	\frac{(2q)!}{2^qq!(q/2+1)} |s_2-s_1|^{q/2+1}\Big)\\
&	\times \exp\Big\{ 2^{q-1}k^q\int_{s_1}^{s_2} (u-s_1)^{q-1}\diffns u\Big\}\\
\leq & C_q|s_2-s_1|\Big(|x_1|^q+ \frac{2^{q-1}k^q}{q}|s_2-s_1|^{q}+
\frac{(2q)!}{2^qq!(q/2+1)} |s_2-s_1|^{q/2}\Big)\\
&	\times \exp\Big\{\frac{ 2^{q-1}k^q}{q} |s_2-s_1|^{q}\Big\}.
					\end{align*}
Thus
					\begin{align}\label{estkolmogor1}
					E\Big[\Big|\int_{s_1}^{s_2}b_n(u,X_u^{n,s_1,x_1})\diffns u\Big|^q\Big]
					\leq& 		 |s_1-s_2|^{q/2}C(q,k,|x_1|),
					\end{align}
				where
				\begin{align*}
C(q,k,|x_1|)=& \Big(1+
C_q\Big\{|x_1|^q+ \frac{2^{q-1}k^q}{q}|s_2-s_1|^{q}+
\frac{(2q)!}{2^qq!(q/2+1)} |s_2-s_1|^{q/2}\Big\} \exp\Big\{\frac{ 2^{q-1}k^q}{q} |s_2-s_1|^{q}\Big\}	 \Big)		\\
& \times 2^{q-1}k^q|s_1-s_2|^{q/2}.
			\end{align*}
A similar bound holds for $E\Big[\Big| \int_{t_1}^{t_2}b_n(u,X_u^{n,s_2,x_2})\diffns u\Big|^q\Big]$.
				
By the  Mean Value Theorem and Proposition \ref{mainEstimate}, we have
				\begin{align}\label{estkolmogor11}
				E\Big[\Big| &\int_{s_2}^{t_1}\Big(b_n(u,X_u^{n,s_1,x_1})-b_n(u,X_u^{n,s_1,x_2})\Big)\diffns u\Big|^q\Big]\notag\\
				=&|x_2-x_1|^q E\Big[\Big| \int_{s_2}^{t_1}\int_0^1b^\prime_n(u,X_u^{n,s_1,x_1+\tau (x_2-x_1)})\frac{\partial}{\partial x}X_u^{n,s_1,x_1+\tau (x_2-x_1)}\diffns u\diffns \tau\Big|^q\Big] \notag\\
				\leq & |x_2-x_1|^q \int_0^1E\Big[\Big| \int_{s_2}^{t_1}b^\prime_n(u,X_u^{n,s_1,x_1+\tau (x_2-x_1)})\frac{\partial}{\partial x}X_u^{n,s_1,x_1+\tau (x_2-x_1)}\diffns u\Big|^q\Big]\diffns \tau \notag\\
					= & |x_2-x_1|^q \int_0^1E\Big[\Big| \frac{\partial}{\partial x}X_{t_1}^{n,s_1,x_1+\tau (x_2-x_1)}-\frac{\partial}{\partial x}X_{s_2}^{n,s_1,x_1+\tau (x_2-x_1)}\Big|^q\Big]\diffns \tau \notag\\
					\leq & C_q|x_2-x_1|^q \underset{t\in [s_1,1],x\in B(0,|x_1|+|x_2|)}{\sup}   E\Big[\Big|\frac{\partial}{\partial x}X_{t_1}^{n,s_1,x}\Big|^q\Big]\notag\\
				\leq &	C(k,q,d)|x_2-x_1|^q.
			\end{align}
			Using the Markov property, we obtain
					\begin{align}\label{estkolmogor111}
				& E\Big[\Big| \int_{s_2}^{t_1}\Big(b_n(u,X_u^{n,s_1,x_2})-b_n(u,X_u^{n,s_2,x_2})\Big)\diffns u\Big|^q\Big]\notag \\
				\leq &  \int_{s_2}^{t_1}E\Big[\Big|\Big(b_n(u,X_u^{n,s_1,x_2})-b_n(u,X_u^{n,s_2,x_2})\Big)\Big|^q\Big]\diffns u\notag \\
				= & \int_{s_2}^{t_1}E\Big[E\Big[\Big|\Big(b_n(u,X_u^{n,s_2,y})-b_n(u,X_u^{n,s_2,x_2})\Big)\Big|^q\Big]\Big|_{y=X_{s_2}^{n,s_1,x_2}}\Big]\diffns u\notag \\
				\leq & CE\Big[\Big|X_{s_2}^{n,s_1,x_2}-x_2\Big|^q\Big]\notag \\
				\leq & C(q,k)|s_2-s_1|^{q/2}.
			\end{align}
			 Combining \eqref{estkolmogor}-\eqref{estkolmogor111}, we get \eqref{estkolmogor+} with $X$ replaced by $X^n$. 
			
			Since $X_{t_2}^{n,s_2,x_2}\rightarrow X_{t_2}^{s_2,x_2}$ and $X_{t_1}^{n,s_1,x_1}\rightarrow X_{t_1}^{s_1,x_1}$ strongly in $L^2(\Omega:\mathbb{R}^d)$ as $n\rightarrow \infty$, it follows that there exists subsequences of $\{X_{t_2}^{n,s_2,x_2}\}_{n\geq 1}$ and $\{X_{t_1}^{n,s_1,x_1}\}_{n\geq 1}$ that converge almost everywhere. The result then follows from Fatou's lemma.
				\end{proof}
					In the next Lemma, we prove that the sequence defined by  $\{X_{t}^{n,x}\}_{n\geq 1}:=\{X_{t}^{n,0,x}\}_{n\geq 1}$ converges to $X_{t}^{x}:=X_{t}^{0,x}$.
					
					\begin{lemm}\label{lemweakconv111}
						
						For any $\varphi \in C^\infty_0(\mathbb{R}^d;\mathbb{R}^d)$ and $t\in [0,T], \, T>0$ the sequence
						$$
					\langle	X_{t}^{n},\varphi \rangle =\int_{\mathbb{R}^d} \langle	 X_{t}^{n,x},\varphi (x)\rangle _{\mathbb{R}^d} \diffns x
											$$
											converges to $\langle	X_{t},\varphi \rangle $ in $L^2(\Omega,\mathbb{R}^d)$.
						\end{lemm}
			\begin{proof}
				Let $D_s$ be the Malliavin derivative and let $U$ be the compact support of $\varphi$. Then we have
				\begin{align*}
					E\Big[	|D_s\langle	X_{t}^{n},\varphi \rangle |\Big]\leq &\|\varphi\|^2_{L^2(\mathbb{R}^d)} |U|\sup_{x\in U} E\Big[	|	D_sX_{t}^{n,x} |^2\Big]
					\end{align*}
				and
					\begin{align*}
						E\Big[	|D_{s_1}\langle	X_{t}^{n},\varphi \rangle-D_{s_2}\langle	 X_{t}^{n},\varphi \rangle |\Big]\leq &\|\varphi\|^2_{L^2(\mathbb{R}^d)} |U|\sup_{x\in U} E\Big[	|	 D_{s_1}X_{t}^{n,x} -D_{s_2}X_{t}^{n,x} |^2\Big].
					\end{align*}
				Using the compactness criteria (Corollary \ref{compactcrit}), there exists a subsequence $\langle	X_{t}^{n(k)},\varphi \rangle $ converging in $L^2(\Omega,\mathbb{R}^d)$ as $k\rightarrow \infty$ to a limit $Y(\varphi)$. Hence as in the proof of Lemma \ref{lemweakcon1}, one can show that $E\Big[\langle	 X_{t}^{n},\varphi\rangle \mathcal{E}\Big(\int_0^Th(u)\diffns u\Big)\Big]$ converges to $E\Big[\langle	 X_{t},\varphi\rangle \mathcal{E}\Big(\int_0^Th(u)\diffns u\Big)\Big]$ for all $h\in C^1_b(\mathbb{R};\mathbb{R}^d)$. Therefore $\langle	 X_{t,n},\varphi\rangle$ converges weakly to $\langle	 X_{t},\varphi\rangle$ and the uniqueness of the limits implies that
				$$
				Y(\varphi)=\langle	X_{t},\varphi\rangle.
				$$
		Using contradiction, suppose that the full sequence does not converge. Then there exist $\varepsilon_0>0$ and a subsequence $\langle	X_{t,n(k)},\varphi\rangle$ for which
			\begin{align}\label{eqrefweconv22}
				\|\langle	X_{t,n(k)},\varphi\rangle-\langle	X_{t,n},\varphi\rangle\|\geq \varepsilon_0
				\end{align}
				for every $k$. But using once more Lemma \ref{lemweakcon1}, one can show that there exist a further subsequence $\langle	X_{t,n(k_i)},\varphi\rangle$ of $\langle	 X_{t,n(k)},\varphi\rangle$ that converges to $\langle	X_{t},\varphi\rangle$. This is a contradiction to \eqref{eqrefweconv22}.
				
				\end{proof}
			
			We are now ready to conclude the proof of Proposition \ref{propmainres1}.
			\begin{proof}[Proof Proposition \ref{propmainres1}]
	For each bounded set $B$ of $\mathbb{R}^d$, we have
 $$\underset{n}{\sup}\underset{x\in B}{\sup}E\Big[\Big|\frac{\partial}{\partial x}X_t^{n,x}\Big|^p\Big]<\infty.$$
			From this, it follows that there exists a subsequence $\frac{\partial}{\partial x}X_t^{n(k),x}$ that converges to an element $Y$ in the weak topology of $L^2(\Omega,L^p(U))$. Hence for any $A\in \mathcal{F}$ and $\varphi\in C^\infty_0(U;\mathbb{R}^d)$
			\begin{align*}
				E[1_A\langle X_t,\varphi^\prime \rangle]=& \underset{k \to \infty}{\lim}E[1_A\langle X_t^{n(k)},\varphi^\prime \rangle]\\
				=&-\underset{k\leftarrow \infty}{\lim}E[1_A\langle \frac{\partial}{\partial x}X_t^{n(k)},\varphi\rangle]\\
				=&- E[1_A\langle X_t,\varphi\rangle].
				\end{align*}
			Hence for any $\varphi\in C^\infty_0(U;\mathbb{R}^d)$,
			\begin{align}\label{eqweakconv23}
				E[1_A\langle X_t,\varphi^\prime \rangle]=- E[1_A\langle X_t,\varphi\rangle] \,\,\,P\text{-a.s.}
		\end{align}
		
		We now construct a measurable set $\Omega_0 \subset \Omega$ of full measure such that $X_t^\cdot$ has a weak derivative $\frac{\partial}{\partial x}X_t^{x}$ on this subset. Let $\{\varphi_n\}$ be a sequence in $C_0^\infty(U;\mathbb{R}^d)$ dense in $W_0^{1,2}(U;\mathbb{R}^d)$. Replace $\varphi$ by $\varphi_n$ and choose a measurable subset $\Omega_n$ of $\Omega$ with full measure such that \eqref{eqweakconv23} holds on $\Omega_n$. Set $\Omega_0=\underset{n\geq 1}{\cap}\Omega_n$, then $\Omega_0$ satisfies the required property.
		\end{proof}
		We get the following result for a weighted Sobolev space.

			\begin{lemm}\label{lemmbond11}
				For all $p\in (1,\infty)$, we have
			$$
			X_t^\cdot \in L^2(\Omega, W^{1,p}(\mathbb{R}^d;\mathfrak{p}(x)\diffns x)).
			$$
				\end{lemm}
			
		\begin{proof}
			Without loss of generality, we assume that $d=1$. We first show that
			\begin{align}\label{equpflow1}
			E\Big[\Big(\int_{\mathbb{R}^d}\Big|\frac{\partial}{\partial x} X^x_t\Big|^p\mathfrak{p}(x)\diffns x\Big)^{2/p}\Big]<\infty.
			\end{align}
			We prove this by successive conditioning on the filtration generated by the Brownian motion and successive use of the flow property and Lemma \ref{lemmaproexp1}. Let $X^{n,x}_t$ be the sequence as defined in Lemma \ref{lemweakconv111}.
			
			Choose once more $\delta_0$ as in Lemma \ref{lemmaproexp1} and set $ \tau:=\frac{\delta_0}{64\sqrt{2} k^2},\, s_i = i\tau, x_i:= X^{n,0,x}_{s_i}, \frac{\partial}{\partial x_i}X^{n,s_i,x_i}_{s_{i+1}}:= \frac{\partial}{\partial x}X^{n,s_i,x}_{s_{i+1}}\bigg |_{x=x_i}, i \geq 1$. There exists $m$ positive integer such that $s_{m}\leq t< s_{m+1}$.

			For $p\geq 2$, using H\"older inequality with respect to the measure $P$ and Fubini Theorem, we have
			\begin{align*}
				E\Big[\Big(\int_{\mathbb{R}^d}\Big|\frac{\partial}{\partial x} X^{n,x}_t\Big|^p\mathfrak{p}(x)\diffns x\Big)^{2/p}\Big]
				\leq & \Big(E\Big[\int_{\mathbb{R}^d}\Big|\frac{\partial}{\partial x} X^{n,x}_t\Big|^p\mathfrak{p}(x)\diffns x\Big]\Big)^{2/p}\\
				\leq & \Big(\int_{\mathbb{R}^d}E\Big[\Big|\frac{\partial}{\partial x} X^{n,x}_t\Big|^p\mathfrak{p}(x)\Big]\diffns x\Big)^{2/p}.
			\end{align*}
			Using the flow property, the chain rule, the Cauchy inequality and Proposition \ref{flowEstimate}, we have
				\begin{align}\label{eqapproxiflow1}
					E\Big[\Big|\frac{\partial}{\partial x} X^{n,0,x}_t\Big|^p\Big]
					= & E \Big ( E\Big [\Big |\frac{\partial}{\partial x_m} X^{n,s_m,x_m}_{t} \cdot \frac{\partial}{\partial x_{m-1}}X^{n,s_{m-1},x_{m-1}}_{s_m} \cdots \frac{\partial}{\partial x}X^{n,0,x}_{s_1}\Big |^p \Big | \mathcal F_{s_m} \Big] \Big ) \notag\\
					=& E \Big ( E\Big [\Big |\frac{\partial}{\partial x_m} X^{n,s_m,x_m}_{t} \Big |^p \mathcal F_{s_m} \Big]\cdot \Big|\frac{\partial}{\partial x_{m-1}}X^{n,s_{m-1},x_{m-1}}_{s_m} \cdots \frac{\partial}{\partial x}X^{n,0,x}_{s_1}\Big |^p  \Big ) \notag\\
					\leq & C  E\Big ( C_p(|x_m|,\| \tilde{b} \|_{\infty})e^{6k^2\tau |x_m|^2}  \Big|\frac{\partial}{\partial x_{m-1}}X^{n,s_{m-1},x_{m-1}}_{s_m} \cdots \frac{\partial}{\partial x}X^{n,0,x}_{s_1}\Big |^p  \Big )\notag\\
					\leq  &  C_{p,k} \Big (Ee^{\delta_0|x_m|^2}\Big )^{1/4}  \Big[ E\Big \{ \Big|\frac{\partial}{\partial x_{m-1}}X^{n,s_{m-1},x_{m-1}}_{s_m} \cdots \frac{\partial}{\partial x}X^{n,0,x}_{s_1}\Big |^{2p} \Big \} \Big]^{1/2} \notag \\
					\leq & C \Big (C_1 e^{C_2\delta_0|x|^2 } \Big )^{1/2}  \Big[ E\Big \{ \Big|\frac{\partial}{\partial x_{m-1}}X^{n,s_{m-1},x_{m-1}}_{s_m} \cdots \frac{\partial}{\partial x}X^{n,0,x}_{s_1}\Big |^{2p} \Big \} \Big]^{1/2} ,
				\end{align}
where the last inequality comes from Lemma \ref{lemmaproexp1} with $C_1, C_2$ positive constants independent of $x$, but may depend on $k$. Successive application of the previous step on the second term of the right hand side of \eqref{eqapproxiflow1} gives

				\begin{align}\label{eqapproxiflow2}
					E\Big[\Big|\frac{\partial}{\partial x} X^{n,0,x}_t\Big|^p\Big]
					\leq	& C_{p,m,k,\delta_0} \Big (e^{C_2\delta_0|x|^2 } \Big )^{1/4} \Big (e^{C_2\delta_0|x|^2 } \Big )^{1/8} \cdots \Big (e^{C_2\delta_0|x|^2 } \Big )^{1/2(m+1)}  \Big[ E\Big \{ \Big|\frac{\partial}{\partial x}X^{n,0,x}_{s_1}\Big |^{2mp} \Big \} \Big]^{1/2(m+1)} \notag\\
					\leq	& C_{p,m,k,\delta_0} \Big (e^{C_2\delta_0|x|^2 } \Big )^{1/4} \Big (e^{C_2\delta_0|x|^2 } \Big )^{1/8} \cdots \Big (e^{C_2\delta_0|x|^2 } \Big )^{1/2(m+1)}  \Big (e^{C_2\delta_0|x|^2 } \Big )^{1/2(m+1)} \notag\\
					\leq&  C_{p,m,k,\delta_0} e^{C_{m,k,\delta_0}^\prime|x|^2 } .
				\end{align}
			Thus
				\begin{align*}
					E\Big[\Big(\int_{\mathbb{R}^d}\Big|\frac{\partial}{\partial x} X^{n,x}_t\Big|^p\mathfrak{p}(x)\diffns x\Big)^{2/p}\Big]
					\leq &C_{p,m,k,\delta_0}  \Big(\int_{\mathbb{R}^d} e^{C_{m,k,\delta_0}^\prime|x|^2 }\mathfrak{p}(x)\diffns x\Big)^{2/p}.
				\end{align*}
				It follows that \eqref{equpflow1} is satisfied for $p\geq 2$.

			Now, choose $1\leq p \leq 2$, it follows from the H\"older's inequality with respect to the measure $\mathfrak{p}(x)\diffns x$ that
			\begin{align*}
				E\Big[\Big(\int_{\mathbb{R}^d}\Big|\frac{\partial}{\partial x} X^{n,x}_t\Big|^p\mathfrak{p}(x)\diffns x\Big)^{2/p}\Big]\leq & \Big(\int_{\mathbb{R}^d}\mathfrak{p}(x)\diffns x\Big)^{\frac{2-p}{2}} \Big(E\Big[\int_{\mathbb{R}^d}\Big|\frac{\partial}{\partial x} X^{n,x}_t\Big|^p\mathfrak{p}(x)\diffns x\Big]\Big).
			\end{align*}
			From this, we get that \eqref{equpflow1} holds for $1\leq p \leq 2$.
			
			Independently of the choice of $p$, there exists a subsequence converging to an object $Y\in L^2(\Omega, L^{q}(\mathbb{R}^d;\mathfrak{p}(x)\diffns x))$ in the weak topology. More specifically, for every $A\in \mathcal{F}$ and $f\in  L^{q}(\mathbb{R}^d;\mathfrak{p}(x)\diffns x)$ (with $q$ such that $\frac{1}{p} +\frac{1}{q}=1$), such that $f\mathfrak{p} \in L^{q}(\mathbb{R}^d;\diffns x)$, we have
			$$
			\underset{n\leftarrow \infty}{\lim}E\Big[1_A  \int_{\mathbb{R}^d}\frac{\partial}{\partial x} X^{n(k),x}_tf(x)\mathfrak{p} (x)\diffns x \Big]=E\Big[1_A \int_{\mathbb{R}^d}  Y(x)f(x)\mathfrak{p} (x)\diffns x \Big].
			$$
			Hence $Y$ must be equal to the weak derivative of $X^x_t$ and the results follows.
			\end{proof}

\begin{proof}[Proof of Theorem \ref{thmmain1}]
Let $\mathbb{R}\times \mathbb{R}\times \mathbb{R}^d \ni (s,t,x)\mapsto \phi_{s,t}(x) \in \mathbb{R}^d  $ be the continuous version of the map $\mathbb{R}\times \mathbb{R}\times \mathbb{R}^d \ni (s,t,x)\mapsto X_{t}^{s,x }$ given in Proposition \ref{propcontver}. Denote by $\Omega^\ast$ the set of all $\omega \in \Omega$ for which there exists a unique spatially Sobolev differentiable family of solutions to equation \eqref{eqmain1111}. Since $(\Omega, \mathcal{F},P) $ is a complete probability space, we get that $\Omega^\ast \in \mathcal{F}$ and $P(\Omega^\ast)=1$. Moreover, the uniqueness of solutions to the SDE \eqref{eqmain1111}, implies the following two-parameter group property
\begin{align}\label{flowprop1}
\phi_{s,t}(\cdot,\omega)= \phi_{u,t}(\cdot,\omega)\circ \phi_{s,u}(\cdot,\omega),\,\,\,\,\phi_{s,s}(\cdot,\omega)=x,
					\end{align}
					 is satisfied for all $s,u,t \in \mathbb{R}$, all $x\in \mathbb{R}^d$ and all $\omega \in \Omega^\ast$. In fact, without loss of generality, we can assume that $u<s<t$ and for $s,t \in \mathbb{R}$, there exists an integer $m$ such that $s_m< t-s\leq s_{m+1}$ and one can verify that the flow property \eqref{flowprop1} holds in this case. Repeated use of \eqref{flowprop1} in small intervals of length $\tau$ and the uniqueness of the solution gives the above two-parameter group property for all $s,u,t \in \mathbb{R}$.
					
					 The proof of the theorem is completed by applying Lemma \ref{lemmbond11} and using the relation $\phi_{s,t}(\cdot,\omega)=\phi_{t,s}^{-1}(\cdot,\omega)$.
					\end{proof}

					\begin{proof}[Proof of Theorem \ref{corcocyl}]
						Let $\Omega^\ast$ be as in the proof of Theorem \ref{thmmain1}. We will show that $\theta(t,\cdot)(\Omega^\ast)=\Omega^\ast$ for $t\in \mathbb{R}$. Fix $t\in \mathbb{R}$ and let $\omega \in \Omega^\ast$. The relation \eqref{eqmain1112} yields
							\begin{align}\label{eqmain11121}
								X_{t+t_1}^{t_1,x}(\omega)=x+\int_{t_1}^{t+t_1} b(X_r^{t_1,x}(\omega))\diffns r +B_{t+t_1}(\omega)-B_{t_1}(\omega), \,\, t_1,t \in \mathbb{R}.
							\end{align}
						Using the helix property of $B$ and a change of variable we get from \eqref{eqmain11121}
						\begin{align}\label{eqmain11122}
							X_{t+t_1}^{t_1,x}(\omega)=x+\int_{0}^{t} b(X_{r+t_1}^{t_1,x}(\omega))\diffns r +B_{t}(\theta(t_1,\omega)),\,\,\,t \in \mathbb{R}.
						\end{align}
							
		Substituting $\omega$ by $\theta(t_1,\omega)$, the relation \eqref{eqmain11122} suggests that the SDE (3.5) 
has a Sobolev differentiable family of solutions.
Therefore,   $\theta(t_1,\omega) \in \Omega^\ast$	and hence $\theta(t,\cdot)(\Omega^\ast)\subseteq \Omega^\ast$. Since this holds for arbitrary $t_1\in \mathbb{R}$, we have $\theta(t,\cdot)(\Omega^\ast)=\Omega^\ast$ for all $t\in \mathbb{R}$.
			
			 Moreover,the uniqueness of solutions of the integral equation \eqref{eqmain11121} gives
				\begin{align}\label{eqmain11123}
			 	X_{t+t_1}^{t_1,x}(\omega)=	X_{t_2}^{0,x}(\theta(t_1,\omega))
			 	\end{align}
			 for all $t_1,t_2 \in \mathbb{R}$, all $x\in \mathbb{R}^d$ and all $\omega \in \Omega^\ast$.
			
			 To prove the perfect cocycle property \eqref{cocyprop1}, note that \eqref{eqmain11123} can be rewritten as
			 \begin{align}\label{eqmain11124}
			 	\phi_{t_1,t+t_1}(x,\omega)=		\phi_{0,t_2}(x,\theta(t_1,\omega)),\,\,\,t_1,t_2 \in \mathbb{R},\,x\in \mathbb{R}^d,\,\omega \in \Omega^\ast.
			 \end{align}
			 The perfect cocycle property  \eqref{cocyprop1} now follows by replacing $x$ in the relation \eqref{eqmain11124} by $\phi_{0,t_1}(x,\omega)$ and using the two parameter flow property \eqref{flowprop1}.
						\end{proof}

\section{Application to stochastic delay differential equation}

In this section we consider the following stochastic delay differential equation

\begin{equation}
\left\{	\begin{array}{llll} \label{eqappli1}
		\diffns X (t) = b (X(t-r), X(t,0,(v,\eta)) \diffns t + \diffns B(t), \quad t \geq 0\\
		(X(0), X_0)= (v, \eta) \in M_2 := \mathbb{R}^d \times L^2 ([-r,0], \mathbb{R}^d)
	\end{array}\right.
\end{equation}

\begin{thm}\label{thmainresappli1}
	Suppose that the drift coefficient $b: \mathbb{R}^d \times \mathbb{R}^d \rightarrow \mathbb{R}^d$ in the SDE \eqref{eqappli1} is a Borel-measurable function bounded in the first argument and has linear growth in the second argument.
	Then there exists a unique global strong solution $X$ to the SDE \eqref{eqappli1} adapted to the filtration $\left\{ \mathcal{F}_{t}\right\}
	_{0\leq t\leq T}$. Furthermore,
	the solution $X_t$ is Malliavin differentiable for all $0 \leq t \leq T$.
\end{thm}

The proof of Theorem \ref{thmainresappli1} also uses the relative compactness criteria. Let  consider a sequence $b_n: \mathbb{R}^d \times \mathbb{R}^d \rightarrow \mathbb{R}^d$, $n\ge 1$ of smooth coefficients with compact support such that $b_n(v_1,v_2)$ is Borel measurable, bounded in $v_1 \in \mathbb{R}^d$ and has linear growth in $v_2 \in \mathbb{R}^d$.  We will view the semiflow $\tilde X^n(t,0,(v,\eta)):= (X^n(t), X^n_t)$ as a process with values in the Hilbert state space $M_2$ in order to use Malliavin calculus. The symbol $X^n_t \in L^2([-r,0],R^d)$ stands for the segment $X^n_t(s) := X^n(t +s), \, s \in [-r,0]$ and $r > 0$ is the delay. For brevity, we will often denote $z := (v, \eta) \in M_2$ and by $p_1 : M_2 \to \mathbb{R}^d, p_2: M_2 \to L^2([-r,0],\mathbb{R}^d)$ the natural projections onto $\mathbb{R}^d$ and $L^2([-r,0],\mathbb{R}^d)$.
For any $t_1 > 0$ denote by $X^n(\cdot, t_1,z)$ the solution of the following approximating SDDE starting at $t_1$:

\begin{equation}\label{eqappli2}
X^n(t, t_1, z) := \begin{cases} p_1(z) + \displaystyle \int_{t_1}^t b_n(X^n(u-r, t_1,z),X^n(u,t_1,z)) \diffns u + B(t)-B(t_1)\,\, \,\mbox{ for}\,\, t_1 \leq t, \\
p_2 (z)(t-t_1)\,\,\, \mbox{ for}\,\, t \in  [t_1-r,t_1). \end{cases}
\end{equation}
By the continuation property, we have
\begin{equation}
X^n(t,0,(v,\eta)) := X^n(t,t_1, \tilde X^n(t_1,0,(v,\eta)))  \quad t \geq t_1, \,\, (v,\eta) \in M_2.
\end{equation}
Apply the chain rule (for Mallliavin derivatives) to the above relation and get:
\medskip
\newline
\begin{equation}\label{eqappli30}
D_sX^n(t,0,(v,\eta)) = D_s X^n(t,t_1,z) \bigg |_{z=\tilde X^n(t_1,0,(v,\eta)))} + D X^n(t,t_1,z) \bigg |_{z=\tilde X^n(t_1,0,(v,\eta)))} \cdot D_s \tilde X^n(t_1,0,(v,\eta))
\end{equation}
for $0 < s < t$ and each deterministic initial data $(v,\eta) \in M_2$.
\medskip

Let $t_1$ be the supremum of all $t \geq r$ such that there exist $C_d > 0$ and the following two estimates hold:
\begin{equation}\label{eqappli31}
\begin{cases}
&\displaystyle \sup_{n \geq 1} E\|D_sX^n(t,0,(v,\eta)) - D_{s'}X^n(t,0, (v,\eta))\|^2 \leq  C_d |s-s'| \,\,   \mbox{ for}\,\, 0 \leq s,s' \leq t, \\
&\displaystyle \sup_{n \geq 1}\displaystyle \sup_{0 \leq s \leq t} E\|D_sX^n(t,0,(v,\eta))\|^2 \leq C_d
\end{cases}
\end{equation}
for all deterministic initial data $(v,\eta) \in M_2$.

\textit{Claim}:
We claim that $t_1 = \infty.$

\begin{proof}

Use contradiction: Suppose $t_1 < \infty$. We will show by continuation and a conditioning argument that there exist $t_2 \in  (t_1,t_1+r)$ (with $t_2-t_1$ possibly small) such that the inequalities \eqref{eqappli31} hold for $t=t_2$. This will contradict the choice of $t_1$ as a supremum.

Consider the following cases:

\medskip
\noindent
Case 1: $t_1 \leq s \leq t \leq t_1+r$:

For deterministic $z \in M_2$, take the Malliavin derivative $D_s$ in \eqref{eqappli2} to obtain
\begin{equation} \label{eqappli32}
D_sX^n(t, t_1, z) = \int_{s}^t D_2 b_n(p_2z(u-r),X^n(u,t_1,z))D_s X^n(u, t_1, z)  \diffns u + I.
\end{equation}
Using successive iterations in the above integral equation, the Girsanov theorem and integrations by parts, we obtain $t_2 > t_1$ with $t_2 -t_1$ possibly small  and a positive constant $C_d$ (independent of $z \in M_2, n \geq 1$) such that
\begin{equation}\label{eqappli33}
\displaystyle \sup_{t_1 \leq s  \leq t} E\|D_sX^n(t,t_1,z)\|^2 \leq C_d.
\end{equation}
Next, we use \eqref{eqappli30}, conditioning on $z = \tilde X^n (t_1, (v,\eta))$, and \eqref{eqappli33} to get the following:
\begin{equation}\label{eqappli34}
E\| D_sX^n(t,0,(v,\eta))\|^2 = E\bigg ( E \bigg ( \|D_s X^n(t,t_1,z)\|^2 \bigg |_{z=\tilde X^n(t_1,0,(v,\eta)))}\big | \tilde X^n(t_1,0,(v,\eta)\bigg )\bigg ) \leq C_d
\end{equation}
for $t_1 <s \leq t \leq t_2$.

Therefore, from the above inequality, we get
\begin{equation}\label{eqappli35}
\displaystyle \sup_{n \geq 1}\displaystyle \sup_{t_1 \leq s \leq t} E\| D_sX^n(t,0,(v,\eta))\|^2  \leq C_d
\end{equation}
for $t_1 \leq t \leq t_2$.

\medskip
\noindent
Case 2: $0 \leq s \leq t_1 \leq t \leq t_2$:



Taking Malliavin derivatives $D_s$ for $s \leq t_1$ in the integral equation \eqref{eqappli2}, we get
\begin{equation}\label{eqappli36}
D_sX^n(t, t_1, z) = \int_{t_1}^t D_2 b_n(p_2z(u-r),X^n(u,t_1,z))D_s X^n(u, t_1, z) \diffns u .
\end{equation}
The above linear integral equation implies that
\begin{equation}\label{eqappli37}
D_sX^n(t, t_1, z)=0, \,\, a.s., \, \, s \leq t_1 \leq t, \,\, n \geq 1
\end{equation}
for any deterministic $z \in M_2$.

As before, we apply the chain rule in \eqref{eqappli30} followed by conditioning with respect to $\tilde X^n(t_1,0,(v,\eta))$ together with the above equality, to get
\begin{equation} \label{eqappli38}
\begin{array}{ll}
E\| D_sX^n(t,0,(v,\eta))\|^2 \leq  &2 E\bigg ( E \bigg ( \|D_s X^n(t,t_1,z)\|^2 \bigg |_{z=\tilde X^n(t_1,0,(v,\eta)))}\big | \tilde X^n(t_1,0,(v,\eta)\bigg )\bigg )\\
&+ 2E \bigg ( E \bigg (\|D X^n(t,t_1,z) \bigg |_{z=\tilde X^n(t_1,0,(v,\eta)))} \cdot D_s \tilde X^n(t_1,0,(v,\eta)) \|^2 \big | \tilde X^n(t_1,0,(v,\eta)\bigg )\bigg ) \\
&= 2 E\bigg ( E \bigg ( \|D_s X^n(t,t_1,z)\|^2 \bigg |_{z=\tilde X^n(t_1,0,(v,\eta)))}\big | \tilde X^n(t_1,0,(v,\eta)\bigg )\bigg )
\\
&=0
\end{array}
\end{equation}
for $s \leq t_1 \leq t \leq t_2$ and all deterministic $(v,\eta) \in M_2$..

Now combine the second estimate in \eqref{eqappli31} (for $t = t_1$) with \eqref{eqappli35} and \eqref{eqappli38} to get
\begin{equation}\label{eqappli39}
\displaystyle \sup_{n \geq 1}\displaystyle \sup_{0 \leq s \leq t} E\| D_sX^n(t,0,(v,\eta))\|^2  \leq C_d
\end{equation}
for $t_1 < t \leq t_2$ with a positive constant $C_d$ (denoted by the same symbol). This shows that the second estimate in \eqref{eqappli31} still holds for $t_2 \geq  t \geq t_1$ and therefore contradicts the maximality of $t_1$ if $t_1 < \infty$.

It remains to show that the first estimate in \eqref{eqappli31} also holds for $t_1 \leq t \leq t_2$. To do this consider the following
cases:

\medskip
\noindent
Case 3: $t_1 < s' < s  \leq t \leq t_2$:

Using the chain rule in \eqref{eqappli30} gives
\begin{equation}\label{eqapplimall20}
D_sX^n(t,0,(v,\eta))- D_{s'}X^n(t,0,(v,\eta)) = [D_s X^n(t,t_1,z) - D_s X^n(t,t_1,z)] \bigg |_{z=\tilde X^n(t_1,0,(v,\eta)))}.
\end{equation}
For any fixed $z \in M_2$, consider the expression
\begin{align} \label{eqapplimall21}
&D_s X^n(t,t_1,z) - D_{s'} X^n(t,t_1,z)  \notag \\
=& \displaystyle \int_{t_1}^t D_2 b_n(p_2z(u-r),X^n(u,t_1,z))[D_s X^n(u, t_1, z)- D_{s'} X^n(u, t_1, z)]\diffns u.
\end{align}

The above linear integral equation implies that
\begin{equation}
[D_s X^n(t,t_1,z) - D_{s'} X^n(t,t_1,z)] =0  \label{eqapplimall22}
\end{equation}
for any fixed $z \in M_2$.

Now taking $E (\| \cdot \|^2)$ on both sides of \eqref{eqapplimall20}, conditioning with respect to
$X^n(t_1,0,(v,\eta))$ and using \eqref{eqapplimall22}, gives
\begin{equation}\label{eqapplimall23}
E\|D_sX^n(t,0,(v,\eta))- D_{s'}X^n(t,0,(v,\eta))\|^2  = 0 \leq C_d |s-s'|, \,\,\, t_1 < s' < s  \leq t \leq t_2
\end{equation}
for any $(v,\eta) \in M_2$.

\medskip
\noindent
Case 4: $s' < t_1 < s  \leq t \leq t_2$:

Fix any $z \in M_2$ and use \eqref{eqappli32} and \eqref{eqappli36} to get
%
\begin{align}\label{eqapplimall24}
& D_s X^n(t,t_1,z) - D_{s'} X^n(t,t_1,z)\notag \\
= &\displaystyle \int_{s}^t D_2 b_n(p_2z(u-r),X^n(u,t_1,z))D_s X^n(u, t_1, z) \diffns u + I  \notag \\
 -& \displaystyle\int_{t_1}^t D_2 b_n(p_2z(u-r),X^n(u,t_1,z))D_{s'} X^n(u, t_1, z) \diffns  u \notag \\
=& \displaystyle \int_{s}^t D_2 b_n(p_2z(u-r),X^n(u,t_1,z))[D_s X^n(u, t_1, z)- D_{s'} X^n(u, t_1, z)]\diffns  u  + I \notag \\
+& \displaystyle \int_{s}^t D_2 b_n(p_2z(u-r),X^n(u,t_1,z))D_{s'} X^n(u, t_1, z) \diffns  u \notag \\
 -& \displaystyle \int_{t_1}^t D_2 b_n(p_2z(u-r),X^n(u,t_1,z))D_{s'} X^n(u, t_1, z) \diffns  u \notag \\
=& I - D_{s'} X^n(s,t_1,z) \notag \\
&+ \displaystyle \int_{s}^t D_2 b_n(p_2z(u-r),X^n(u,t_1,z))[D_s X^n(u, t_1, z)- D_{s'} X^n(u, t_1, z)]\diffns  u.
\end{align}
%
Using successive iterations in the last equality above, the Girsanov Theorem and successive integrations by parts, we obtain $t_2 > t_1$
(with $(t_2 -t_1)$ possibly small) and a positive constant $\tilde C_d = \tilde C_d (\| \tilde b \|_{\infty})$, independent of $z \in M_2$ and $n \geq 1$ such that
\begin{equation}\label{eqapplimall25}
\displaystyle \sup_{n \geq 1} E\|D_sX^n(t,t_1,z)- D_{s'}X^n(t,t_1,z)\|^2 \leq  \tilde C_d |s-s'| \,\,\,\mbox{for}\,\, s' < t_1 < s < t < t_2.
\end{equation}
Next, we use the above estimate, the chain rule (as in \eqref{eqappli38}), and conditioning on $z = \tilde X^n (t_1, (v,\eta))$, to get
\begin{align}\label{eqapplimall26}
& E \|D_sX^n(t,0,(v,\eta))- D_{s'}X^n(t,0,(v,\eta))\|^2 \notag \\
\leq & 2E \bigg (E \bigg ( \|[D_s X^n(t,t_1,z) - D_s X^n(t,t_1,z)]\|^2 \bigg |_{z=\tilde X^n(t_1,0,(v,\eta)))} \bigg |\tilde X^n (t_1, (v,\eta) \bigg )\bigg )\notag \\
\leq & C_d |s-s'| \,\,\,\mbox{for}\,\, s' < t_1 < s < t < t_2,
\end{align}
where $C_d =2 \tilde C_d$.

\medskip
\noindent
Case 5: $s' < s < t_1 <  t \leq t_2$:

In this case, for fixed $z \in M_2$, we have

\begin{align}\label{eqapplimall27}
& [D_s X^n(t,t_1,z) - D_{s'} X^n(t,t_1,z)] \notag\\
= &\displaystyle \int_{t_1}^t D_2 b_n(p_2z(u-r),X^n(u,t_1,z))[D_s X^n(u, t_1, z)- D_{s'} X^n(u, t_1, z)]\diffns  u,
\end{align}
which implies that
\begin{equation}\label{eqapplimall28}
\begin{array}{ll}
[D_s X^n(t,t_1,z) - D_{s'} X^n(t,t_1,z)] = 0,  \,\, a.s.
\end{array}
\end{equation}
Therefore, using similar arguments as before (chain rule and conditioning), we obtain
\begin{equation}\label{eqapplimall29}
E \|D_sX^n(t,0,(v,\eta))- D_{s'}X^n(t,0,(v,\eta))\|^2 =0 \,\,\,\mbox{for}\,\, s' < s < t_1 < t < t_2.
\end{equation}
Finally, putting together the first estimate in \eqref{eqappli31} with $t =t_1$, \eqref{eqapplimall23}, \eqref{eqapplimall26} and \eqref{eqapplimall29}, it follows that the first estimate in \eqref{eqappli31} holds for $t_1 < t < t_2$. This contradicts the maximal choice of $t_1$ and completes the proof of the proposition.  Thus $t_1 = \infty$.
	\end{proof}
	
	Further application of the results in Section \ref{main results} are given in \cite{MenouA17}
	

	
	
\appendix
\setcounter{section}{0}
\numberwithin{equation}{section}
\section{Compactness criteria}

The proposed construction of the strong solution and the stochastic flow for the SDE \eqref{Itodiffusion}
is based on the following relative compactness criteria from Malliavin calculus due to \cite{DPMN92}.


\begin{thm}
	\label{MCompactness}Let $\left\{ \left( \Omega ,\mathcal{A},P\right)
	;H\right\} $ be a Gaussian probability space, that is $\left( \Omega ,%
	\mathcal{A},P\right) $ is a probability space and $H$ a separable closed
	subspace of Gaussian random variables in $L^{2}(\Omega )$, which generate
	the $\sigma $-field $\mathcal{A}$. Denote by $\mathbf{D}$ the derivative
	operator acting on elementary smooth random variables in the sense that%
	\begin{equation*}
	\mathbf{D}(f(h_{1},\ldots,h_{n}))=\sum_{i=1}^{n}\partial
	_{i}f(h_{1},\ldots,h_{n})h_{i},\text{ }h_{i}\in H,f\in C_{b}^{\infty }(\mathbb{R%
	}^{n}).
	\end{equation*}%
	Further let $\mathbf{D}_{1,2}$ be the closure of the family of elementary
	smooth random variables with respect to the norm%
	\begin{equation*}
	\left\Vert F\right\Vert _{1,2}:=\left\Vert F\right\Vert _{L^{2}(\Omega
		)}+\left\Vert \mathbf{D}F\right\Vert _{L^{2}(\Omega ;H)}.
	\end{equation*}%
	Assume that $C$ is a self-adjoint compact operator on $H$ with dense image.
	Then for any $c>0$ the set%
	\begin{equation*}
	\mathcal{G}=\left\{ G\in \mathbf{D}_{1,2}:\left\Vert G\right\Vert
	_{L^{2}(\Omega )}+\left\Vert C^{-1} \mathbf{D} \,G\right\Vert _{L^{2}(\Omega ;H)}\leq
	c\right\}
	\end{equation*}%
	is relatively compact in $L^{2}(\Omega )$.
\end{thm}

The relative compactness criteria in our setting required the subsequent result (see  \cite[Lemma 1] {DPMN92}).

\begin{lemm}
	\label{DaPMN} Let $v_{s},s\geq 0$ be the Haar basis of $L^{2}([0,1])$. For
	any $0<\alpha <1/2$ define the operator $A_{\alpha }$ on $L^{2}([0,1])$ by%
	\begin{equation*}
	A_{\alpha }v_{s}=2^{k\alpha }v_{s}\text{, if }s=2^{k}+j\text{ }
	\end{equation*}%
	for $k\geq 0,0\leq j\leq 2^{k}$ and%
	\begin{equation*}
	A_{\alpha }1=1.
	\end{equation*}%
	Then for all $\beta $ with $\alpha <\beta <(1/2),$ there exists a constant $%
	c_{1}$ such that%
	\begin{equation*}
	\left\Vert A_{\alpha }f\right\Vert \leq c_{1}\left\{ \left\Vert f\right\Vert
	_{L^{2}([0,1])}+\left( \int_{0}^{1}\int_{0}^{1}\frac{\left\vert
		f(t)-f(t^{\prime })\right\vert ^{2}}{\left\vert t-t^{\prime }\right\vert
		^{1+2\beta }}\diffns t\diffns t^{\prime }\right) ^{1/2}\right\} .
	\end{equation*}
\end{lemm}
The next compactness criteria whoch plays a key role in the proof of our results is a direct consequence of Theorem \ref{MCompactness} and Lemma \ref{DaPMN}.

\begin{cor} \label{compactcrit}
	Let $X_n\in \mathbf{D}_{1,2}$, $n=1,2...$, be a sequence of $\mathcal{F}_1$-measurable random variables such that there exist constants $\alpha > 0$ and $C>0$ with
	$$
	\sup_n E \left[ | D_t X_n - D_{t'} X_n |^2 \right] \leq C |t -t'|^{\alpha}
	$$
	for $0 \leq t' \leq t \leq 1$, and
	$$
	\sup_n\sup_{0 \leq t \leq 1} E \left[ | D_t X_n |^2 \right] \leq C \,.
	$$
	Then the sequence $X_n$, $n=1,2,\ldots$, is relatively compact in $L^{2}(\Omega )$.
\end{cor}

\appendix
\setcounter{section}{1}
\numberwithin{equation}{section}
\section{Proof of Proposition \ref{mainEstimate}}

Here we give the proof of Proposition \ref{mainEstimate}. Before we proceed, we need some notations and intermediate results.

Without loss of generality, assume that $\| \tilde{b}_i \|_{\infty} \leq 1$ for $i = 1, 2 \dots, n$. Let $z = (z^{(1)}, \dots z^{(d)})$ be a generic element of $\mathbb{R}^d$ and $| \cdot |$ be the Euclidean norm. Let $P(t,z) = (2 \pi t)^{-d/2} e^{-|z|^2/2t}$ be the Gaussian kernel, then the left hand side of \eqref{estimate} can be written as
$$
\left|   \int_{t_0 < t_1 < \dots < t_n < t} \int_{\mathbb{R}^{dn}} \prod_{i=1}^n  D^{\alpha_i} b_i(t_i,z_i) P(t_i - t_{i-1}, z_i - z_{i-1} ) \diffns z_1 \dots \diffns z_n \diffns t_1 \dots \diffns t_n   \right| \,.
$$
Define
$$
J_n^{\alpha} (t_0,t, z_0) := \int_{ t_0 < t_1 < \dots < t_n < t} \int_{\mathbb{R}^{dn}} \prod_{i=1}^n  D^{\alpha_i} b_i(t_i,z_i) P(t_i - t_{i-1}, z_i - z_{i-1} ) \diffns z_1 \dots \diffns z_n \diffns t_1 \dots \diffns t_n,
$$
with $\alpha = (\alpha_1, \dots \alpha_n) \in \{0,1 \}^{nd}$. To prove the proposition, it is enough to prove that t $|J_n^{\alpha}(t_0,t,0)| \leq C^n(t-t_0)^{n/2} / \Gamma( n/2 + 1)$.

To this end, we shift the derivatives from the $b_i$'s onto $P$ by using the integration by part. This is done by introducing the alphabet $\mathcal{A}(\alpha) = \{ P, D^{\alpha_1}P, \dots , D^{\alpha_n}P, D^{\alpha_1}D^{\alpha_2}P, \dots D^{\alpha_{n-1}}D^{\alpha_n}P \}$. Here $D^{\alpha_i}$, $D^{\alpha_i}D^{\alpha_{i+1}}$ stands for the derivative of $P(t,z)$ with respect to the space.

Choose a string $S = S_1 \cdots S_n$ in $\mathcal{A}(\alpha)$ and define
$$
I_S^{\alpha} (t_0,t,z_0) := \int_{t_0 < \dots < t_n <t} \int_{\mathbb{R}^{dn}} \prod_{i=1}^n b_i(t_i, z_i) S_i(t_i - t_{i-1}, z_i - z_{i-1}) \diffns z_1 \dots \diffns z_n \diffns t_1 \dots \diffns t_n \,.
$$
In the following, we say that a string is \emph{allowed} if, when all the $D^{\alpha_i}P$'s are taken out from the string, a string of type $ P \cdot D^{\alpha_s}D^{\alpha_{s+1}} P  \cdot P \cdot D^{\alpha_{s+1}}D^{\alpha_{s+2}} P \cdots P \cdot D^{\alpha_r}D^{\alpha_{r+1}} P$ for $s \geq 1$, $r \leq n-1$ remains. Moreover, assume that the first derivatives $D^{\alpha_i}P$ are written in an increasing order with respect to $i$.


\begin{lemm}\label{lemm1mainpro}
	The following representation holds
	$$
	J_n^{\alpha}(t_0,t,z_0) = \sum_{j=1}^{2^{n-1}} \epsilon_j I_{S^j}^{\alpha}(t_0,t,z_0),
	$$
	where each $\epsilon_j$ is either $-1$ or $1$ and each $S^j$ is an allowed string in $\mathcal{A}(\alpha)$.
\end{lemm}

\begin{proof}
	We use induction on $n \geq 1$. Clearly, the representation holds for $n=1$. Now assume that  it holds for $n\ge 1$, and let $b_0$ be another function satisfying the assumptions of the proposition. as well as $\alpha_0$. Then
	\begin{align*}
	J_{n+1}^{(\alpha_0,\alpha)} (t_0,t,z_0)  = & \int_{t_0}^t \int_{\mathbb{R}^d} D^{\alpha_0}b_0(t_1,z_1) P(t_1 - t_0, z_1 - z_0) J_n^{\alpha}(t_1,t,z_1) \diffns z_1 \diffns t_1
	\\
	= & - \int_{t_0}^t \int_{\mathbb{R}^d} b_0(t_1,z_1) D^{\alpha_0}P(t_1 - t_0, z_1 - z_0) J_n^{\alpha}(t_1,t,z_1) \diffns z_1 \diffns t_1
	\\
	&- \int_{t_0}^t \int_{\mathbb{R}^d} b_0(t_1,z_1) P(t_1 - t_0, z_1 - z_0) D^{\alpha_0} J_n^{\alpha}(t_1,t,z_1) \diffns z_1 \diffns t_1 \,.
	\end{align*}
	Moreover
	$$
	D^{\alpha_0} I_S^{\alpha}(t_1,t,z_1) = -I_{\tilde{S}}^{(\alpha_0, \alpha)}(t_1,t,z_1),
	$$
	with
	$$
	\tilde{S} = \left\{
	\begin{array}{ll}
	D^{\alpha_0}P \cdot S_2 \cdots S_n & \textrm{ if } S = P \cdot S_2 \cdots S_n \\
	D^{\alpha_0} D^{\alpha_1} P \cdot S_2 \cdots S_n & \textrm{ if } S = D^{\alpha_1}P \cdot S_2 \cdots S_n  \,.\\
	\end{array}
	\right.
	$$
	Clearly, $\tilde{S}$ is not an allowed string in $\mathcal{A}(\alpha)$. Using the induction hypothesis $D^{\alpha_0} J_n^{\alpha}(t_0,t,z_0) = \sum_{j=1}^{2^{n-1}} - \epsilon_j I_{\tilde{S}}^{(\alpha_0, \alpha)} (t_0,t,z_0)$ and we have
	$$
	J_{n+1}^{(\alpha_0, \alpha)} = \sum_{j=1}^{2^{n-1}} - \epsilon_j I_{D^{\alpha_0}P \cdot S^j}^{(\alpha_0, \alpha)} + \sum_{j=1}^{2^{n-1}} \epsilon_j I_{P \cdot \tilde{S}^j}.
	$$
	One verifies that both $D^{\alpha_0}P \cdot S^j$ and $P \cdot \tilde{S}^j$ are allowed strings in $\mathcal{A}(\alpha_0, \alpha)$ whenever $S^j$ is an allowed string in $\mathcal{A}(\alpha)$.
	     \end{proof}
	
	Assuming that $S$ is a allowed string, we will give an upper bound of $I_S^{\alpha}$, thus the proof of Proposition \ref{mainEstimate} will be completed using the above representation.

	\begin{lemm}
		Let $\phi, h : [0,1] \times \mathbb{R}^d \rightarrow \mathbb{R}$ be measurable functions satisfying $ \frac{|\phi(s,z)|}{1+|z|} \leq e^{-|z|^2 /3s}$ and $\| \tilde{h} \|_{\infty} \leq 1$, with $\tilde{h}(s,y):=\frac{h(s,y)}{1+|y|}$. Let $\alpha,\beta \in \{0,1\}^d$ be multiindices satisfying $|\alpha| = |\beta| = 1$. Then one can find a universal constant $C$ \textup{(}independent of $\phi$, $h$, $\alpha$ and $\beta$\textup{)} satisfying
		$$
		\left| I \right| \leq C ,\,
		$$
		where \begin{equation}\label{eqlmb2I}
			I=\int_{1/2}^1 \int_{t/2}^t \int_{\mathbb{R}^d} \int_{\mathbb{R}^d} \phi(s,z) h(t,y) D^{\alpha} D^{\beta} P(t-s, y -z) \diffns y \diffns z \diffns s \diffns t.
			\end{equation}
	\end{lemm}
	
	\begin{proof}
		Let $l,m \in \mathbb{Z}^d$ and define $[l, l+1) := [l^{(1)}, l^{(1)} +1) \times \dots \times [l^{(d)}, l^{(d)} +1)$ and similarly for $[m,m+1)$. Moreover, define $\phi_l(s,z) := \phi(s,z) 1_{[l, l+1)}(z)$ and $h_m(t,y) := h(t,y) 1_{[m,m+1)}(y)$.
		
		We denote by $I_{l,m}$ the integral defined in \eqref{eqlmb2I} when $\phi$, $h$ are replaced by $\phi_l$, $h_m$.  
	Then $I = \sum_{l,m \in \mathbb{Z}^d} I_{l,m}$. In the following, $C$ denotes a constant that may change from one line to the other.
		
		We rewrite $I_{l,m}$ as:
		
		\begin{align}
		I_{l,m}:=&\int_{1/2}^1 \int_{t/2}^t \int_{\mathbb{R}^d} \int_{\mathbb{R}^d} \phi_l(s,z) h_m(t,y) D^{\alpha} D^{\beta} P(t-s, y -z) \diffns y \diffns z \diffns s \diffns t \notag\\
		\leq &\int_{1/2}^1 \int_{t/2}^t \int_{\mathbb{R}^d} \int_{\mathbb{R}^d} \frac{|\phi_l(s,z)|}{1+|z|}(1+|z|) \frac{|h_m(t,y)|}{1+|y|}(1+|y-z| +|z|)D^{\alpha} D^{\beta} P(t-s, y -z) \diffns y \diffns z \diffns s \diffns t\notag\\
		\leq & I^1_{l,m} + I^2_{l,m},
		\end{align}
		where
		\begin{align*}
		I^1_{l,m}:=&\int_{1/2}^1 \int_{t/2}^t \int_{\mathbb{R}^d} \int_{\mathbb{R}^d} \frac{|\phi_l(s,z)|}{1+|z|}(1+|z|)\frac{|h_m(t,y)|}{1+|y|}(1+|y-z|)D^{\alpha} D^{\beta} P(t-s, y -z) \diffns y \diffns z \diffns s \diffns t,\\
		I^2_{l,m}:=&\int_{1/2}^1 \int_{t/2}^t \int_{\mathbb{R}^d} \int_{\mathbb{R}^d} \frac{|\phi_l(s,z)|}{1+|z|}(1+|z|)^2\frac{|h_m(t,y)|}{1+|y|}D^{\alpha} D^{\beta} P(t-s, y -z) \diffns y \diffns z \diffns s \diffns t.
		\end{align*}
		We only give an estimate of $I^1_{l,m}$. The corresponding estimate  for $I^2_{l,m}$ follows similarly.
		
		suppose that $\| l - m \|_{\infty} := \max_{i}|l^{(i)} - m^{(i)}| \geq 2$. For $z \in [l, l+1)$ and $y \in [m, m+1)$ we have $| z - y | \geq \|l - m\|_{\infty} - 1$.
		
		Assume $\alpha \neq \beta$, then
		\begin{align*}
		|(1+|y-z|)D^{\alpha} D^{\beta} P(t-s,  z-y) |=& |\frac{(z^{(i)} - y^{(i)})(z^{(j)} - y^{(j)})}{(t-s)^2}|(1+|y-z|)P(t-s,y-z)\\
		\leq & C |\frac{(z^{(i)} - y^{(i)})(z^{(j)} - y^{(j)})}{(t-s)^2}|\tilde{P}(t-s,y-z)
		\end{align*}
		for appropriate choice of $i,j$, where $\tilde{P}(t,z) = (2 \pi t)^{-d/2} e^{-|z|^2/4t}$. Then there exists a positive constant $C$ such that
		$$
		|(1+|y-z|)D^{\alpha} D^{\beta} P(t-s, z-y)| \leq C e^{-(\|l - m\|_{\infty}- 2)^2/8}.
		$$
		Assume $\alpha = \beta$, then
		$$
		(1+|y-z|)(D^{\alpha})^2 P(t-s, y-z) = \left( \frac{(y^{(i)} - z^{(i)})^2}{t-s} - 1 \right)(1+|y-z|)\frac{P(t-s,  y-z)}{t-s}
		$$
		and similarly there exists a positive constant $C$ such that
		$$
		|(1+|y-z|)(D^{\alpha})^2 P(t-s, y-z)| \leq C e^{-(\|l -m \|_{\infty} - 2)^2/8} \,.
		$$
		Using once more the fact that $(1+|z|)e^{-|z|^2 /3s}$ is bounded by $Ce^{-|z|^2 /6s}$ for $s\in[0,1]$, we get in both cases $| I^1_{l,m} | \leq C e^{-\|l\|^2/16}e^{-(\| l-m\|_{\infty} - 2)^2/8}$; thus 
		
		$$
		\sum_{\|l-m\|_{\infty} \geq 2} |I^1_{l,m}| \leq C .
		$$
		
		Suppose that $\| l -m \|_{\infty} \leq 1$ and denote by $\hat{\phi}_l(s,u)$ and $\hat{h}_m(t,u)$ the Fourier transform of $\phi$  and $h$ in the second variable. Then
		$$
		\hat{h}_{m}(t,u) := (2 \pi)^{-d/2} \int_{\mathbb{R}^d} h_m(t,x) e^{-i (u,x)} \diffns x,
		$$
		and similarly for $\hat{\phi}_l(s,u)$. It follows from the Plancherel theorem that
		\begin{align*}
		\int_{\mathbb{R}^d} \hat{\phi}_l(s,u)^2 \diffns u =& \int_{\mathbb{R}^d} \phi_l(s,z)^2 \diffns z \\
		= & \int_{\mathbb{R}^d} \frac{\phi_l(s,z)^2}{(1+|z|)^2} (1+|z|)^2\diffns z \\
		\leq & \int_{\mathbb{R}^d} e^{-2|z|^2 /3s} (1+|z|)^21_{[l, l+1)}(z)\diffns z\\
		\leq & C\int_{\mathbb{R}^d} e^{-|z|^2 /6s} 1_{[l, l+1)}(z)\diffns z\\
		\leq& C e^{-\|l\|^2 / 24}
		\end{align*}
		for all $s \in [0,1]$ and
		$$
		\int_{\mathbb{R}^d} \hat{h}_m(t,u)^2 \diffns u = \int_{\mathbb{R}^d} h_m(t,y)^2 \diffns y.
		$$
		We have
		\begin{equation} \label{2to1integral}
		I_{l,m} = \int_{1/2}^1 \int_{t/2}^t \int_{\mathbb{R}^d} \hat{\phi}_l(s,u) \hat{h}_m(t,-u) u^{(i)}u^{(j)}(t-s) e^{-(t-s)|u|^2/2}\diffns u \diffns s \diffns t .
		\end{equation}
		
		This can be seeing by applying the Fubini's theorem to the righ hand side of \eqref{2to1integral} to get 

		\begin{align*}
		& \int_{\mathbb{R}^d} \hat{h}_m(t,-u) \hat{\phi}_l(s,u) u^i u^j(t-s) e^{-(t-s) |u |^2 /2} \diffns u  \\
		& = (2 \pi)^{-d} \int_{\mathbb{R}^d} \int_{\mathbb{R}^d} \int_{\mathbb{R}^d} h_m(t,x) e^{i(u,x)} \phi_l(s,y) e^{-i(u,y)} u^i u^j(t-s) e^{-(t-s) |u|^2/2} \diffns u \diffns x \diffns y   \\
		& =  \int_{\mathbb{R}^d} \int_{\mathbb{R}^d} h_m(t,x) \phi_l(s,y)(t-s) \left[(2 \pi)^{-d} \int_{\mathbb{R}^d} e^{i(u,x-y)} u^iu^j e^{-(t-s) |u|^2/2} du \right] \diffns x \diffns y.
		\end{align*}
		
		Consider the expression in the square brackets. Substituting $v = \sqrt{t-s} u$, we get
		\begin{align*}
		& (2 \pi)^{-d} \int_{\mathbb{R}^d} e^{i(u,x-y)} u^iu^j e^{-(t-s) |u|^2/2} \diffns u \\
		&= (2 \pi)^{-d} (t-s)^{-d/2} \int_{\mathbb{R}^d} e^{i(\frac{v}{\sqrt{t-s}}, x-y)} \frac{v^i}{\sqrt{t-s}} \frac{v^j}{\sqrt{t-s}} e^{-|v|^2 /2} \diffns v \\
		&= (2 \pi)^{-d} (t-s)^{-d/2} (t-s)^{-1} \int_{\mathbb{R}^d} e^{i(v, \frac{x-y}{\sqrt{t-s}})} v^i v^j e^{-|v|^2 /2} \diffns v.
		\end{align*}
		Set $f(v) = e^{-|v|^2 /2}$ and $p(v) = v^{(i)}v^{(j)}$. The properties of the Fourier transform yield $\widehat{pf} = D^{\alpha} D^{\beta} \hat{f} $ and $\hat{f} = f$. Hence, the above expression can be written by
		$$
		(2 \pi)^{-d/2}(t-s)^{-d/2}(t-s)^{-1}D^{\alpha} D^{\beta} f \left( \frac{x-y}{\sqrt{t-s}} \right) = (t-s)^{-1} D^{\alpha} D^{\beta} P(t-s, x-y),
		$$
		from which we obtain  \eqref{2to1integral}.
		
		
		Using the inequality $ab \leq \frac{1}{2} a^2c + \frac{1}{2} b^2c^{-1}$ with $a = \hat{\phi}_l(s,u)u^{(i)}$, $b = \hat{h}_m(t,-u)u^{(j)}$ and $c = e^{\|l\|^2/48}$, we have
		\begin{align*}
		| I_{l,m} | & \leq  \frac{1}{2} \int_{1/2}^1 \int_{t/2}^t \int_{\mathbb{R}^d} \hat{\phi}_l(s,u)^2(u^{(i)})^2 e^{\|l\|^2/48} e^{-(t-s)|u|^2/2} \diffns u \diffns s \diffns t
		\\
		& + \frac{1}{2} \int_{1/2}^1 \int_{t/2}^t \int_{\mathbb{R}^d} \hat{h}_m(t,-u)^2(u^{(j)})^2 (t-s)^2e^{-\|l\|^2/48} e^{-(t-s)|u|^2/2} \diffns u \diffns s \diffns t
		\\
		& \leq \frac{1}{2} \int_{1/2}^1 \int_{t/2}^t \int_{\mathbb{R}^d} \hat{\phi}_l(s,u)^2|u|^2 e^{\|l\|^2/48} e^{-(t-s)|u|^2/2} \diffns u \diffns s \diffns t
		\\
		& + \frac{1}{2} \int_{1/2}^1 \int_{t/2}^t \int_{\mathbb{R}^d} \hat{h}_m(t,-u)^2|u|^2 (t-s)^2e^{-\|l\|^2/48} e^{-(t-s)|u|^2/2} \diffns u \diffns s \diffns t .
		\end{align*}
		Integrating the first term of the right hand side in the above inequality with respect to $t$ gives 
		$$
		\int_{1/2}^1 \int_{t/2}^t \int_{\mathbb{R}^d} \hat{\phi}_l(s,u)^2|u|^2 e^{\|l\|^2/48} e^{-(t-s)|u|^2/2} \diffns u \diffns s \diffns t\leq C e^{-\|l\|^2/48}.
		$$
		Now let assume $m= 0$ that is $m^{(i)}=0$ for $i=1,\ldots,d$ then we have $h_m=h_0$ and 
		$$
		\int_{\mathbb{R}^d} \hat{h}_0(t,u)^2 \diffns u = \int_{\mathbb{R}^d} h_0(t,y)^2 \diffns y\leq C(1+d).
		$$
		Hence, integrating the second term with respect to $s$ gives
		\begin{align*}
		\int_{1/2}^1 \int_{t/2}^t \int_{\mathbb{R}^d} \hat{h}_0(t,-u)^2|u|^2(t-s)^2 e^{-\|l\|^2/48} e^{-(t-s)|u|^2/2} \diffns u \diffns s \diffns t \leq C(1+d) e^{-\|l\|^2/48}
		\end{align*}
		Let now assume that $\|m\|\neq 0$, then $\|m\|\geq 1$ and we have 
		\begin{align*}
		&\int_{1/2}^1 \int_{t/2}^t \int_{\mathbb{R}^d} \hat{h}_m(t,-u)^2|u|^2(t-s)^2 e^{-\|l\|^2/48} e^{-(t-s)|u|^2/2} \diffns u \diffns s \diffns t\\
		=&\int_{1/2}^1 \int_{t/2}^t \int_{\mathbb{R}^d} \hat{h}_m(t,-u)^2(t-s) \Big(|u|\sqrt{(t-s)}\Big)^2e^{-\|l\|^2/48} e^{-(t-s)|u|^2/2} \diffns u \diffns s \diffns t\\
		\leq & C e^{-\|l\|^2/48} \int_{1/2}^1\int_{t/2}^t \int_{\mathbb{R}^d} \hat{h}_m(t,-u)^2(t-s) e^{-(t-s)|u|^2/8}\diffns u \diffns s \diffns t\\
		\leq & C e^{-\|l\|^2/48} \int_{1/2}^1\int_{t/2}^t \int_{\mathbb{R}^d} h_m(t,-u)\ast h_m(t,-u) \frac{2^d}{(t-s)^{d/2}}e^{-2|u|^2/ (t-s)} \diffns u \diffns s \diffns t \\
		=& C e^{-\|l\|^2/48} \int_{1/2}^1\int_{t/2}^t \int_{\mathbb{R}^d} \int_{\mathbb{R}^d} h_m(t,v) h_m(t,v-u) \frac{2^d}{(t-s)^{d/2}}e^{-2|u|^2/ (t-s)}  \diffns v \diffns u \diffns s \diffns t\\
		=& C e^{-\|l\|^2/48} \int_{1/2}^1\int_{t/2}^t \int_{[m, m+1)^d} \int_{[m, m+1)^d} (1+|v|)(1+|u-v|) \frac{2^d}{(t-s)^{d/2}}e^{-2|u|^2/ (t-s)}  \diffns v\diffns u \diffns s \diffns t\\
		\leq &  C e^{-\|l\|^2/48}  (1+\|m+1\|^2)2^d\int_{1/2}^1\int_{t/2}^t \frac{1}{(t-s)^{d/2}}e^{-2\|m\|^2/ (t-s)}    \diffns s \diffns t.
		\end{align*}
		We show that when we integrate first with respect to $s$, the integral is bounded above. We split the proof in several cases.
		\begin{itemize}
			\item For $d=1$
			\begin{align*}
			\int_{t/2}^t \frac{1}{(t-s)^{1/2}}e^{-2\|m\|^2/ (t-s)}   \diffns s =&\Big[-2\sqrt{t-s}e^{-2\|m\|^2/ (t-s)}\Big]_{t/2}^t-	\int_{t/2}^t \frac{4\|m\|^2}{(t-s)^{3/2}}e^{-2\|m\|^2/ (t-s)}   \diffns s\\
			\leq& \Big[-2\sqrt{t-s}e^{-2\|m\|^2/ (t-s)}\Big]_{t/2}^t\leq 2e^{-4\|m\|^2}.
			\end{align*}
			\item
			For $d=2$, we have
			\begin{align*}
			\int_{t/2}^t \frac{1}{(t-s)}e^{-2\|m\|^2/ (t-s)}   \diffns s =&\Big[\frac{-(t-s)}{2\|m\|^2}e^{-2\|m\|^2/ (t-s)}\Big]_{t/2}^t-\frac{1}{2\|m\|^2}	\int_{t/2}^t e^{-2\|m\|^2/ (t-s)}   \diffns s\\
			\leq&\Big[\frac{-(t-s)}{2\|m\|^2}e^{-2\|m\|^2/ (t-s)}\Big]_{t/2}^t\leq \frac{1}{4\|m\|^2}e^{-4\|m\|^2} .
			\end{align*}
			
			\item For $d=3$, we have
			\begin{align*}
			\int_{t/2}^t \frac{1}{(t-s)^{3/2}}e^{-2\|m\|^2/ (t-s)}   \diffns s =&	 \int_{t/2}^t \frac{(t-s)^{1/2}}{(t-s)^{2}}e^{-2\|m\|^2/ (t-s)}   \diffns s\\
			\leq& 	\int_{t/2}^t \frac{1}{(t-s)^{2}}e^{-2\|m\|^2/ (t-s)}   \diffns s\\
			=&\Big[\frac{-1}{2\|m\|^2}e^{-2\|m\|^2/ (t-s)}\Big]_{t/2}^t\\
			\leq& \frac{1}{4\|m\|^2}e^{-4\|m\|^2}.
			\end{align*}
			
			\item For $d=4$, we have
			\begin{align*}
			\int_{t/2}^t \frac{1}{(t-s)^{2}}e^{-2\|m\|^2/ (t-s)}   \diffns s =&\Big[\frac{-1}{2\|m\|^2}e^{-2\|m\|^2/ (t-s)}\Big]_{t/2}^t\\
			\leq& \frac{1}{4\|m\|^2}e^{-4\|m\|^2} .
			\end{align*}
			
			\item For $d>4$, we use induction. Let $J_d=	\int_{t/2}^t \frac{1}{(t-s)^{d/2}}e^{-2\|m\|^2/ (t-s)}   \diffns s$. Assume that $J_{d-1}\leq \frac{C}{4\|m\|^2}e^{-4\|m\|^2}$. We will  show that $J_{d}\leq \frac{C}{4\|m\|^2}e^{-4\|m\|^2}$.
			\begin{align*}
			J_d= &	\int_{t/2}^t \frac{1}{(t-s)^{d/2-2+2}}e^{-2\|m\|^2/ (t-s)}   \diffns s\\
			=&\Big[\frac{-(t-s)^{-d/2+2}}{2\|m\|^2}e^{-2\|m\|^2/ (t-s)}\Big]_{t/2}^t+\frac{(d/2-2)}{2\|m\|^2}	\int_{t/2}^t \frac{1}{(t-s)^{d/2-1}}e^{-2\|m\|^2/ (t-s)}   \diffns s\\
			\leq&\Big[\frac{-(t-s)^{-d/2+2}}{2\|m\|^2}e^{-2\|m\|^2/ (t-s)}\Big]_{t/2}^t+(d/2-2)	 \int_{t/2}^t \frac{1}{(t-s)^{(d-1)/2}}e^{-2\|m\|^2/ (t-s)}   \diffns s\\
			\leq &\frac{C}{4\|m\|^2}e^{-4\|m\|^2},
			\end{align*}
			which gives 
			\begin{align*}
			 | I_{l,m}| \leq& C (1+\|m+1\|^2)e^{-\|l\|^2/48} e^{-4\|m\|^2} \\
			 \leq& C (1+\|m\|^2 +d^2)e^{-\|l\|^2/48} e^{-4\|m\|^2} .
			 \end{align*}
			  In any cases, we have $ | I_{l,m}| \leq C (1+\|m\|^2 +d^2)e^{-\|l\|^2/48} e^{-4\|m\|^2}$. Thus
			
			$$
			\sum_{\|l-m\|_{\infty} \leq 1} |I_{l,m}| \leq C .
			$$
		\end{itemize}
	\end{proof}

	\begin{cor} \label{frstCor} Let $g,h: [0,1] \times \mathbb{R}^d \to \mathbb{R}$ be measurable functions satisfying $\|\tilde g\|_{\infty} \leq 1, \|\tilde h\|_{\infty} \leq 1$  where $\tilde{g}(s,y):=\frac{g(s,y)}{1+|y|}$ and $\tilde{h}(s,y):=\frac{h(s,y)}{1+|y|}$ for $(s,y) \in [0,1] \times \mathbb{R}^d$. Then
		there is a positive constant $C$ such that
		$$
		\left| \int_{1/2}^1 \int_{t/2}^t \int_{\mathbb{R}^d} \int_{\mathbb{R}^d} g(s,z)P(s,z) h(t,y) D^{\alpha}D^{\beta}P(t-s,y-z) \diffns y\diffns z\diffns s\diffns t \right| \leq C
		$$
		and
		$$
		\left| \int_{1/2}^1 \int_{t/2}^t \int_{\mathbb{R}^d} \int_{\mathbb{R}^d} g(s,z)D^{\gamma}P(s,z) h(t,y) D^{\alpha}D^{\beta}P(t-s,y-z) \diffns y\diffns z\diffns s\diffns t \right| \leq C \,.
		$$
		Furthermore,  $\int_{\mathbb{R}^d} P(t,z)\diffns z = 1$ and
		\begin{equation} \label{frstDerEst}
		\int_{\mathbb{R}^d} |D^{\alpha} P(t,z)| \diffns z \leq C t^{-1/2}  \,,
		\end{equation}
		\begin{equation} \label{scndDerEst}
		\int_{\mathbb{R}^d} |D^{\alpha} D^{\beta} P(t,z)| \diffns z \leq C t^{-1} \,.
		\end{equation}
	\end{cor}

	\begin{lemm} \label{cnvLmm}
		We can find an absolute constant $C$ such that for every Borel-measurable functions $g$ and $h$ such that $\tilde{h}$ and $\tilde{g}$ are bounded by 1, and $r \geq 0$
		\begin{align*}
		& \Big| \int_{t_0}^t \int_{t_0}^{t_1} \int_{\mathbb{R}^d} \int_{\mathbb{R}^d} g(t_2,z)P(t_2-t_0,z) h(t_1,y)D^{\alpha}D^{\beta}P(t_1-t_2,y-z)(t-t_1)^r \diffns y\diffns z\diffns t_1\diffns t_2 \Big|  \\
		\leq & C(1+r)^{-1} (t-t_0)^{r+1}
		\end{align*}
		and
		\begin{align*}
		&	\Big| \int_{t_0}^t \int_{t_0}^{t_1} \int_{\mathbb{R}^d} \int_{\mathbb{R}^d} g(t_2,z)D^{\gamma}E(t_2-t_0,z) h(t_1,y)D^{\alpha}D^{\beta}P(t_1-t_2,y-z)(t-t_1)^r \diffns y\diffns z\diffns t_1\diffns t_2 \Big| \\
		\leq & C(1+r)^{-1/2} (t-t_0)^{r+1/2} \,.
		\end{align*}
	\end{lemm}
	
	\begin{proof}
		We first prove the estimate for $t = 1, t_0 = 0$. Using Corollary \ref{frstCor} it follows that for each $k\ge 0$
		\begin{align*}
		\Big| \int_{2^{-k-1}}^{2^{-k}} \int_{t/2}^t \int_{\mathbb{R}^d} \int_{\mathbb{R}^d} g(s,z)P(s,z) h(t,y)D^{\alpha}D^{\beta}P(t-s,y-z)(1-t)^r \diffns y\diffns z\diffns s\diffns t   \Big|
		\leq  C(1-2^{-k-1})^r 2^{-k} \,.
		\end{align*}
		In fact, put $t' = 2^k t$ and $s' = 2^k s$ and use the fact that $P(at,z) = a^{-d/2}P(t,a^{-1/2}z)$, then substitute $z' = 2^{k/2}z$ and $y' = 2^{k/2}y$. The result follows by sing $ h_1(t,y) := \frac{(1-t)^r}{(1-2^{-k-1})^r} h(t,y) $ in Corollary \ref{frstCor}.
		
		Summing the above inequalities over $k$ gives
		\begin{multline*}
		\left| \int_{0}^1 \int_{t/2}^t \int_{\mathbb{R}^d} \int_{\mathbb{R}^d} g(s,z)P(s,z) h(t,y)D^{\alpha}D^{\beta}P(t-s,y-z)(1-t)^r \diffns y\diffns z\diffns s\diffns t  \right|
		\leq C (1+r)^{-1}.
		\end{multline*}
		Using the bound \eqref{scndDerEst}, we get
		\begin{align}
		&\int_0^1 \int_0^{t/2} \int_{\mathbb{R}^d} \int_{\mathbb{R}^d} g(s,z) P(s,z) h(t,y)D^{\alpha}D^{\beta} P(t-s,y-z) (1-t)^r \diffns y\diffns z\diffns s\diffns t  \nonumber\\
		\leq&\int_0^1 \int_0^{t/2} \int_{\mathbb{R}^d} \int_{\mathbb{R}^d} \frac{|g(s,z)|}{1+|z|} (1+|z|)P(s,z) \frac{|h(t,y)|}{1+|y|}(1+|y|)D^{\alpha}D^{\beta} P(t-s,y-z) (1-t)^r \diffns y\diffns z\diffns s\diffns t  \nonumber\\
		\leq& 2\Big( \int_0^1 \int_0^{t/2} \int_{\mathbb{R}^d} \int_{\mathbb{R}^d}  (1+|z|^2)P(s,z) D^{\alpha}D^{\beta} P(t-s,y-z) (1-t)^r \diffns y\diffns z\diffns s\diffns t   \nonumber \\
		&+\int_0^1 \int_0^{t/2} \int_{\mathbb{R}^d} \int_{\mathbb{R}^d}  (1+|z|)P(s,z) (1+|y-z|)D^{\alpha}D^{\beta} P(t-s,y-z) (1-t)^r \diffns y\diffns z\diffns s\diffns t\Big)  \nonumber\\
		\leq & C \int_0^1 \int_0^{t/2} (t-s)^{-1}(1-t)^r \diffns s \diffns t \leq C(1+r)^{-1}, \label{scndDerEstnew}
		\end{align}
		where in the third inequality, we used the following estimates:
		\begin{itemize}
			\item $(1+|z|)P(s,z)\leq C \tilde{P}(s,z)=C(2 \pi s)^{-d/2} e^{-|z|^2/4s}$,
			\item $(1+|z|^2)P(s,z)\leq C P_1(s,z)= C (2 \pi s)^{-d/2} e^{-|z|^2/8s}$,
			\item $(1+|y-z|)D^{\alpha}D^{\beta} P(t-s,y-z)\leq CD^{\alpha}D^{\beta} \tilde{P}(t-s,y-z)$ and one can also show that $\int_{\mathbb{R}^d} |D^{\alpha} D^{\beta} \tilde{P}(t,z)| \diffns z \leq C t^{-1} $.
		\end{itemize}
		
		Combination of these bounds gives the first assertion for $t=1, t_0 = 0$. For general $t$ and $t_0$ use the change of variables $t_1' = \frac{t_1 - t_0}{t-t_0}$, $t_2 = \frac{t_2 - t_0}{t - t_0}$, $y' = (t-t_0)^{-1/2}y$ and $z' = (t-t_0)^{-1/2}z$.
		
		The second assertion of the lemma follows similarly.
	\end{proof}
	
	We are now ready to give a proof of Proposition \ref{mainEstimate}.

		\begin{proof}[Proof of Proposition \protect\ref{mainEstimate}]:
	We prove that there is a constant $M$ such that for each allowed string $S$ in the alphabet $\mathcal{A}(\alpha)$ we have
	$$
	I_S^{\alpha}(t_0,t,z_0) \leq \frac{M^n (t-t_0)^{n/2}}{\Gamma (\frac{n}{2} + 1)}.
	$$
	We use induction on $n$. The case $n=0$ is straightforward. Assume $n > 0$ and that the above inequality is valid for all allowed strings of length less than $n$.
	There are three possibilities.
	\begin{enumerate}
		\item
		$S = D^{\alpha_1}P \cdot S'$ where $S'$ is a string in $\mathcal{A}(\alpha')$ and $\alpha' := (\alpha_2, \dots, \alpha_n)$
		
		\item
		$S = P \cdot D^{\alpha_1} D^{\alpha_2} P \cdot S'$ where $S'$ is a string in $\mathcal{A}(\alpha')$ and $\alpha' := (\alpha_3, \dots, \alpha_n)$
		
		\item
		$S = P \cdot D^{\alpha_1}P \cdots D^{\alpha_m} P \cdot D^{\alpha_{m+1}} D^{\alpha_{m+2}}P \cdot  S'$ where $S'$ is a string in $\mathcal{A}(\alpha')$ and $\alpha' := (\alpha_{m+3}, \dots, \alpha_n)$.
	\end{enumerate}
	
	In each possibility, $S'$ is an allowed string in the given alphabet.
	
	\begin{enumerate}
		
		\item
		Use the inductive hypothesis to get an upper bound for $I_{S'}^{\alpha'}(t_1,t,z_1)$ and using \eqref{frstDerEst}, we have
		\begin{align*}
		| I_S^{\alpha}(t_0,t,z_0)| =&  \left| \int_{t_0}^t \int_{\mathbb{R}^d} b_1(t_1,z_1) D^{\alpha_1}P(t_1 - t_0, z_1 - z_0) I_{S'}^{\alpha'}(t_1,t,z_1) \diffns z_1 \diffns t_1 \right|
		\\
		\leq &\int_{t_0}^t \int_{\mathbb{R}^d} \frac{|b_1(t_1,z_1)|}{1+|z_1|}(1+|z_1|) D^{\alpha_1}P(t_1 - t_0, z_1 - z_0) I_{S'}^{\alpha'}(t_1,t,z_1) \diffns z_1 \diffns t_1  \\
		\leq&  \frac{M^{n-1}}{\Gamma( \frac{n+1}{2})} \int_{t_0}^t (t-t_1)^{(n-1)/2}\int_{\mathbb{R}^d} (1+|z_1|)|D^{\alpha_1}P(t_1 - t_0, z_1 - z_0) | \diffns z_1 \diffns t_1
		\\
		\leq&  \frac{M^{n-1}}{\Gamma( \frac{n+1}{2})} \int_{t_0}^t (t-t_1)^{(n-1)/2}\Big(\int_{\mathbb{R}^d} (1+|z_0|)|D^{\alpha_1}P(t_1 - t_0, z_1 - z_0) |\diffns z_1 \diffns t_1
		\\
		&+\int_{\mathbb{R}^d} (1+|z_1-z_0|)|D^{\alpha_1}P(t_1 - t_0, z_1 - z_0) | \diffns z_1 \Big)\diffns t_1\\
		\leq& \frac{M^{n-1}C}{\Gamma( \frac{n+1}{2})} (1+|z_0|)\int_{t_0}^t (t-t_1)^{(n-1)/2}(t_1 - t_0)^{-1/2} \diffns t_1
		\\
		=&  \frac{M^{n-1}C \sqrt{\pi} (t-t_0)^{n/2}}{\Gamma( \frac{n}{2} + 1)} (1+|z_0|),
		\end{align*}
		where the last inequality is similar to \eqref{frstDerEst} and follows as in \eqref{scndDerEstnew}. The result follows if $M$ is large enough.
		
		\item
		Here, write
		\begin{multline*}
		\ \ \hspace{0.6cm} I_S^{\alpha} (t_0,t, z_0) = \int_{t_0}^t \int_{t_1}^t \int_{\mathbb{R}^d} \int_{\mathbb{R}^d} b_1(t_1,z_1) b_2(t_2,z_2) \\
		\times P(t_1-t_0, z_1 - z_0) D^{\alpha_1}D^{\alpha_2} P(t_2 - t_1, z_2 - z_1) I_{S'}^{\alpha'}(t_2,t,z_2) \diffns z_1\diffns z_2 \diffns t_2\diffns t_1 .
		\end{multline*}
		Define $h(t_2,z_2) := b_2(t_2,z_2) I_{S'}^{\alpha'}(t_2,t,z_2) (t-t_2)^{1-n/2}$. Hence it follows from the inductive hypothesis that
		$$
		\frac{| h(t_2,z_2) |}{1+|z_2|}\leq M^{n-2}/\Gamma(n/2).
		$$
		Use the above bound in the first part of Lemma \ref{cnvLmm} with $g = b_1$ and integrating first with respect to $t_2$, we get
		$$
		|I_S^{\alpha}(t_0,t,z_0)| \leq \frac{CM^{n-2}(t-t_0)^{n/2}}{n\Gamma(n/2)}
		$$
		and the result follows if $M$ is large enough.

		\item
		We have
		\begin{align*}
		\ \ \hspace{0.6cm} I_S^{\alpha}(t_0,t,z_0) & = \int_{t_0 < \dots t_{m+2} < t} \int_{\mathbb{R}^{(m+2)d}} P(t_1 - t_0, z_1 - z_0) \prod_{j=1}^{m+2} b_j(t_j,z_j)
		\\
		& \times \prod_{j=2}^{m+1} D^{\alpha_j}P(t_j - t_{j-1},z_j - z_{j-1}) D^{\alpha_{m+1}}D^{\alpha_{m+2}}P(t_{m+2} - t_{m+1},z_{m+2} - z_{m+1})
		\\
		& \times I_{S'}^{\alpha'}(t_{m+2},t,z_{m+2}) \diffns z_1 \dots \diffns z_{m+2} \diffns t_1 \dots \diffns t_{m+2} \,.
		\end{align*}
		Define $h(t_{m+2},z_{m+2}) = b_{m+2}(t_{m+2},z_{m+2}) I_{S'}^{\alpha'}(t_{m+2},t,z)(t-t_{m+2})^{(2+m-n)/2}$. It follows from the induction hypothesis that
		$$
		\frac{| h(t_{m+2},z_{m+2}) |}{1+|z_{m+2}|}\leq M^{n-m-2}/\Gamma((n-m)/2).
		$$
		
		Let
		\begin{align*}
		\Omega(t_m,z_m) & := \int_{t_m}^t \int_{t_{m+1}}^t \int_{\mathbb{R}^{2d}} b_{m+1}(t_{m+1},z_{m+1}) h(t_{m+2},z_{m+2})
		\\
		&\times (t-t_{m+2})^{(n-m-2)/2} D^{\alpha_m}P(t_{m+1} - t_m, z_{m+1} - z)
		\\
		& \times D^{\alpha_{m+1}}D^{\alpha_{m+2}}P(t_{m+2} - t_{m+1}, z_{m+2} - z_{m+1}) \diffns z_{m+1} \diffns z_{m+2} \diffns t_{m+1}\diffns t_{m+2} \,.
		\end{align*}
		Using Lemma \ref{cnvLmm},  we have
		$$
		|\Omega(t_m,z_m)| \leq \frac{C(n-m)^{-1/2} M^{n-m-2}(t-t_m)^{(n-m-1)/2}}{\Gamma( \frac{n-m}{2} )}  \,.
		$$
		Using the above bound in
		\begin{multline*}
		\ \ \hspace{0.6cm} I_S^{\alpha}(t_0,t,z_0) = \int_{t_0 < \dots t_{m+2} < t} \int_{\mathbb{R}^{(m+2)d}} P(t_1 - t_0, z_1 - z_0) \prod_{j=1}^m b_j(t_j,z_j)  \\
		\times \prod_{j=1}^{m-1}D^{\alpha_j}P(t_j - t_{j-1}, z_j- z_{j-1}) \Omega(t_m,z_m) \diffns z_1 \dots \diffns z_m \diffns t_1 \dots \diffns t_m \,,
		\end{multline*}
		and using \eqref{frstDerEst} repeatedly, we get
		\begin{align*}
		\ \ \hspace{0.6cm} |I_S^{\alpha}(t_0,t,z_0)| & \leq C^{m+1} (n-m)^{-1/2} \frac{M^{n-m-2}}{\Gamma((n-m)/2)}
		\\
		& \times \int_{t_0 < \dots t_m < t} (t_2 - t_1)^{-1/2} \dots (t_m - t_{m-1})^{-1/2}(t-t_m)^{(n-m-1)/2} \diffns t_1 \dots \diffns t_m
		\\
		& = C^{m+1}(n-m)^{-1/2} \frac{M^{n-m-2}\pi^{(m-1)/2} \Gamma( \frac{n-m+1}{2}) }{\Gamma( \frac{n-m}{2}) \Gamma( \frac{n}{2} + 1) } (t-t_0)^{n/2}
		\end{align*}
		and the result follows when $M$ is large enough, thus completing the induction argument. The proof of Proposition \ref{mainEstimate} is now completed.
	\end{enumerate}
\end{proof}

				\end{document}